\documentclass[11pt]{amsart}

\usepackage{amsmath,amssymb,amsxtra,tikz,tikz-cd,combelow,mathtools,bbm}
\usepackage[normalem]{ulem}

\newtheorem{lemma}{Lemma}[section]
\newtheorem{theorem}[lemma]{Theorem}
\newtheorem{conjecture}[lemma]{Conjecture}
\newtheorem{corollary}[lemma]{Corollary}
\newtheorem{proposition}[lemma]{Proposition}
\theoremstyle{definition}
\newtheorem{example}[lemma]{Example}
\newtheorem{remark}[lemma]{Remark}
\newtheorem{definition}[lemma]{Definition}

\newcommand{\CA}{\mathcal{A}}
\newcommand{\CB}{\mathcal{B}}
\newcommand{\CCA}{\mathcal{CA}}
\newcommand{\CK}{\mathcal{K}}

\newcommand{\C}{\mathbb{C}}
\newcommand{\Z}{\mathbb{Z}}

\newcommand{\xx}{\underline{x}}
\newcommand{\inv}{^{-1}}
\newcommand{\e}{\varepsilon}

\newcommand{\Br}{\mathrm{Br}}
\newcommand{\fg}{\mathfrak{g}}
\newcommand{\Tr}{\mathrm{Tr}}
\newcommand{\Ker}{\mathrm{Ker}}
\newcommand{\Imm}{\mathrm{Im}}
\newcommand{\diag}{\mathrm{diag}}
\newcommand{\SBim}{\mathrm{SBim}}
\newcommand{\Hom}{\mathrm{Hom}}
\newcommand{\Ext}{\mathrm{Ext}}

\newcommand{\strict}{\mathrm{str}}

\newcommand{\Vol}{\mathrm{Vol}}

\newcommand{\eNG}{\overline{\mathrm{NG}}}
\newcommand{\bNG}{\mathrm{NG}}

\newcommand{\Id}{\mathrm{Id}}
\newcommand{\id}{\mathrm{Id}}

\newcommand{\op}{\mathrm{op}}
\newcommand{\HH}{\mathrm{HH}}
\newcommand{\HHH}{\mathrm{HHH}}
\newcommand{\CY}{\mathcal{Y}}

\newcommand{\HY}{\mathrm{HY}}
\newcommand{\oHHH}{\overline{\HHH}}
\newcommand{\one}{\mathbbm{1}}

\newcommand{\FT}{\mathrm{FT}}

\newcommand{\Mod}{\mathrm{mod}}

\newcommand{\End}{\mathrm{End}}

\newcommand{\quasiis}{\cong_{\text{q.is.}}}

\newcommand{\dbyd}[1]{\frac{\partial}{\partial {#1}}}
\newcommand{\wt}{\widetilde}
\newcommand{\wh}{\widehat}

\RequirePackage{color}

\definecolor{references}{rgb}{0,0,1}
\RequirePackage[pdftex,
colorlinks = true,
urlcolor = references, 
citecolor = references, 
linkcolor = references, 
]
{hyperref}

\title[Tautological classes and symmetry]{Tautological classes and symmetry in Khovanov-Rozansky homology}
\author{Eugene Gorsky}
\address{Department of Mathematics\\ University of California Davis\\ One Shields Avenue, Davis CA 95616 USA}
\email{egorskiy@math.ucdavis.edu}
\author{Matthew Hogancamp}
\address{Department of Mathematics\\Northeastern University\\ 360 Huntington Ave, Boston, MA 02115}
\email{m.hogancamp@northeastern.edu}
\author{Anton Mellit}
\address{Faculty of Mathematics, University of Vienna\\ Oskar-Morgenstern-Platz 1, Vienna 1090, Austria}
\email{anton.mellit@univie.ac.at}

\begin{document}

\begin{abstract}
    We define a new family of commuting operators $F_k$ in Khovanov-Rozansky link homology, similar to the action of tautological classes in cohomology of character varieties. We prove that $F_2$ satisfies ``hard Lefshetz property" and hence exhibits the symmetry in Khovanov-Rozansky homology conjectured by Dunfield, Gukov and Rasmussen.
\end{abstract}

\maketitle


\section{Introduction}

In 2005 Dunfield, Gukov and Rasmussen \cite{DGR} proposed a remarkable conjecture about the structure of triply graded Khovanov-Rozansky link homology \cite{KR2,Kh} categorifying HOMFLY-PT polynomial.

\begin{conjecture}
\label{conj: DGR}
Let $K$ be a knot and $\oHHH(K)=\bigoplus \oHHH_{i,j,k}(K)$ be its {\em reduced} triply graded homology, where $i$ is the $a$-grading, $j$ is the quantum grading and $k$ is the homological grading. Then $\dim\oHHH_{i,-2j,k}(K)=\dim \oHHH_{i,2j,k+2j}(K)$.
\end{conjecture}

The conjecture was motivated by the well-known symmetry of the HOMFLY-PT polynomial $\overline{P}_{K}(a,q)=\overline{P}_{K}(a,q^{-1})$, where
$$
\overline{P}_{K}(a,q)=\sum_{i,j,k}a^{i}q^{j}(-1)^{k}\dim \oHHH_{i,j,k}(K).
$$

This conjecture was verified in \cite{DGR}  in numerous examples, and was later related to deep results about algebraic geometry of compactified Jacobians \cite{OS,ORS}, Hilbert schemes of points on the plane \cite{GH,GN,GNR, OR,OR20}, representation theory of rational Cherednik algebras \cite{GORS}  and combinatorics of $q,t$-Catalan numbers \cite{G,GM1,EH,H,HM}.  In particular, the third author \cite{Mknots} computed Khovanov-Rozansky homology for all torus knots, but the resulting combinatorial expression is not manifestly symmetric. Nevertheless, the work of the third author on rational shuffle conjecture \cite{Mshuffle} implies that it is indeed symmetric, and Conjecture \ref{conj: DGR} holds for torus knots.

For general knots, Conjecture \ref{conj: DGR} remained open until recent series of papers of Oblomkov-Rozansky \cite{OR,OR2,OR3,OR4,OR20} who resolved it in general, and Galashin-Lam \cite{GL} who proved it for a  class of knots related to positroid varieties. Both papers used very heavy machinery from geometric representation theory: matrix factorizations on Hilbert schemes of points and graded Koszul duality for category $\mathcal{O}$.

In this paper we give a more direct algebraic proof of this conjecture.

\begin{theorem}
\label{thm:intro symmetry}
Conjecture \ref{conj: DGR} is true for all knots.
\end{theorem}

For the unreduced homology, or for links with several components, the conjecture cannot be extended verbatim. Indeed, the unreduced Khovanov-Rozansky homology $\HHH(L)$ is a finitely generated graded module over the polynomial ring $\C[x_1,\ldots,x_c]$ where $c$ is the number of  components of $L$, and the symmetry in Conjecture \ref{conj: DGR} must change the degrees of these variables.

In \cite{GH} the first and second authors proposed a solution to this problem by introducing $y$-ified link homology $\HY(L)$ which is naturally a module over $\C[x_1,\ldots,x_c,y_1,\ldots,y_c]$. In Theorem \ref{thm:intro y symmetry} we prove that the $y$-ified homology is indeed symmetric in the sense of Conjecture \ref{conj: DGR}, and this symmetry exchanges the action of $x_i$ and $y_i$.
In the case of knots $c=1$, this implies Theorem \ref{thm:intro symmetry}.

The key idea of the proof comes from the recent proof of the ``curious hard Lefshetz" property for character varieties by the third author \cite{Mchar}. Following \cite{Mchar}, to a positive braid $\beta$ on $n$ strands one can associate a character variety  (also known as braid variety) $X_{\beta}$, and by \cite{WW} the homology of $X_{\beta}$ is closely related to the Khovanov-Rozansky homology of the closure of $\beta$. Given a symmetric function $Q(x_1,\ldots,x_n)$ of degree $r$, one can define a closed algebraic $(2r-2)$-form $u_Q$ on the character variety $X_{\beta}$, which represents a certain tautological cohomology class. In particular, for $Q=\sum x_i^2$ we get a class $u_2\in H^2(X_{\beta})$. The main result of \cite{Mchar} then states that cup product with certain powers of $u_2$ satisfies the ``curious hard Lefshetz" property, that is,
yields an isomorphism between certain associated graded components of the weight filtration in the cohomology of  $X_{\beta}$. The proof goes by using a geometric analogue of the skein relation to decompose the varieties into strata, and verifying the Lefschetz property on each stratum by a direct computation.

In this paper, we define analogues of the forms $u_Q$ acting in Khovanov-Rozansky homology. The construction outlined below is completely formal and uses the properties of Soergel bimodules. We have not been able to find a direct geometric connection between the two stories, but the geometry of character varieties and related constructions in group cohomology \cite{Jeffrey}, nevertheless, motivates a lot of the work in this paper, see Appendix \ref{sec: groups} for more details.

\subsection{The dg algebra $\CA$}

The key role in our proof of Conjecture \ref{conj: DGR} is played by a remarkable dg algebra $\CA$ which was first constructed by Abel and the second author in \cite{AH}, although they used a slightly different presentation. It has generators $x_1,\ldots,x_n,x'_1,\ldots,x'_n$ of homological degree 0, $\xi_1,\ldots,\xi_n$ of homological degree $1$ and $u_1,\ldots,u_n$ of homological degree $2$ with the following differential:
$$
d(x_i)=d(x'_i)=0,\ d(\xi_i)=x_i-x'_i,\ d(u_k)=\sum_{i=1}^{n} h_{k-1}(x_i,x'_i)\xi_i,
$$
where $h_k$ is the complete symmetric function of degree $k$. In addition, we impose the relations
$$f(x_1,\ldots,x_n)=f(x'_1,\ldots,x'_n)$$ for arbitrary symmetric functions $f$.

The algebra $\CA$ is naturally an $R$-$R$ bimodule where $R=\C[x_1,\ldots,x_n]$   acts on the left by $x_i$ and on the right  by $x'_i$. One can prove that $H_0(\CA)\simeq R$ and the higher homologies vanish (see Proposition \ref{prop: A resolution}), so $\CA$ is a free resolution of $R$ over $R\otimes_{R^{S_n}} R$.

The homotopy category of dg modules over $\CA$ is localized by morphisms which become homotopy equivalences when restricted to the subalgebra generated by $x_i$ and $x_i'$. Modules which become isomorphic in this category are called \emph{weakly equivalent}.

Any homomorphism $\Delta:\CA\to \CA\otimes_{R}\CA$ defines an $\CA$-module structure on the tensor product $M\otimes_{R}N$ of arbitrary $\CA$-modules $M, N$.
\begin{theorem}
\label{thm:intro coproduct}
There is a coproduct $\Delta:\CA\to \CA\otimes_{R}\CA$ inducing a tensor product of $\CA$-modules which is associative up to weak equivalence.
\end{theorem}

Recall that to a braid $\beta$ one can associate a complex $T_{\beta}$ of Soergel bimodules (called {\em Rouquier complex}) which is a tensor product of Rouquier complexes  $T_i,T_i^{-1}$ for crossings (see Section \ref{sec:rouquier} for more details). By defining the action of $\CA$ on $T_i,T_i^{-1}$ and using the coproduct $\Delta$ we can extend it to Rouquier complexes of arbitrary braids, and arrive at the following result:

\begin{theorem}
\label{thm:intro action}
For an arbitrary braid $\beta$ on $n$ strands there is an action of $\CA$ on the Rouquier complex $T_{\beta}$. Here the action of $x_i$ is standard and the action of $x'_i$ is twisted by the action of the permutation $w({\beta})$ corresponding to $\beta$.
The action of $\CA$ is well defined and is invariant under Reidemeister moves up to weak equivalence.
\end{theorem}

\begin{example}
The minimal Rouquier complex for the full twist on two strands has the form $B\to B\to R$. The action of the dg algebra $\CA$ on it is shown on the following diagram:
\begin{equation}
\label{eq: intro full twist}
\begin{tikzcd}
B\arrow{r} & B \arrow{r} \arrow[bend left]{l}{\xi_{1}=-\xi_2}& R \arrow[bend right,swap]{ll}{u_2}
\end{tikzcd}
 \end{equation}
\end{example}

More generally, we have the following explicit description of $u_2$ (see Lemma \ref{lem: u2}):

\begin{theorem}
\label{thm: intro u2}
One can  present the action of  $u_2$ on a Rouquier complex $T_{\beta}$ explicitly as follows:
\begin{equation}
\label{eq: intro u2}
u_2=\sum_{i=1}^{n}\sum_{j<k} \xi_i^{(j)}\otimes \xi_{w_{jk}(i)}^{(k)}
\end{equation}
Here $\xi_i^{(j)}$ is the action of $\xi_i$ at the $j$-th crossing, and $w_{jk}$ is the permutation corresponding to the piece of the braid $\beta$ between $j$-th and $k$-th crossings.
\end{theorem}

\subsection{Action in link homology}

Next, we consider the impact of the dg algebra $\CA$ on Khovanov-Rozansky homology. Recall that the Khovanov-Rozansky homology $\HHH(L)$  of the link $L$ obtained as the closure of $\beta$ is defined as the homology of the complex $\HH(T_{\beta})$, where $\HH$ denotes Hochschild homology of Soergel bimodules. Since $\HH$ identifies the actions of $x_i$ and $x'_i$, we arrive at the following result.

\begin{theorem}
Consider the dg algebra:
\[
    \CCA_{c,n} = \C\left[(x_i)_{i=1}^c, (\xi_i)_{i=1}^c, (u_k)_{k=1}^n \;|\; d x_i=d\xi_i=0,\; d u_k = k \sum_{i=1}^c x_i^{k-1} \xi_i\right].
    \]
If $\beta$ is a braid on $n$ strands whose closure has $c$ connected components, then $\HH(T_{\beta})$ is a dg module over $\CCA_{c,n}$.
\end{theorem}

In fact, we prove in Proposition \ref{prop: twisted projector}  that $\CCA_{c,n}$ is quasi-isomorphic to the quotient of $\CA_n$ by the relations $x_i=x'_{w(i)}$. Next, we consider the dg algebra $\CCA_{c,\infty}$ which is the inverse limit of $\CCA_{c,n}$ as $n\to \infty$ and the homotopy category of dg modules over it localized by quasi-isomorphisms.\footnote{Note that this category is equivalent to the homotopy category of $A_{\infty}$-modules over $\CCA_{c,\infty}$ viewed as an $A_{\infty}$ algebra.} Again, we call isomorphic objects in the localized category weakly equivalent.

\begin{theorem}
\label{thm: intro markov}
Let $L$ be a link with $c$ components, presented as the closure of a braid $\beta$. Then its Khovanov-Rozansky complex $\HH(T_{\beta})$ is a dg module over $\CCA_{c,\infty}$, and this module structure does not depend on the presentation of $L$ as a braid closure up to weak equivalence.
\end{theorem}

\subsection{Relation to $y$-ification}

A part of Theorem \ref{thm:intro action} implies that there is an action of operators $\xi_i$ of $(A,Q,T)$ degree $(0,2,1)$ on an arbitrary Rouquier complex $T_{\beta}$
such that $d(\xi_i)=x_i-x'_{w^{-1}(i)}$, so that $\xi_i$ is a homotopy between $x_i$ and $x'_{w^{-1}(i)}$. Here $w$ denotes the permutation associated to the braid $\beta$. Such homotopies were considered before as ``dot-sliding homotopies" \cite{BS,CK,S} and played an important role in the construction of the $y$-ified Khovanov-Rozansky homology by the first and second authors in \cite{GH}. Namely, we tensor the Rouquier complex with the polynomial ring $\C[y_1,\ldots,y_n]$ and deform the differential $d$ to
$$
d_y=d+\sum_{i=1}^{n} \xi_i y_i 
$$
Here the formal variables $y_i$ have  $(A,Q,T)$-degrees $(0,-2,-2)$, so that the differential $d_y$ is homogeneous of degree $(0,0,-1)$.
The paper \cite{GH} also defined the $y$-ified link homology $\HY(\beta)$ by taking the Hochschild homology of the $y$-ified Rouquier complex for a braid $\beta$. Here we prove the following:

\begin{theorem}
\label{thm:intro Dk}
For any braid $\beta$ with the associated permutation $w$ there exist operators
$$
F_k:=\sum_{i=1}^{n} h_{k-1}\left(x_i,x'_{w^{-1}(i)}\right)\frac{\partial}{\partial y_i}+u_k
$$
 of $(A,Q,T)$-degree $(0,2k,2)$ on the $y$-ified Rouquier complex
such that
$$
[d_y,F_k]=0,\ [F_k,F_{m}]=0,\ [F_k,x_i]=0,\ [F_k,y_i]=h_{k-1}\left(x_i,x'_{w^{-1}(i)}\right).
$$
\end{theorem}

In particular, $F_k$ is a chain map and hence defines an interesting endomorphism of $\HY(\beta)$ . Following Theorem \ref{thm: intro markov} we prove the following result:

\begin{theorem}
\label{thm: intro yified markov}
Suppose that $L$ is a link with $c$ components. There is an action of  operators $x_1,\ldots,x_c,y_1,\ldots,y_c,F_k\ (k\ge 1)$ on the $y$-ified homology $\HY(L)$ satisfying the equations
$$
[F_k,F_{m}]=0,\ [F_k,x_i]=0,\ [F_k,y_i]=kx_i^{k-1}.
$$
This action is a link invariant and does not depend on a presentation of $L$ as a braid closure.
\end{theorem}

\subsection{Basic objects}
In Lemma \ref{lem:skein} we interpret the skein relation
$$
[T_i\to T_i^{-1}]\simeq [R\to R]
$$
as a distinguished triangle of $\CA$-modules involving the so-called Koszul objects, which already appeared in \cite{GH}. Applying the skein relations and Markov moves we can simplify links until the invariant of any link is represented as an iterated cone of products of Koszul objects, see Proposition \ref{prop:reduction 1}. We call such products \emph{basic objects}. They are completely classified in Section \ref{sec:basic objects classification}, we give a formula for the action of $u_k$ on them in Proposition \ref{prop:basic object}, and their $y$-ification is given explicitly in Section \ref{sec:yified building blocks}. Combinatorially, basic objects are classified by the same data as Riemann surfaces with boundary on the link, and we think of the decomposition into basic objects as an algebraic counterpart of the stratification in \cite{Mchar}.

\subsection{Hard Lefshetz and symmetry}

Finally, we apply all of  the above results to prove Conjecture \ref{conj: DGR}. By Theorem \ref{thm:intro Dk} there is an operator
$$
F_2=\sum_{i=1}^{n}\left(x_i+x'_{w^{-1}(i)}\right)\frac{\partial}{\partial y_i}+u_2
$$
 of degree $(A,Q,T)=(0,4,2)$ acting on the $y$-ified Khovanov-Rozansky homology $\HY(\beta)$.

\begin{theorem}
\label{thm:intro y symmetry}
Suppose that $j\ge 0$. Then the operator $(F_2)^{j}$ satisfies a ``hard Lefshetz" property, that is, yields isomorphisms:
$$(F_2)^{j}:\HY_{i,-2j,k}(\beta)\xrightarrow{\sim} \HY_{i,2j,k+2j}(\beta)$$
for all $i,j,k$ and an arbitrary braid $\beta$.
\end{theorem}

Analogously to \cite{Mchar}, the Lefschetz property is proved for basic objects directly, and then the theorem is deduced from Lemma \ref{lem:5 lemma}.

\begin{corollary}
\label{cor:sl2 intro}
There is an action of the Lie algebra $\mathfrak{sl}_2=\langle E,F,H\rangle$ on $\HY(\beta)$ where $F=F_2$ and $H$ acts on $\HY_{i,2j,k}(\beta)$ by a scalar $j$. This action intertwines $x_i$ and $y_i$ such that $\C[x_1,\ldots,x_c,y_1,\ldots,y_c]\simeq (S^{\ast}V)^{\otimes c}$ where $V$ is the defining representation of  $\mathfrak{sl}_2$.
\end{corollary}


We deduce Theorem \ref{thm:intro symmetry} from Theorem \ref{thm:intro y symmetry} as follows. Recall \cite{GH} that for a knot $K$ one can write
$\HY(K)=\oHHH(K)\otimes \C[x,y]$, where $x=\sum x_i$ and $y=\sum y_i$. Since $[F_2,y]=2x$ and $[F_2,x]=0$, it is easy to see that one can write  $F_2=\overline{F_2}+2x\frac{\partial}{\partial y}$ for some operator $\overline{F_2}$ on $\oHHH(K)$, and $\overline{F_2}^{j}$ induces an isomorphism between $\oHHH_{i,-2j,k}(K)\simeq \oHHH_{i,2j,k+2j}(K)$.

\begin{remark}
After this paper appeared on the arXiv, Chandler and the first author proved in \cite{CG} that the operators $F_k$ commute with the differentials in the Rasmussen spectral sequence \cite{RasDiff} from triply graded Khovanov-Rozansky homology to $\mathfrak{sl}_N$ homology, and hence yield well-defined operations in $\mathfrak{sl}_N$ homology. 

The action of $E$, however, does not commute with Rasmussen differentials, and both the action of $\mathfrak{sl}_2$ from Corollary \ref{cor:sl2 intro} and the symmetry from Theorem \ref{thm:intro symmetry} are specific to the triply graded homology.
See \cite{CG} for more details and examples.
\end{remark}

\subsection{Structure of the paper}

In Section \ref{sec:homotopy theory} we recall several important constructions in homological algebra which are probably well-known in the case of a dg algebra over a field, but for which we were not able to find a reference for the relative setting. In particular, we study the localization of the category of dg modules over a dg algebra $\CA$ by the class of morphisms which admit $\CB$-linear homotopy inverses (but not necessary $\CA$-linear homotopy inverses), where $\CB$ is a subalgebra of $\CA$. We give a concise description of the corresponding dg category as a full subcategory of the homotopy category of $\CA$-modules in Theorem \ref{thm:localization}. Homomorphisms in this category are related to the Hochschild homology in Proposition \ref{prop:localization hochschild}. We also study the pullbacks of dg modules along dg algebra homomorphisms $\CA\to \CA'$ which preserve $\CB$ and give a sufficient condition when  pullbacks induce equivalences of categories (Lemma \ref{lem:algebra weak equivalence}) and under which pullbacks over different homomorphisms produce weakly equivalent modules (Lemma \ref{lem:algebra homomorphisms}). Finally, we formulate the Koszul duality for such categories in the special case of a linear dg algebra in Theorem \ref{thm:koszul}, which is later used to understand the connection between $\CA$-modules and $y$-ifications.

In Section \ref{sec:algebra} we define the dg algebra $\CA$ and study its properties. We construct an explicit coproduct on $\CA$ in Section \ref{sec:coproduct}, and prove Theorem \ref{thm:intro coproduct}.
The results of Section \ref{sec:homotopy theory} are applied to the algebra $\CA$ and its subalgebra $\CB=R\otimes_{R^{S_n}}R$. 



In Section \ref{sec:modules} we define Rouquier complexes and construct the action of $\CA$ on them, proving Theorem \ref{thm:intro action}.  We compute the action of $u_2$ explicitly and prove equation \eqref{eq: intro u2} in Lemma \ref{lem: u2}.

In Section \ref{sec:markov} we study invariance under the Markov moves and prove Theorem \ref{thm: intro markov} (Theorem \ref{thm:link invariant}).

 In Section \ref{sec:yified} we extend these results to $y$-ified Rouquier complexes and prove Theorem \ref{thm:intro Dk} (Corollary \ref{cor: full Markov}).

Finally, in Section \ref{sec:Lefshetz} we study various useful properties of complexes with Lefshetz endomorphisms and conclude the proof of Theorem \ref{thm:intro y symmetry} (Theorem \ref{thm:full lefschetz}).

The two appendices discuss ``higher $A_{\infty}$ coproducts'' on $\CA$, and their relation to group cohomology. We believe that the $\CA$-algebra structure we construct here is a manifestation of the existence of tautological classes on character varieties constructed via transgression and group homology. Although there is a visible similarity between formulas in the appendices, the problem of finding a direct connection remains open.

\section*{Acknowledgments}

The authors are thankful to Michael Abel, Roger Casals, Alex Chandler, Mikhail Gorsky, Bernhard Keller, Nitu Kitchloo, Alexei Oblomkov, Jacob Rasmussen, Lev Rozansky, Vivek Shende, Jose Simental Rodriguez and Paul Wedrich for many useful discussions.   E. G. was partially supported by the NSF grants DMS-1700814, DMS-1760329 and DMS-2302305.  M. H. was supported by NSF grant DMS-2034516.    A. M. was supported by the projects Y963-N35 and
P-31705 of the Austrian Science Fund, and by the consolidator grant No. 101001159 of the European Research Council.

\section{Homotopy theory}\label{sec:homotopy theory}

\subsection{Basic definitions}
Discussing modules and algebras endowed with several different gradings, there will always be one special grading called the homological grading and that grading will be responsible for the sign rules\footnote{Of course, only the parity of that grading matters for the sign rules.}. The homological degree of a homogeneous element $x$ is denoted by $|x|=\deg_h x$. In expressions involving degrees we assume, but do not always mention, that the respective elements are homogeneous. All dg algebras are assumed to be unital. The differential has degree $-1$.

Let $\CA$ be a dg algebra over $\C$. We consider the category of dg modules over $\CA$, which we simply call $\CA$-modules. Thus an $\CA$-module $X$ is a sequence of vector spaces $(X_i)_{i\in\Z}$ endowed with a differential $d:X_i\to X_{i-1}$ and an $\CA$-action $\CA \times X \to X$ satisfying $|a x| = |a|+ |x|$ and the property that for any $a\in \CA$ the action on $X$ of the \emph{super-commutator} $[d,a] = d a - (-1)^{|a|} a d$ coincides with the action of $d(a)$:
\[
[d,a] = d(a)\qquad (a\in\CA).
\]
The \emph{homological shift} is defined by $(X[k])_i = X_{k+i}$. For $x\in X_j$ we have $x[k]\in (X[k])_{j-k}$, the differential and the action change signs:
\[
d (x[k]) = (-1)^k (dx)[k],\quad a (x[k]) = (-1)^{k |a|}(ax)[k].
\]
The morphisms of $\CA$-modules form a complex
\[
\Hom_k^\CA(X, Y) = \left\{ f\in\Hom_\C(X, Y) \;|\;\deg_h f = k,\; [a, f] = 0 \; (a\in A) \right\},
\]
with the differential
\[
d: \Hom_k^\CA(X, Y) \to  \Hom_{k-1}^\CA (X, Y),\ d(f) = [d,f].
\]
With our sign conventions we have
\[
\Hom_k^\CA(X, Y) = \Hom_0^\CA(X, Y[k]).
\]
Note that the elements of $\Hom_0(X,Y)$ which lie in the kernel of $d$ are precisely the dg module homomorphisms, and two such homomorphisms are homotopic precisely when their images in $H_0(\Hom_\bullet(X,Y))$ coincide.

Let $\CA-\Mod$ denote the homotopy category of dg modules over $\CA$. Its objects are dg modules over $\CA$ and morphisms are morphisms of dg modules viewed up to homotopy, i.e.
\[
\Hom_{\CA-\Mod}(X, Y) = H_0(\Hom_\bullet(X,Y)) = \frac{\Hom_{\CA-\Mod_\strict}(X, Y)}{[d,\Hom_1(X,Y)]}.
\]
Actual morphisms of dg modules will be called \emph{strict} morphisms. The corresponding category is denoted $\CA-\Mod_{\strict}$.
Two objects $X,Y$ are \emph{homotopy equivalent} if they are isomorphic in $\CA-\Mod$. Unwrapping this definition, a homotopy equivalence is specified by two strict morphisms $f:X\to Y$, $g:Y\to X$ and two homotopies $f\circ g \cong \Id_{Y}$, $g\circ f \cong \Id_{X}$. A strict morphism $f:X\to Y$ is a \emph{homotopy equivalence} if there exists a morphism $g$ and homotopies as above.

An object is \emph{contractible} if it is homotopy equivalent to the zero object, equivalently if the identity map is homotopic to the zero map, equivalently if some invertible endomorphism is homotopic to the zero map.
A very useful lemma is
\begin{lemma}\label{lem:contractible cone}
	A strict morphism $f:X\to Y$ is a homotopy equivalence if and only if the cone $[X \to \underline{Y}]$ is contractible.
\end{lemma}

Here and below we underline the term in homological degree zero. The category $\CA-\Mod_\strict$ is pre-triangulated (see \cite{BK}), which means that twisted complexes of $\CA$-modules are again $\CA$-modules. Twisted complexes of $\CA$-modules are defined as follows. Given a bounded below sequence of dg $\CA$-modules $(X_i)_{i\in \Z}$ and a collection of maps $q_{i,j}\in\Hom_{i-1-j}(X_i, X_j)$ $(i>j)$  satisfying
$$
d(q_{i,j}) + \sum_{k} q_{i,k} q_{k,j} = 0
$$
we form a new object $X$ by setting
\begin{equation}\label{eq:twisted complex}
X=\left(\bigoplus_{i\in \Z} X_i [-i], \widetilde d\right),\quad \widetilde d = d  + \sum_{j<i} q_{i,j}.
\end{equation}
This makes sense because for each $i$ the map $q_{i,j}$ vanishes for all but finitely many $j$. When the maps are clear from the context, we visualize twisted complexes as follows:
\[
\left[\cdots \to X_{1} \to \underline{X_0} \to X_{-1} \to \cdots \right].
\]
The following is standard:
\begin{lemma}\label{lem:twisted complex}
	Suppose a twisted complex $X=\left(\bigoplus_{i\in \Z} X_i [-i], \widetilde d\right)$ is bounded from below
and each $X_i$ is contractible. Then $X$ is contractible.
\end{lemma}

\begin{remark}
    There is a dual definition where we require sequence $(X_i)_{i\in \Z}$ to be bounded above and use the direct product instead of the direct sum in \eqref{eq:twisted complex}. The corresponding analogue of Lemma \ref{lem:twisted complex} is also true. In the case the sequence is bounded both above and below, the two constructions coincide.
\end{remark}

\begin{remark}
Sometimes we will need to work with {\em curved} dg modules over $\CA$ introduced in \cite{pos1993}. A curved dg module $X$ with curvature $W$ has a differential satisfying $d^2=W$ where $W$ is a closed central element of $\CA$. All other definitions are as above, and one can check that for fixed $W$ curved dg modules form a pre-triangulated dg category.
\end{remark}

\subsection{Restriction/induction functors}
A homomorphism of dg algebras $\varphi:\CB\to\CA$ is a linear map which preserves the grading and the unit, and commutes with the product and the differential. Given such a homomorphism, any $\CA$-module $N$ is naturally a $\CB$-module, which we denote by $\varphi^* N$ and call the \emph{restriction} of $N$. Given a $\CB$-module $M$, the \emph{induction} resp. \emph{coinduction} of $M$ is defined by
\[
\varphi_! M := \CA\otimes_\CB M\quad\text{resp.}\quad \varphi_* M := \Hom^\CB_\bullet(\CA, M).
\]

Clearly, all three operations are functorial in the sense that for any $N,N'$ resp. $M, M'$ we have morphisms of $\Hom$-complexes
\[
\Hom^\CA_\bullet(N, N')\to \Hom^\CB_\bullet(\varphi^* N, \varphi^* N'),\quad \Hom^\CB_\bullet(M, M')\to \Hom^\CA_\bullet(\varphi_! M, \varphi_! M'),
\]
and similarly for the coinduction $\varphi_*$. Note that in the definitions of $\varphi_!, \varphi_*$ we do not take any resolutions, so the functors are not exact in any sense. Nevertheless, all three functors commute with cones and induce functors on the corresponding homotopy categories.

We have natural adjunction isomorphisms on the level of complexes
\[
\Hom^\CA_\bullet(\varphi_! M, N) = \Hom^\CB_\bullet(M, \varphi^* N),
\quad \Hom^\CA_\bullet(N, \varphi_* M) = \Hom^\CB_\bullet(\varphi^* N, M).
\]

\begin{definition}
    Given a homomorphism of dg algebras $\varphi:\CB\to\CA$, an $\CA$-module $N$ is called \emph{induced} if it is isomorphic to a bounded below twisted complex of modules $X_i$ where each $X_i$ is of the form $\varphi_!(M_i)$ for some $\CB$-module $M_i$.
\end{definition}
\begin{remark}
    We can similarly define coinduced modules if we consider bounded above twisted complexes, but then we need to use direct product instead of the direct sum in \eqref{eq:twisted complex}.
\end{remark}
\begin{remark}
    In the case when $\CB=\C$, the notion of an induced module coincides with the more familiar notion of a semi-free module.
\end{remark}

\subsection{Localization}

\begin{definition}
    Given a homomorphism of dg algebras $\varphi:\CB\to\CA$, a morphism of $\CA$-modules $f:M\to M'$ is a \emph{weak equivalence} and the modules $M$ and $M'$ are called \emph{weakly equivalent} relative to $\CB$ if $\varphi^*f$ is a homotopy equivalence of $\CB$-modules. If $f$ itself is a homotopy equivalence of $\CA$-modules, we say that $f$ is a \emph{strong equivalence} and the modules are homotopy (or strongly) equivalent.
\end{definition}
So we have four notions of equivalence related as follows
\[
\text{isomorphism}\Rightarrow\text{strong equivalence}\Rightarrow \text{weak equivalence} \Rightarrow \text{quasi-isomorphism}.
\]

\begin{lemma}\label{lem:lifting}
    Given a homomorphism of dg algebras $\varphi:\CB\to\CA$, $\CA$-modules $M$, $N$, $N'$ and morphisms $f$, $g$ below, suppose that $M$ is induced and $g$ is a weak equivalence.
    \[
    \begin{tikzcd}
        & N' \arrow{d}{g}\\
        M \arrow{r}{f} \arrow[dashed]{ru}{f'} & N
    \end{tikzcd}
    \]
    Then there exists a unique up to homotopy morphism $f'$ making the diagram commutative up to homotopy.
\end{lemma}
\begin{proof}
    Composition with $g$ induces a morphism of complexes
    \[
    \Hom^\CA_\bullet(M, N') \to \Hom^\CA_\bullet(M, N).
    \]
    It is enough to show that this morphism is a quasi-isomorphism. By Lemma \ref{lem:contractible cone} it is sufficient to show that
    \[
    \left[\Hom^\CA_\bullet(M, N') \to \Hom^\CA_\bullet(M, N)\right] = \Hom^\CA_\bullet\left(M, [N'\to N]\right)
    \]
    is contractible. Since $M$ is a twisted complex of modules of the form $\varphi_!(M_i)$ and $\Hom$ commutes with the formation of twisted complexes\footnote{The $\Hom$ complex is strictly functorial and sends direct sums to direct products}, by Lemma \ref{lem:twisted complex} it is sufficient to show that each $\Hom^\CA_\bullet\left(\varphi_!(M_i), [N'\to N]\right)$ is contractible. Using the adjunction, this complex is equivalent to the complex
    \[
    \Hom^\CB_\bullet\left(M_i, [\varphi^* N'\to \varphi^* N]\right) = \left[\Hom^\CB_\bullet(M_i, \varphi^* N') \to \Hom^\CB_\bullet(M_i, \varphi^* N)\right],
    \]
     which is contractible since $\varphi^*(g)$ is a homotopy equivalence.
\end{proof}

\begin{definition}
    Given a homomorphism of dg algebras $\varphi:\CB\to\CA$ and an $\CA$-module $M$, a \emph{resolution} of $M$ relative to $\CB$ is an induced $\CA$-module $\widetilde{M}$ together with a weak equivalence $\widetilde{M}\to M$ called the \emph{counit}.
\end{definition}

Now Lemma \ref{lem:lifting} easily implies
\begin{corollary}
    For any $\CA$-module $M$, if a resolution exists it is unique up to a strong homotopy equivalence.
\end{corollary}

Finally, we have
\begin{lemma}\label{lem:resolution}
    Given a homomorphism of dg algebras $\varphi:\CB\to\CA$, any $\CA$-module $M$ has a resolution. Moreover, this resolution can be chosen functorially on $M$.
\end{lemma}
\begin{proof}
For any $k\geq 1$ let $\widetilde{M}_k = \CA^{\otimes_{\CB} k} \otimes_{\CB} M$. These $\widetilde{M}_k$ form a twisted complex
\[
\widetilde M := \left[\cdots \to \CA \otimes_{\CB} \CA \otimes_{\CB} \CA \otimes_{\CB} M \to \CA \otimes_{\CB} \CA \otimes_{\CB} M \to \underline{\CA \otimes_{\CB} M}\right],
\]
with the bar differential
\[
d_b(a_1 \otimes \cdots \otimes a_k \otimes x)= \sum_{i=1}^{k} (-1)^{i-1} a_1 \otimes \cdots \otimes a_{i} a_{i+1} \otimes \cdots \otimes a_k \otimes x.
\]
The higher differentials vanish.
For instance, the first  few maps are
\[
d_b(a_1 \otimes a_2 \otimes x) = a_1 a_2 \otimes x - a_1\otimes a_2 x,
\]
\[
d_b(a_1 \otimes a_2 \otimes a_3\otimes x) = a_1 a_2 \otimes a_3 \otimes x - a_1 \otimes a_2 a_3 \otimes x + a_1\otimes a_2 \otimes a_3 x.
\]
It is a standard fact that $d_b^2=0$.
Let $\varepsilon: \widetilde{M} \to M$ be the counit map induced by the product map
\[
\widetilde{M}_1=\CA \otimes_{\CB} M \to M.
\]
Let $\widetilde{M}_0=M$. The cone of the counit map is the complex
\[
\left[\cdots \to \CA \otimes_{\CB} \CA \otimes_{\CB} M \to \CA \otimes_{\CB} M \to \underline{M}\right].
\]
Define a homotopy $h: \widetilde{M}_i\to \widetilde{M}_{i+1}$ by
\[
h(a_1\otimes\cdots\otimes a_i\otimes x) = 1\otimes a_1\otimes\cdots\otimes a_i\otimes x.
\]
This map is not $\CA$-linear, but it is $\CB$-linear, it commutes with the differential in each $\widetilde{M}_i$ and satisfies $h d + d h = \Id$. Thus the cone of the complex is contractible as an object of $\CB-\Mod$ and therefore the counit map is a homotopy equivalence in $\CB-\Mod$ by Lemma \ref{lem:contractible cone}.

Thus $\widetilde M$ together with the counit map $\varepsilon$ is a resolution of $M$. Its construction is clearly functorial in $M$.
\end{proof}

\begin{definition}
    The category $\CA/\CB-\Mod$ (abbreviated as $\CA/\CB$) is the localization of the homotopy category $\CA-\Mod$ with respect to weak equivalences relative to $\CB$. Note that $\CA-\Mod$ itself can be viewed as $\CA/\CA$.
\end{definition}

Explicitly, this is the category whose objects are $\CA$-modules and morphisms are represented by zig-zags
\begin{equation}\label{eq:sequence of hats}
    X \leftarrow X_1 \to X_2 \leftarrow X_3\rightarrow \cdots \to Y.
\end{equation}
in which all the arrows pointing to the left are weak equivalences. Two morphisms are considered equivalent if they can be related by a sequence of transformations where we are allowed to replace any arrow by a homotopic arrow; identity arrow can be inserted or removed; two consecutive arrows pointing in one direction can be replaced by their composition; finally the composition of a weak equivalence and its formal inverse is equivalent to the identity.

\begin{theorem}\label{thm:localization}
    Given a homomorphism of dg algebras $\varphi:\CB\to\CA$, the localization category $\CA/\CB$ is equivalent to the full subcategory of $\CA-\Mod$ whose objects are induced from $\CB$.
\end{theorem}
\begin{proof}
	Existence of resolutions implies that any object in $\CA/\CB$ is isomorphic to an object induced from $\CB$. Let $M$ be such an object. By Lemma \ref{lem:lifting} the functor $\Hom_{\CA-\Mod}(M, -)$ sees all weak equivalences as isomorphisms. Hence morphisms from $M$ to any object $N$ in $\CA-\Mod$ are in a natural bijection with morphisms in $\CA/\CB$.
\end{proof}

\begin{remark}
    Explicitly, morphisms in $\CA/\CB-\Mod$ from $M$ to $M'$ can be described as $\Hom_{\CA-\Mod}(\widetilde M, M')$, where $\widetilde M$ is the resolution from Lemma \ref{lem:resolution}. Suppose $\CB=\C$. Then, unwrapping the construction of $\widetilde M$ leads to the fact that morphisms in $\CA/\C$ are nothing else but $A_\infty$ module homomorphisms up to homotopy. For the case when $\CB$ is an arbitrary dg algebra, we have not seen a relative version of the notion of $A_\infty$ module homomorphisms in the literature, but our definition looks like a natural generalization.
\end{remark}

\begin{remark}
    We sketch how the construction of the resolution in Lemma \ref{lem:resolution} can be naturally obtained from attempting to solve the localization problem directly. First, we note that instead of localizing with respect to all weak equivalences, it is sufficient to localize with respect to those weak equivalences $f:M\to N$ for which $\varphi^* f$ is a retraction, i.e. there exists $g:\varphi^* N \to \varphi^* M$ such that $\varphi^*f\circ g=\Id_N$. In this case the homotopy $h\in\Hom^\CB_1(M, M)$ connecting $g\circ \varphi^*f$ and $\Id_M$ can be chosen in such a way that $f\circ h = 0$. Now fix $N$ and consider all possible such $M,f,g,h$. One can produce elements of $M$ by applying $g$ to the elements of $N$ and then acting by $h$ and by elements of $\CA$. Acting by $f$ does not give anything new because $f \circ h=0$ and $\varphi^*f\circ g=\Id_N$. So we can attempt to construct the universal such $M$ by taking formal combinations of these operations. It turns out that we arrive precisely at the construction of $\widetilde N$. The chain map from $\widetilde N$ to $M$ is given by the collection of morphisms $f_k:\widetilde{N}_k\to M$ defined by
    \[
    \widetilde f_{k} (a_1 \otimes a_2 \otimes \cdots\otimes a_k \otimes x) = (-1)^{s} a_1 h a_2 h \cdots h a_k g(x),
    \]
    where $s=|a_{k-1}|+|a_{k-3}|+\cdots$.
\end{remark}

\subsection{Relationship to the Hochschild cohomology}

\begin{proposition}\label{prop:localization hochschild}
	Suppose $\varphi(\CB)$ is in the center of $\CA$ and $\CA$ is free over $\CB$. Then for any $k$ and any $\CA$-modules $M$, $N$ we have
	\[
	H_k\left(\Hom_\bullet^\CA\left(\widetilde M, N\right)\right) = \Ext^{-k}_{\CA\otimes_{\CB} \CA^{\op}}\left(\CA, \Hom_\bullet^{\CB}(M, N)\right) = \HH^{-k}_\CA\left(\Hom_\bullet^{\CB}(M, N)\right),
	\]
	where $\Hom_\bullet^{\CB}(M, N)$ is viewed as a bimodule over $\CA$ in the obvious way\footnote{Here we are using the assumption that $\varphi(\CB)$ is in the center of $\CA$.}, and $\HH$ denotes the Hochschild cohomology. The forgetful functor corresponds to the natural map $\HH^{-k}_\CA\left(\Hom_\bullet^{\CB}(M, N)\right) \to \Hom^{\CB}_k(M, N)$ induced by the homomorphism $\CA\otimes_{\CB} \CA^\op\to \CA$.
\end{proposition}
\begin{proof}
	For any $\CA\otimes_{\CB}\CA^{\op}$-module $Y$ we have
	\[
	\Hom^\CA_\bullet\left(Y \otimes_{\CA} M, N\right) = \Hom^{\CA\otimes_{\CB}\CA^{\op}}_\bullet\left(Y, \Hom^{\CB}_\bullet(M, N)\right).
	\]
	Applying it to $Y_k=\CA^{\otimes_{\CB} k+1}$ ($k\geq 1$) we obtain
	\[
	\Hom_\bullet^\CA\left(\widetilde{M}_{k+1}, N\right) = \Hom_\bullet^{\CA\otimes_{\CB}\CA^{\op}}\left(\CA^{\otimes_{\CB} k+1}, \Hom^{\CB}_\bullet(M, N)\right),
	\]
    where $\widetilde{M}_{k+1}=\CA^{\otimes_{\CB} k+1}\otimes_\CB M$ from Lemma \ref{lem:resolution}.
	It remains to notice that in the case when $\CA$ is free over $\CB$ the modules $Y_k$ form an explicit free resolution of $\CA$ over $\CA\otimes_{\CB}\CA^{\op}$ whose maps are compatible with the maps defining $\widetilde M$.
\end{proof}

\subsection{Relative induction/restriction}
Suppose we have two dg algebras $\CA$, $\CA'$ over $\CB$. A homomorphism of algebras over $\CB$ is a homomorphism $\psi:\CA\to\CA'$ making the following diagram commutative:
\[
\begin{tikzcd}
    \CA \arrow{r}{\psi}  & \CA' \\
    \CB \arrow{u}{\varphi}\arrow{ru}[swap]{\varphi'} &
\end{tikzcd}
\]
Clearly, the restriction $\psi^*$ sends weak equivalences to weak equivalences, and therefore defines a functor $\CA'/\CB\to \CA/\CB$. In the case $\CA=\CB$ we will call $\psi^*$ the \emph{forgetful functor}.

The induction functor $\psi_!$ does not preserve weak equivalences\footnote{Suppose $\CB=\C$. Then weak equivalences are quasi-isomorphisms. If $\psi_!$ preserved quasi-isomorphisms, it would be an exact functor, but we know that it is in general not exact.}. Denote by $R_{\CA}$ resp. $R_{\CA'}$ any functorial resolution on $\CA-\Mod$ and $\CA'-\Mod$ (for example, the one from Lemma \ref{lem:resolution}). Then we have a functor $\psi_!\circ R_{\CA}$.
\begin{lemma}
We have the adjunction
\begin{equation}
\label{eq: relative adjunction}
\Hom_{\CA'/\CB}\left(\psi_! R_{\CA} M, N\right) = \Hom_{\CA/\CB}\left(M, \psi^* N\right),
\end{equation}
where $M$ resp. $N$ is an $\CA$ resp. $\CA'$ module.
\end{lemma}
\begin{proof}
Indeed, the module $\psi_! R_{\CA} M$ is induced, and therefore we can replace the left hand side by
$$
\Hom_{\CA'/\CB}\left(\psi_! R_{\CA} M, N\right)=\Hom_{\CA'}\left(\psi_! R_{\CA} M, N\right).
$$
On the right hand side we can replace $M$ by its resolution $R_{\CA}M$, so the adjunction \eqref{eq: relative adjunction} comes from the usual adjunction
\[
\Hom_{\CA'}\left(\psi_! R_{\CA} M, N\right) = \Hom_{\CA}\left(R_{\CA} M, \psi^* N\right)=\Hom_{\CA/\CB}\left(M, \psi^* N\right).
\]
\end{proof}
The adjunction \eqref{eq: relative adjunction} implies existence of natural homomorphisms
\begin{equation}
\label{eq: relative adjunction 2}
M \to \psi^* \psi_! R_{\CA} M,\qquad \psi_! R_{\CA} \psi^* N \to N.
\end{equation}
We have

\begin{lemma}\label{lem:algebra weak equivalence}
    Suppose $\psi$ is  homotopy equivalence when viewed as a homomorphism of $\CB$-bimodules. Then the adjunction homomorphisms \eqref{eq: relative adjunction 2} are weak equivalences for any $\CA$-module $M$ and $\CA'$-module $N$ and therefore $\psi^*,\psi_! R_{\CA}$ are mutually inverse equivalences of categories $\CA/\CB$ and $\CA'/\CB$.
\end{lemma}
\begin{proof}
    The first homomorphism corresponds to the natural homomorphism of $\CA$-modules
    \[
    R_{\CA}M \to \psi^* \psi_! R_{\CA} M.
    \]
    The module $R_{\CA}M$ is induced, so it is a twisted complex whose components are modules of the form $\varphi_! K$ for $\CB$-modules $K$. By the usual argument (using Lemmas \ref{lem:contractible cone}, \ref{lem:twisted complex}, and the fact that $\psi_!, \psi^*$ preserve direct sums) it is enough to show that the map  
    \[
    \varphi_! K \to \psi^* \psi_! \varphi_! K,
    \]
 is a homotopy equivalence
for each term of the twisted complex. In other words, that for each $\CB$-module $K$ the natural map
    \begin{equation}\label{eq:map for K}
    \CA\otimes_\CB K \to \CA'\otimes_\CB K
    \end{equation}
    is a weak equivalence. The operation $\otimes_\CB K$ is a functor from $\CB$-bimodules to $\CB$-modules, so it sends homotopy equivalences to homotopy equivalences. So the first adjunction homomorphism is an equivalence.

    The second homomorphism is the composition of homomorphisms
    \[
    \psi_! R_{\CA} \psi^* N \to \psi_! \psi^* N \to N,
    \]
    where the first arrow is induced by the counit and the second arrow is the usual adjunction. The functors $\psi_!, \psi^*$ are strict and commute with direct sums and we may assume that $R_{\CA}$ is a strict functor commuting with direct sums. So it is enough to prove the claim for $N=\varphi'_! K = \CA'\otimes_\CB K$. The maps become
    \[
    \CA'\otimes_\CA R_{\CA}(\CA'\otimes_\CB K) \to \CA'\otimes_\CA \CA'\otimes_\CB K \to \CA'\otimes_\CB K.
    \]
     Since \eqref{eq:map for K} is a weak equivalence and $\CA\otimes_\CB K$ is induced, the module $\CA\otimes_\CB K$ is a resolution for $\CA'\otimes_\CB K$. So there exists a strong homotopy equivalence $R_{\CA}(\CA'\otimes_\CB K) \to \CA\otimes_\CB K$ and the counit map factors through it. So it is sufficient to prove that the composition
     \[
    \CA'\otimes_\CA \CA\otimes_\CB K \to \CA'\otimes_\CA \CA'\otimes_\CB K \to \CA'\otimes_\CB K
     \]
     is a weak equivalence. Clearly, this map is an isomorphism.
\end{proof}
\begin{remark}
     If $\CA, \CA'$ are super-commutative then morphisms of $\CB$-bimodules are simply morphisms of $\CB$-modules, and so the assumptions of the Lemma are also necessary.
\end{remark}

\begin{definition}
    A homomorphism of algebras $\CA\to \CA'$ over $\CB$ which is a weak equivalence of $\CB$-bimodules (modules in the case when $\CA,\CA'$ are supercommutative) is called a weak equivalence of algebras.
\end{definition}

The following Lemma will be used to compare the results of pull-backs of modules via different homomorphisms:
\begin{lemma}\label{lem:algebra homomorphisms}
    Suppose $\CA, \CA'$ are dg algebras over a dg algebra $\CB$.
	Suppose $\psi_1,\psi_2:\CA\to \CA'$ are homomorphisms of algebras over $\CB$. Suppose there exists a weak equivalence $\gamma:\CA'\to \CA''$ to a dg algebra $\CA''$ such that $\gamma\circ \psi_1=\gamma\circ \psi_2$. Then for any $\CA'$-module $M$ the pullbacks $\psi_1^* M$ and $\psi_2^* M$ are canonically weakly equivalent.
\end{lemma}
\begin{proof}
    We can replace $M$ by $\gamma^* \gamma_! R_{\CA'} M$ by Lemma \ref{lem:algebra weak equivalence}. The modules $\psi_1^*  \gamma^* \gamma_! R_{\CA'} M$ and $\psi_2^* \gamma^* \gamma_! R_{\CA'} M$ are simply the same modules.
\end{proof}

\begin{example}
    The following is a simple example which illustrates the introduced notions. Consider
    \[
    \CA=\C[x,\xi\,|\,dx=0,\,d\xi=x],\qquad \CA'=\C
    \]
    where $x$ is even, $\xi$ is odd, $\psi:\CA\to\CA'$ sends $x, \xi$ to $0$. The map $\psi$ is a quasi-isomorphism, so if we take $\CB=\C$ then $\psi$ is a weak equivalence, and  $\psi^*, \psi_!R_{\CA}$ are equivalences of categories.

On the other hand, if $\CB=\C[x]$ then $\psi$ is not a weak equivalence, for instance because there is no non-zero map from $\CA'$ to $\CA$ over $\CB$.

 For an interesting $\CA$-module, take $M=\C[\xi]$, the variable $x$ acts by zero. In the case $\CB=\C$ this module is not induced. As a resolution we can take the complex
    \[
    R_{\CA} M = \left[\C[x,\xi]\xrightarrow{x} \underline{\C[x,\xi]}\right].
    \]
    Applying $\psi_!$ produces $\C\oplus\C[-1]$. Pulling back we obtain $\C\oplus\C[-1]$. So in the category $\CA/\C$ the object $M$ is isomorphic to its homology $\C\oplus\C[-1]$.

In the case $\CB=\C[x]$ the picture is very different. First, the object $M$ is already induced. Applying $\psi_!$ gives $\C$. Clearly, $M$ is not isomorphic to $\C$ since it is not even quasi-isomorphic to $\C$.
\end{example}

\subsection{Resolutions and Koszul equivalences}\label{sec:resolutions}
In this section we compute $\Hom_{\CA/\CB}(M,N)$ in special cases and identify it with $\Hom$ between certain deformations of $M,N$. Assume $\CA$ is a super-polynomial algebra 
over a super-commutative algebra $\CB$:
\[
\CA = \CB\left[u_1,\ldots,u_m\,|\, d u_i = c_i(u_1,\ldots,u_{m})\right],\qquad |u_i|=k_i.
\]
The degrees $k_i$ can be odd or even. Assume that both $\CA$ and $\CB$ are generated in non-negative degrees. In particular, we have $k_i\geq 0$. The coefficients of $c_i$ are in $\CB$. Set
\[
\widetilde\CA = \CA\left[\Delta_1,\ldots,\Delta_m,\theta_1,\ldots,\theta_m\,|\, d \theta_i = \Delta_i,\, d \Delta_i=0\right],\qquad |\Delta_i| = k_i,\;|\theta_i|=k_i+1.
\]
The map $\varepsilon:\widetilde\CA\to \CA$ defined by $\varepsilon(\Delta_i)=\varepsilon(\theta_i)=0$ is a homotopy equivalence of $\CA$-modules. Since the algebra $\widetilde\CA$ is free over $\theta_i,\Delta_i$ and $u_i$, we can define the corresponding partial derivatives, which satisfy
\begin{equation}
\label{eq: d partials}
 \left[\frac\partial{\partial \theta_i}, d\right]=0,\qquad \left[\frac\partial{\partial \Delta_i}, d\right]=\frac\partial{\partial \theta_i},\qquad
\left[\frac\partial{\partial u_i}, d\right]=\sum_{j} \frac{\partial c_j}{\partial u_i}\frac{\partial}{\partial u_j}.
\end{equation}
\begin{lemma}
We have
\[
\widetilde\CA = \CA\otimes_\CB \CA[\theta_1,\ldots,\theta_m\,|\, d \theta_i = \Delta_i],
\]
where $\Delta_i = u_i'-u_i-\cdots$ where $u_i,u'_i$ are the actions
of $u_i$ coming from the two different factors of $\CA\otimes_{\CB}\CA$ and
$\cdots$ belongs  to the ideal generated by $\theta_1,\ldots,\theta_m$.
\end{lemma}
\begin{proof}
 Set
\[
h = \sum_{i=1}^m \theta_i \frac{\partial}{\partial u_i}:\widetilde{\CA}\to\widetilde\CA,\qquad D := [h,d]=h d + d h:\widetilde{\CA}\to\widetilde\CA.
\]
Note that $h$ is homogeneous of degree 1.
We have $h^2=\frac12 [h,h]=0$ because the operators of the form $\theta_i$, $\frac{\partial}{\partial u_i}$ pairwise super-commute.

Therefore $D$ is a derivation of degree zero, it commutes with $d$ and $h$, and satisfies
\[
D(\theta_i)=D(\Delta_i)=0,\qquad D(u_i) = \Delta_i + \sum_{j=1}^{m} \theta_j \frac{\partial c_i}{\partial u_j}.
\]
Any infinite series in the variables $\theta_i$ which is homogeneous must terminate by degree reasons. In particular, $D$ is locally nilpotent, 
therefore $\exp(D) = \sum_{i=0}^\infty \frac{D^i}{i!}$ is a well-defined dg algebra automorphism of $\widetilde{\CA}$. Let
\[
u_i' = \exp(D) (u_i).
\]
Since $D$ commutes with $d$ we have
\[
d u_i' = c_i(u_1',\ldots,u_{i-1}').
\]
We have
\[
u_i' = u_i + \Delta_i + \cdots,
\]
where $\cdots$ belongs  to the ideal generated by $\theta_1,\ldots,\theta_m$. Since the degrees of $\theta_i$ are strictly positive, the variables $\Delta_j$ contained in $\cdots$ must satisfy $k_j<k_i$. Therefore the change of coordinates from $\Delta_i$ to $u_i'$ is invertible, so we can use $u_i'$ instead of $\Delta_i$ to freely generate $\widetilde\CA$.
\end{proof}
 Recall that $k_i\geq 0$. Then for any $\CA$-module $M$ the module
\begin{equation}\label{eq:M u presentation}
\widetilde M := \widetilde\CA\otimes_\CA M = M[u_1',\ldots,u_m',\theta_1,\ldots,\theta_m \,|\, d u_i' = c_i(u_1',\ldots,u_{i-1}'),\; d \theta_i = \Delta_i]
\end{equation}
on which $\CA$ acts via $u_i'$ is a bounded below twisted complex consisting of direct sums of copies of $\CA\otimes_\CB M$, and so is a resolution of $M$. Note that as a module over $\CB$ it also has a presentation
\begin{equation}\label{eq:M Delta presentation}
\widetilde M := \widetilde\CA\otimes_\CA M = M[\Delta_1,\ldots,\Delta_m,\theta_1,\ldots,\theta_m \,|\, d \Delta_i=0,\; d \theta_i = \Delta_i]
\end{equation}

\begin{example}
    In the simplest case,   $c_i$ are elements of $\CB$. In this case we have $D(u_i) = \Delta_i$, $u_i' = u_i + \Delta_i$  and the resolution \eqref{eq:M Delta presentation} looks like
    \[
    \widetilde M=M[u_1',\ldots,u_m',\theta_1,\ldots,\theta_m \,|\, d u_i' = c_i,\; d \theta_i = u_i'-u_i],
    \]
    where $\CA$ acts via $u_i'$.
\end{example}

\begin{example}
\label{ex: linear}
    Suppose more generally that $c_i$ depends linearly on the $u_j$:
$$
d u_i = c_i=B_i+\sum_j u_j A_{ij}.
$$
Here $B_i$ are homogeneous of degree $k_i-1$ and $A_{ij}$ are homogeneous of degree $k_i-k_j-1$.
Then we have
$D(u_i)=\Delta_i+\sum_j \theta_j A_{ij}$ and
    \begin{equation}\label{eq:linear Delta}
    \Delta_i = d \theta_i = u_i' - u_i - \sum_j \theta_j A_{ij}.
    \end{equation}
Also note that the equation $d^2=0$ implies
\begin{equation}
\label{eq: d square linear}
d(B_i)+ \sum_j B_j A_{ij}=0,\qquad (-1)^{k_\ell}d(A_{i\ell})+\sum_{j}A_{j\ell}A_{ij}=0.
\end{equation}
\end{example}

\begin{definition}
    We say $\CA$ is linear over $\CB$ if $\CA$ is a super-polynomial algebra over $\CB$ and for each generator $u_i$ the differential $d u_i$ depends at most linearly on the other generators, as in the Example \ref{ex: linear} above.
\end{definition}

Now let $M$ and $N$ be $\CA$-modules. Morphisms in the category $\CA/\CB$ are obtained from the complex $\Hom^\CA_{\bullet}(\widetilde M, N)$. To obtain an explicit presentation of the category we construct an explicit section $\Hom^\CA_{\bullet}(\widetilde M, N) \to \Hom^\CA_{\bullet}(\widetilde M, \widetilde N)$:
    \[
\begin{tikzcd}
    \wt M \arrow{rd}{f} \arrow[dashed]{r}{\wt f} & \wt N \arrow{d}{\varepsilon_N}\\
      & N
\end{tikzcd}
\]

 For any $\CA$-module $M$ we endow $\widetilde{M}$ with an action of $\CB$-linear operators $\frac{\partial}{\partial \theta_i}$ and $\frac{\partial}{\partial \Delta_i}$ using the presentation \eqref{eq:M Delta presentation}. 

\begin{definition}\label{def:flat morphism}
    For any $\CA$-modules $M$, $N$ a morphism $f\in\Hom^{\CA}_\bullet(\wt M, \wt N)$ is \emph{flat} if it commutes with $\frac{\partial}{\partial \theta_i}$ and $\frac{\partial}{\partial \Delta_i}$ ($i=1\ldots,m$).
\end{definition}

\begin{lemma}\label{lem:lifting f}
   Suppose $M$ and $N$ are $\CA$-modules and let $ f:\widetilde M\to N$ be any $\CB$-linear morphism.
\begin{itemize}
\item[(a)] There exists a unique flat $\CB$-linear morphism $\widetilde f: \widetilde M \to \widetilde N$ 
 whose composition with the counit map satisfies $\varepsilon_N \circ \widetilde f = f$.

\item[(b)] The lifting $f \to \widetilde f$ commutes with the differential.

\item[(c)] If $\CA$ is linear and $f$ is an $\CA$-linear map, then the lift $\widetilde f$ is also $\CA$-linear.
\end{itemize}
\end{lemma}

\begin{proof}
a) Suppose such $\tilde{f}$ is given. For any $x\in\widetilde M$ the element $\widetilde f(x)$ is a polynomial in the variables $\theta_i, \Delta_i$ with coefficients in $N$. The operation $\varepsilon_N$ extracts the constant term of a polynomial. The coefficients of $\widetilde f(x)$ can be extracted from the constant terms of iterated partial derivatives by the Taylor formula. So $\tilde f$ is unique. Explicitly, $\tilde f$ can be produced as follows:
\[
\tilde f(x) = \sum_{\substack{i_1,\ldots,i_m,\\j_1,\ldots,j_m}} \frac{\theta_1^{i_1} \cdots \theta_m^{i_m} \Delta_1^{j_1} \cdots \Delta_m^{j_m}}{i_1! \cdots i_m! j_1! \cdots j_m!} f\left(\dbyd{\Delta_m}\cdots\dbyd{\Delta_1} \dbyd{\theta_m}\cdots\dbyd{\theta_1} x\right)
\]

 b)  Let $f:\widetilde M\to N$ be a $\CB$-linear map.
Using the super-Jacobi identity, commutation relations \eqref{eq: d partials} and the fact that $\widetilde f$ commutes with $\frac\partial{\partial \theta_i}$, $\frac\partial{\partial \Delta_i}$, we obtain:
    \[
    \left[\frac\partial{\partial \theta_i}, [d, \widetilde f] \right] = \left[\left[\frac\partial{\partial \theta_i}, d\right], \widetilde f \right]=0,
    \quad
    \left[\frac\partial{\partial \Delta_i}, [d, \widetilde f] \right] = \left[\left[\frac\partial{\partial \Delta_i}, d\right], \widetilde f \right] =\left[\frac\partial{\partial \theta_i},\widetilde f \right]=0.
    \]
    Since $\varepsilon_N \circ [d,\widetilde f] = [d, f]$, we obtain that $\widetilde{[d,f]} = [d,\widetilde f]$ by the uniqueness.

c) Assume $\CA$ is linear. Then we have
    \[
    \left[\frac\partial{\partial \theta_i}, u_j'\right]=A_{ji}\in\CB,\qquad \left[\frac\partial{\partial \Delta_i}, u_j'\right]=\delta_{ij}\in\CB.
    \]
    This implies
    \[
    \left[\frac\partial{\partial \theta_i}, [u_j', \widetilde f] \right] = \left[\left[\frac\partial{\partial \theta_i}, u_j'\right], \widetilde f \right]=0,
    \]
    and similarly for $\frac\partial{\partial \Delta_i}$. If $f$ is $\CA$-linear, then
    \[
    \varepsilon_N \circ [u_j',\widetilde f] = u_j f - (-1)^{|u_j| |f|} f u_j' =0,
    \]
    so $[u_j',\widetilde f]$ is the lift of the zero map, and by the uniqueness must vanish.
\end{proof}

From now on we assume that $\CA$ is linear. Set
\begin{equation}
\widehat\theta_i = -\frac{\partial}{\partial \theta_i} + \sum_{j} A_{j i} \frac{\partial}{\partial \Delta_j}.
\end{equation}
We remind the reader that the operators $\frac{\partial}{\partial \theta_i}$, $\frac{\partial}{\partial \Delta_j}$ correspond to the presentation of $\CA$ as generated by $\theta_i, \Delta_j$.
\begin{lemma}
\label{lem: theta hat}
The operators $\widehat{\theta_i}$ satisfy
$$
\widehat{\theta_i}(u_j')=0,\ [d,\widehat{\theta_i}]=\sum_{j}  (-1)^{k_j+1}A_{ji} \widehat{\theta_j}.
$$
\end{lemma}
\begin{proof}
For the first equation, recall that by Example \ref{ex: linear} we have
$$
u_i'=u_i+\Delta_i+\sum_{j}  \theta_j A_{ij},
$$
hence
$$
\widehat{\theta_i}(u'_j)=-A_{ji}+A_{ji}=0.
$$
For the second equation, we note that both sides are derivations, so it is sufficient to verify it on generators $u_\ell',\theta_\ell$. Both sides clearly vanish on $u_\ell'$ by the first equation. We have
\[
[d,\wh\theta_i](\theta_\ell) = d \wh \theta_i(\theta_\ell) - (-1)^{k_i+1} \wh\theta_i d (\theta_\ell) = (-1)^{k_i} \wh\theta_i \Delta_\ell = (-1)^{k_i} A_{\ell i}.
\]
This clearly matches the right hand side applied to $\theta_\ell$.
%
%
%
\end{proof}


For a multi-index $I=(i_1,\ldots,i_m)$ denote
\[
\widehat\theta^I = \widehat\theta_1^{i_1} \cdots \widehat\theta_m^{i_m}.
\]

\begin{lemma}\label{lem:flat presentation}
    Let $M$ and $N$ be modules over a linear algebra $\CA$. Then flat morphisms in $\Hom^\CA_\bullet(\wt M, \wt N)$ are precisely morphisms which can be written as series
    \begin{equation}\label{eq:series f_I}
    \sum_I f_I \widehat\theta^I \qquad (f_I\in\Hom^\CB_\bullet(M,N)).
    \end{equation}
    Here each morphism $f_I\in\Hom^\CB_\bullet(M,N)$ is naturally extended to a morphism in $\Hom^\CA_\bullet(\widetilde M, \widetilde N)$ using presentation \eqref{eq:M u presentation}. Moreover, the presentation \eqref{eq:series f_I} is unique.
\end{lemma}
\begin{proof}
    In the case of a linear algebra $\CA$ we can replace  $\frac{\partial}{\partial \theta_i}$ by $\widehat\theta_i$ in Definition \ref{def:flat morphism}.

    Morphisms $f_I$ and $\widehat\theta_i$ are flat. So maps of the form \eqref{eq:series f_I} are flat. Conversely, suppose we have a flat map $f$. By subtracting maps of the form \eqref{eq:series f_I} we obtain a flat map $f'$ such that $\varepsilon\circ f'$ vanishes on $M[\theta_1,\ldots,\theta_m]$. By $u_i'$-linearity, $\varepsilon\circ f'$ must vanish identically. By Lemma \ref{lem:lifting f} we have $f'=0$.

    Uniqueness follows by applying the map to elements of $M[\theta_1,\ldots,\theta_m]$.
\end{proof}

Now we are ready to prove
\begin{theorem}\label{thm:koszul}
    Suppose $\CB$ is generated in non-negative degrees and $\CA=\CB[u_1,\ldots,u_m]$ with $|u_i|=k_i\geq 0$ is linear, i.e. $d u_i = B_i + \sum_j u_j A_{ij}$ where $B_i, A_{ij}\in \CB$. Then the category $\CA/\CB$ is equivalent to a full subcategory of the category of curved dg modules over
    \[
    \widehat\CA = \CB\left[\left[\widehat\theta_1,\ldots,\widehat\theta_m \,|\, d \widehat\theta_i = (-1)^{k_i+1} \sum_{j} A_{j i} \widehat\theta_j \right]\right],\qquad |\widehat\theta_i| = -1-k_i
    \]
    with curvature $\sum_{i=1}^m B_i \widehat\theta_i$.
    An $\CA$-module $M$ corresponds to a curved $\widehat\CA$-module $\widehat M = M\left[\left[\widehat\theta_1,\ldots,\widehat\theta_m\right]\right]$ with the differential deformed according to the following rule:
    \[
    d_{\widehat M} = d_{M\otimes_{\CB}\widehat\CA} + \sum_{i} u_i \widehat\theta_i.
    \]
\end{theorem}
\begin{proof}
    Morphisms from $M$ to $N$ in the category $\CA/\CB$ are the homology groups of the complex $\Hom^\CA_\bullet(\widetilde M, N)$, where $\widetilde M$ is the resolution given in \eqref{eq:M u presentation}. Elements of this complex are in bijection with flat morphisms in $\Hom^\CA_\bullet(\widetilde M, \widetilde N)$ by Lemma \ref{lem:lifting f}. By Lemma \ref{lem:flat presentation} they precisely correspond to series of the form \eqref{eq:series f_I}. Elements of the complex $\Hom^{\widehat\CA}_\bullet(\widehat M, \widehat N)$ also correspond to such series. So it remains to compare the differentials. We claim that the differentials agree, i.e. for any $f$ of the form \eqref{eq:series f_I} we have
    \begin{equation}\label{eq:d commutators}
    d_{\widetilde N} f - (-1)^{|f|} f d_{\widetilde M} = d_{\widehat N} f - (-1)^{|f|} f d_{\widehat M}.
\end{equation}

    The differential on $\widetilde M$ is given in terms of the differential on $M$ as follows (we use presentation \eqref{eq:M u presentation}):
    \[
    d_{\widetilde M} = d_M - \sum_i \left(u_i'-u_i - \sum_j \theta_j A_{i j}\right) \widehat\theta_i + \sum_{i} \left(B_i + \sum_j u_j' A_{ij}\right) \frac\partial{\partial u_i'},
    \]
    and similarly for $d_{\wt N}$.
    Flat maps commute with $u_i'$, $\widehat\theta_i$ and $\frac\partial{\partial u_i'}$ (the latter operator coincides with the operator $\frac\partial{\partial \Delta_i}$ when using the presentation \eqref{eq:M Delta presentation}.) Thus, when computing the left hand side of \eqref{eq:d commutators} we can replace $d_{\widetilde M}$ resp. $d_{\widetilde N}$ with
    \[
    (d_M\;\text{resp.}\; d_N) + \sum_{i,j} \theta_j A_{i j}\widehat\theta_i + \sum_{i} u_i \widehat\theta_i.
    \]
    Note that $|A_{ij}| = k_i-k_j-1$, in particular $A_{ii}=0$ for degree reasons. This implies $\theta_j A_{i j}\widehat\theta_i = A_{i j}\widehat\theta_i \theta_j$. We have
    \[
    [\theta_j, \widehat\theta_\ell] = (-1)^{k_j+1} \delta_{\ell j},
    \]
    and each $f_I$ in the expansion \eqref{eq:series f_I} commutes with $\theta_j$, so we can replace $\theta_j$ by the operator $(-1)^{k_j+1} \frac{\partial}{\partial \widehat\theta_j}$. We arrive at the following expression for the differential:
    \[
    (d_M\;\text{resp.}\; d_N) + \sum_{i,j}  (-1)^{k_j+1} A_{i j}\widehat\theta_i \frac{\partial}{\partial \widehat\theta_j}  + \sum_{i} u_i \widehat\theta_i.
    \]
    This is precisely the differential $d_{\widehat M}$ resp. $d_{\widehat N}$ so we have shown \eqref{eq:d commutators}.

    Notice that $d^2=0$ on $\widehat\CA$ follows from \eqref{eq: d square linear}. For the deformed differential on $\widehat M$ we have the curvature
    \[
    d_{\widehat M}^2 = \left[d_{M\otimes_{\CB}\widehat\CA}, \sum_{i=1}^m u_i \widehat\theta_i\right] = \sum_{i=1}^m \left(B_i + \sum_j u_j A_{ij}\right) \widehat\theta_i - \sum_{i=1}^m u_i A_{ji}\widehat\theta_j = \sum_{i=1}^m B_i \widehat\theta_i.
    \]
    By construction the ring $\widehat\CA$ is super-commutative, so the curvature is central.
\end{proof}
\begin{remark}
	In our main application for the ring $\CCA_{c,\infty}$ below the curvature will vanish, so we will obtain honest dg modules.
\end{remark}

\section{The dg algebra $\CA$}
\label{sec:algebra}

\subsection{Algebras}
Let $R=\C[x_1,\ldots,x_n]$ and $R^e = \C[x_1,\ldots,x_n,x_1',\ldots,x_n']$.  We identify $R^e$ with $R\otimes_\C R$ by sending $x_i\mapsto x_i\otimes 1$ and $x_i'\mapsto 1\otimes x_i$.  Note that $R^e$-mod and $(R,R)$-bimod are equivalent as categories, but not as monoidal categories (the monoidal structures are given by $\otimes_{R^e}$ and $\otimes_R$ respectively).

Below we will consider various objects which are simultaneously algebra objects in $R^e$-mod and coalgebra objects in $(R,R)$-bimod.

First, let $S_n$ act on $R$ by permuting variables, and define $B:=R\otimes_{R^{S_n}} R$, where $R^{S_n}$ is the algebra of symmetric polynomials.  We can and will regard $B$ as the following quotient of $R^e$:
\[
B=\frac{\C[x_1,\ldots,x_n,x_1',\ldots,x_n']}{\left(\sum_{i=1}^n x_i^k = \sum_{i=1}^n x_i'^k \quad (k=1,\ldots,n)\right)}.
\]
In this presentation, the power sum functions can be replaced by any other set of generators of $R^{S_n}$. 

This quotient makes it obvious that $B$ is an $R^e$-algebra.  In addition to this structure, $B$ is a coalgebra object in the category of $(R,R)$-bimodules.  The counit is defined by
\[
B\rightarrow R,\qquad\quad x_i\mapsto x_i, \qquad x_i'\mapsto x_i,
\]
and the comultiplication is defined by 
\[
B\rightarrow B\otimes_R B,\qquad x_i\mapsto x_i\otimes 1,\qquad x_i'\mapsto 1\otimes x_i'.
\]
Note that $B$ and $B\otimes_R B$ are $R^e$ algebras, and the counit and comultiplication are algebras maps.

\begin{remark}
 If $M$ is a $B$-module, then we may regard $M$ as an $R^e$-module (equivalently, an $(R,R)$-bimodule) by restriction.  Thus if $M,N$ are $B$-modules, then $M\otimes_R N$ is defined.  This tensor product inherits the structure of a $B$-module via the coproduct $B\rightarrow B\otimes_R B$.

Indeed, the category of $B$-modules coincides with the full subcategory of $(R,R)$-bimodules for which the left and right actions of symmetric polynomials coincide; this full subcategory is obviously closed under $\otimes_R$.
\end{remark}
\begin{remark}
If $M$, $N$ are $B$-modules, we will often write $M\otimes_R N$ simply as $MN$, for simplicity.
\end{remark}

\begin{remark}
The algebras $B$, $R$ are graded by placing the generators $x_i$, $x_i'$ in degree $2$. This grading is called the $q$-grading. By placing $B$ or $R$ in homological degree $0$ they are viewed as differential graded algebras with zero differential.
\end{remark}

Next we will define an $R^e$-algebra $\CA$ which is also a coalgebra object in the category of $(R,R)$-bimodules, similarly to $B$ (with the exception that coassociativity for $\CA$ holds only up to homotopy).  Consequently the category of $\CA$-modules up to weak equivalence inherits the structure of a monoidal category.


\begin{definition}\label{def:A}
Let $\CA$ be the free differential graded commutative algebra over $B$ with generators $\xi_i$ ($i=1,\ldots,n$), $u_k$ ($k=1,\ldots,n$) of degrees $\deg_h \xi_i=1$, $\deg_h u_k=2$. The differential (of degree $-1$) is given by
	\[
	d \xi_i = x_i - x_i',\qquad d u_k = \sum_{i=1}^n \frac{x_i^k-x_i'^k}{x_i-x_i'} \xi_i=\sum_{i=1}^n h_{k-1}\left(x_i,x'_i\right)\xi_i,
	\]
where $h_{k}$ is the complete homogeneous symmetric polynomial of degree $k$.
	The $q$-grading extends to $\CA$ by $\deg_q \xi_i = 2$ and $\deg_q u_k = 2k$. Clearly, the differential preserves the $q$-grading and the condition $d^2=0$ is satisfied.
\end{definition}

Note that $B$ is a dg subalgebra of $\CA$, by construction.  Moreover, $\CA$ is free as a $B$-module and supported in non-negative homological degrees.


\subsection{Counit}

The commutative dg algebra $\CA$ is a resolution of $R$ over $B$ because of the following.
\begin{proposition}\label{prop:resolution}
\label{prop: A resolution}
Let $\e\colon \CA\rightarrow R$ be the algebra map sending $x_i,x_i'\mapsto x_i$ and $\xi_i,u_i\mapsto 0$.  Then $\e$ is a quasi-isomorphism of $B$-modules.
\end{proposition}
\begin{proof}
	The idea is to compare $\CA$ to the Koszul resolution of $R$ over the algebra $R \otimes R=\C[x_1,\ldots,x_n,x_1',\ldots,x_n']$. This resolution is described by
	\[
	\CA_1 := \C[(x_i),(x_i'),(\xi_i)\;|\; d \xi_i = x_i-x_i',\; d x_i = d x_i' = 0].
	\]
	We are going to construct a commutative diagram of algebra homomorphisms three of which are quasi-isomorphisms. Therefore, the remaining one will be a quasi-isomorphism:
	\[
	\begin{tikzcd}
	\CA_2=\C[(x_i),(x_i'),(\xi_i),(u_k),(v_k)] \arrow{r}{v_k\to 0} \arrow{d}{v_k'\to 0, u_k\to 0} &  \CA=B[(\xi_i),(u_k)] \arrow{d}{\xi_i\to 0, u_k\to 0, x_i'\to x_i}\\
	\CA_1=\C[(x_i),(x_i'),(\xi_i)] \arrow{r}{\xi_i\to 0,x'_i\to x_i} & R=\C[(x_i)]
	\end{tikzcd}
	\]
The algebra $\CA$ is not free over $R^e$, it is only free over $B$. The free resolution of $B$ over $R^e$ has the form:
	\begin{equation}\label{eq:quasiis B}
	\C\left[(x_i), (x_i'), (v_k) \;|\; d v_k = \sum_i x_i^k - \sum_i x_i'^k,\; d x_k=d x_k' = 0\right] \quasiis B.
	\end{equation}
	We replace each copy of $B$ in $\CA$ by this resolution and notice that the differential can be extended keeping the property $d^2=0$ as follows:
	\[
	\CA_2:=\C\Big[(x_i), (x_i'), (\xi_i), (u_k), (v_k)\;\Big|\; d x_i=d x_i'=0,\; d \xi_i = x_i-x_i',\;
         \]
\[
 d v_k = \sum_i x_i^k - \sum_i x_i'^k,\;
	 d u_k = -v_k + \sum_{i=1}^n \frac{x_i^k-x_i'^k}{x_i-x_i'} \xi_i\Big].
	\]
	The counit homomorphism $\CA_2\to \CA$ is given by sending $v_k$ to $0$ and is a quasi-isomorphism by Lemmas \ref{lem:contractible cone} and \ref{lem:twisted complex}: the complexes $\CA_2$ and $\CA$ are filtered by degree in $u_k$ first, and then by degree in $\xi_i$, and the graded pieces for this filtration are copies of \eqref{eq:quasiis B}. We perform the following change of variables in $\CA_2$:
	\[
	v_k' = -v_k + \sum_{i=1}^n \frac{x_i^k-x_i'^k}{x_i-x_i'} \xi_i.
	\]
	So $\CA_2$ is isomorphic to the following algebra:
	\[
	\CA_2\cong \C[(x_i), (x_i'), (\xi_i), (u_k), (v_k')\;|\; d x_i=d x_i'=0,\; d \xi_i = x_i-x_i',\; d v_k' = 0, d u_k = v_k']
	\]
	\[
	= \CA_1 \otimes \C[(u_k), (v_k')\;|\; d v_k' = 0, d u_k = v_k'],
	\]
	which is quasi-isomorphic to $\CA_1$ because the algebra $\C[(u_k), (v_k')\;|\; d v_k' = 0, d u_k = v_k']$ is quasi-isomorphic to $\C$.
\end{proof}

\begin{proposition}\label{prop:counit hop}
The morphism $\e\otimes \id_{\CA} - \id_{\CA}\otimes \e \colon \CA\otimes_R \CA\rightarrow \CA$ is null-homotopic.
\end{proposition}
\begin{proof}
The tensor product $\CA\otimes_R\CA$ is a complex (bounded below) of projective $B$ modules of the form $R\otimes_{R^{S_n}} R\otimes_{R^{S_n}} R$.  Therefore, the functor $\Hom_{\bullet}^{B}(\CA\otimes_R \CA, - )$ sends acyclic complexes to contractible complexes, and sends quasi-isomorphisms to homotopy equivalences.  In particular post-composing with the quasi-isomorphism $\e$ defines a homotopy equivalence
\[
\Hom_{\bullet}^{B}\left(\CA\otimes_R \CA,\CA\right) \buildrel\simeq\over \rightarrow \Hom_{\bullet}^{B}\left(\CA\otimes_R \CA,R\right) .
\]
On the other hand, post-composing with $\e$ annihilates $\e\otimes \id_{\CA} - \id_{\CA}\otimes \e$, hence this morphism must be null-homotopic.
\end{proof}

\subsection{Coproduct}
\label{sec:coproduct}

We wish to define the coproduct on $\CA$ as an algebra map $\CA\rightarrow \CA\otimes_R \CA$.  To compactify some of the formulas, we will regard $B\otimes_R B$ as the quotient
\[
\C[\xx,\xx',\xx'']\twoheadrightarrow B\otimes_R B,\qquad x_i\mapsto x_i\otimes 1,\qquad x_i'\mapsto x_i'\otimes 1 = 1\otimes x_i,\qquad x_i'' \mapsto 1\otimes x_i'.
\]
Note that $\CA\otimes_R \CA$ is naturally a $B\otimes_R B$ algebra, and can be viewed as a $B$-algebra via the coproduct $B\rightarrow B\otimes_R B$.

\begin{definition}
Let $\Delta\colon \CA \rightarrow \CA\otimes_R \CA$ be the $B$-algebra map defined by
\[
\Delta(x_i)=x_i\otimes 1,\ \Delta(x'_i)=1\otimes x_i',\ \Delta(\xi_i)=\xi_i\otimes 1+1\otimes \xi_i,
\]
\[
\Delta(u_1)=1\otimes u_1+u_1\otimes 1
\]
\[
\Delta(u_k)=u_k\otimes 1+1\otimes u_k + \sum_{i=1}^{n} h_{k-2}\left(x_i,x'_i,x''_i\right) \xi_i\otimes \xi_i\ (k\ge 2).
\]
\end{definition}

We call the map $\Delta$ the coproduct on $\CA$ (over $R$). As we will see below, it is a chain map which is coassociative up to homotopy. It is easy to see that the map $\varepsilon: \CA\to R$ is a counit for this coproduct.

\begin{lemma}
	\label{lem:coproduct chain}
	The coproduct $\Delta$ is a chain map.
\end{lemma}

\begin{proof}
	Let us check that $\Delta$ commutes with the differential. Indeed,
	$$
	d(\Delta(\xi_i))=d(\xi_i\otimes 1+1\otimes \xi_i)=(x_i-x'_i)+(x'_i-x''_i)=x_i-x''_i=\Delta(d(\xi_i))
	$$
	while
	$$
	d(\Delta(u_k))=d\left(u_k\otimes 1+1\otimes u_k + \sum_i h_{k-2}\left(x_i,x'_i,x''_i\right) \xi_i\otimes \xi_i\right)=
	$$
	\begin{multline}
		\label{eq: coproduct}
		\sum_i [h_{k-1}\left(x_i,x'_i\right)\xi_i\otimes 1 +h_{k-1}\left(x'_i,x''_i\right)1\otimes \xi_i+\\ h_{k-2}\left(x_i,x'_i,x''_i\right)\left((x_i-x_i')1\otimes \xi_i+(x''_i-x'_i)\xi_i\otimes 1 \right)]
	\end{multline}
	Note that
	$$
	h_{k-1}\left(x_i,x''_i\right)=h_{k-1}\left(x_i,x'_i\right)+h_{k-2}\left(x_i,x'_i,x''_i\right)(x''_i-x'_i)=$$
$$
h_{k-1}\left(x'_i,x''_i\right)+h_{k-2}\left(x_i,x'_i,x''_i\right)(x_i-x'_i),
	$$
	so \eqref{eq: coproduct} equals $\sum_{i} h_{k-1}\left(x_i,x''_i\right)\left(\xi_i\otimes 1+1\otimes \xi_i\right)=\Delta(d(u_k))$.
\end{proof}

\begin{example}
We have $\Delta(u_2)=u_2\otimes 1+1\otimes u_2+\sum_{i=1}^{n}\xi_i\otimes \xi_i.$
\end{example}

Recall Proposition \ref{prop:counit hop} which states that $\e\otimes \id$ and $\id\otimes \e$ are homotopic as morphisms $\CA\otimes \CA\rightarrow \CA$.

\begin{corollary}\label{cor:counit equivalence}
The maps $\Id\otimes \e\simeq \e\otimes \Id$ and $\Delta$ are inverse homotopy equivalences of $(R,R)$-bimodules $\CA\simeq \CA\otimes_R \CA$.
\end{corollary}
\begin{proof}
It is easy to verify the identity
\begin{equation}
\label{eq: epsilon delta}
(\Id\otimes \varepsilon)\circ \Delta = (\varepsilon\otimes \Id)\circ \Delta = \Id.
\end{equation}

In the other direction, we compute
\[
\Delta\circ (\Id\otimes \e)  =  (\Id\otimes \Id\otimes \e) \circ (\Delta\otimes \Id) \simeq   (\Id\otimes \e\otimes \Id) \circ (\Delta\otimes \Id) = \Id\otimes \Id.\qedhere
\]
\end{proof}

\begin{proposition}\label{prop:coassociativity}
The comultiplication on $\CA$ is coassociative up to homotopy.
\end{proposition}
\begin{proof}
The complex $\CA$ is a bounded below complex of free $B$-modules, and $B$ is free as a right $R$-module so if $\phi:M\rightarrow N$ is a quasi-isomorphism of left $R$-modules, then $\Id_{\CA}\otimes \phi:\CA\otimes_R M\rightarrow \CA\otimes_R N$ is a quasi-isomorphism.  From this it follows that $\e\otimes \cdots \otimes \e$ defines a quasi-isomorphism $\CA\otimes_R \cdots \otimes_R \CA\rightarrow R$.  

Also since $\CA$ is free over $B$ and bounded above, the functor $\Hom^{B}_\bullet(\CA,-)$ sends the quasi-isomorphism $\e\otimes \cdots\otimes \e$ to a homotopy equivalence
\[
\Hom^{B}_\bullet\left(\CA,\CA\otimes_R \cdots \otimes_R \CA\right)\buildrel \simeq\over \rightarrow \Hom^{R^e}_\bullet(\CA,R),
\]
Finally, $(\Delta\otimes\Id)\circ \Delta - (\Id\otimes \Delta)\circ \Delta$ is annihilated by $\e\otimes \e\otimes \e$, hence is null-homotopic.
\end{proof}

\begin{remark}
In Appendix \ref{sec: higher} we explicitly write the homotopy which realizes the coassociativity of $\Delta$, and its higher analogues.
\end{remark}

\subsection{Twists of $\CA$}

\begin{definition}\label{def:std bimodule}
Each permutation $w\in S_n$ determines a \emph{standard bimodule} $R_w$, which is defined to be the quotient of $R^e$ by the ideal generated by elements $x_{w(i)}-x_i'$ for all $i$. 
\end{definition}

Note that $R_w$ is the quotient of $R^e$ by a two-sided ideal, hence is an algebra object in $R^e$-mod. Thus $R_w\otimes_R A\otimes_R R_v$ is an algebra object in $R^e$-mod whenever $A$ is.  Such algebras are called \emph{twists of $A$}.  The following says that all twists of $B$ are isomorphic.

\begin{lemma}\label{lemma:std times B}
We have $R_w\otimes_R B\cong B\cong B\otimes_R R_w$ for all $w\in S_n$.
\end{lemma}
\begin{proof}
We may identify $R_w\otimes_R B$ with the quotient of $\C\left[\xx,\xx',\xx''\right]$ by the ideal generated by $x_{w(i)}-x_i'$ and $f(\xx') - f(\xx'')$ for all symmetric polynomials $f$.  It is easy to check that the mapping
\[
f\left(\xx,\xx',\xx''\right)\rightarrow f\left(\xx,w\inv(\xx'),\xx''\right)
\]
defines an isomorphism of bimodules $R_w\otimes_R B \rightarrow R\otimes_R B\cong B$.
\end{proof}

Next we consider twists of $\CA$.  Twisting preserves the subalgebra $B\subset \CA$ by the above lemma, but acts non-trivially on the differentials of $\xi_i$ and $u_k$.

\begin{definition}[Twists of $\CA$]\label{def:twists of A}
Let $R_w\CA R_v\cong R_w\otimes_R \CA\otimes_R R_v$ be the free differential graded commutative algebra over $B$ with generators $\xi_i$ ($i=1,\ldots,n$), $u_k$ ($k=1,\ldots,n$) of degrees $\deg_h \xi_i=1$, $\deg_h u_k=2$, with differential (of degree $-1$) given by
\[
d \xi_i = x_{w(i)} - x_{v\inv(i)}',\qquad d u_k = \sum_{i=1}^n h_{k-1}\left(x_{w(i)},x'_{v\inv(i)}\right)\xi_i,
\]
where $h_{k}$ is the complete homogeneous symmetric polynomial of degree $k$. 
\end{definition}

\begin{lemma}
We have $R_w\CA\cong  \CA R_w$ as $B$-algebras, via the map sending $\xi_i\mapsto \xi_{w(i)}$ and $u_k\mapsto u_k$ for all $k$.
\end{lemma}
\begin{proof}
We need only verify that the $B$-algebra map $R_w\CA\rightarrow \CA R_w$ sending $\xi_i\mapsto \xi_{w(i)}$ and $u_k\mapsto u_k$ is a chain map.  This is follows from the fact that, inside $R_w\CA$ we have
\[
d \xi_i = x_{w(i)}-x_i',\qquad d u_k = \sum_i h_{k-1}\left(x_{w(i)},x_i'\right) \xi_i,
\]
while inside $\CA R_w$ we have
\[
d \xi_{w(i)} = x_{w(i)}-x_{i}',\qquad  d u_k = \sum_i h_{k-1}\left(x_i,x_{w\inv(i)}'\right) \xi_i = \sum_i h_{k-1}\left(x_{w(i)},x_i'\right) \xi_{w(i)}.\qedhere
\]
\end{proof}

\begin{definition}\label{def:Aw}
Henceforth, we will denote $\CA_w:= \CA R_w$.  An $\CA_w$-module will also be referred to as a \emph{$w$-twisted $\CA$-module}.
\end{definition}

For each pair of permutations $w,v$, the coproduct of $\CA$ induces a $B$-algebra map $\CA_{vw}\rightarrow \CA_v\otimes_R \CA_w$ of the form
\[
\Delta_{v,w}:\CA R_{vw} \rightarrow \CA \CA R_vR_w \cong \CA R_v \CA R_w,
\]
which we refer to as the \emph{twisted coproduct}.  These maps satisfy the appropriate notion of coassociativity up to homotopy.  Explicitly, the twisted coproduct $\Delta_{v,w}$ satisfies:
\begin{equation}\label{eq:Delta vw}
\Delta_{v,w}\ : \ \begin{cases}
\xi_i\mapsto \xi_i\otimes 1+ 1\otimes \xi_{v\inv(i)}\\
u_k\mapsto u_k\otimes 1 + 1\otimes u_k + \sum_i h_{k-2}\left(x_i,x_{v\inv(i)}',x_{w\inv v\inv (i)}''\right) \xi_i\otimes \xi_{v\inv(i)}.
\end{cases}
\end{equation}
In addition to the twisted coproduct there is also the \emph{twisted counit} $\e_v : \CA_v\cong \CA\otimes_R R_v\rightarrow R_v$, which satisfies:
\begin{equation}\label{eq:epsilon v}
\e_v(x_i) = x_i,\qquad \e_v(x_i')=x_{v(i)},\qquad \e_v(\xi_i)=0,\qquad \qquad \e_v(u_k)=0
\end{equation}
for all $1\leq i,k\leq n$.

\subsection{The category of modules}

We will apply the constructions of Section \ref{sec:homotopy theory} to the case $\CB=B\subset \CA_w$. Thus we have a category of dg modules over $\CA_w$ where the morphisms are viewed up to homotopy, and we localize by the class of morphisms which become homotopy equivalences when viewed as morphisms of complexes of $B$-modules. The resulting category is denoted $\CA_w/B-\Mod$ or simply $\CA_w/B$ and there is a forgetful functor from this category to the homotopy category of complexes of $B$-modules. Objects isomorphic in this category are called weakly equivalent (relative to $B$).

Given a $\CA_w$-module $M$ and a $\CA_v$-module $N$, the tensor product $M\otimes_R N$ is naturally an $\CA_w\otimes_R\CA_v$-module; pulling back along the $B$-algebra map $\CA_{wv}\rightarrow \CA_w\otimes_R \CA_v$ allows us to regard $M\otimes_R N$ as a $\CA_{wv}$-module.

Consider three $\CA$-modules $M, N, K$. Applying Lemma \ref{lem:algebra homomorphisms} for $\psi_1=(\Delta \otimes \Id)\circ\Delta$, $\psi_2=(\Id \otimes \Delta)\circ\Delta$ and $\gamma=(\varepsilon\otimes\varepsilon\otimes 1)$ we obtain
\begin{proposition}\label{prop:associativity on modules}
The isomorphism of $B$-modules $$(M\otimes_R N)\otimes_R K\cong M\otimes_R (N\otimes_R K)$$ induces a weak equivalence of $\CA$-modules, i.e.~ an isomorphism in $\CA/B$.  Similar remarks apply if, instead, $M,N,K$ are modules over twisted algebras $\CA_u,\CA_v,\CA_w$.
\end{proposition}

\section{Modules over $\CA$}
\label{sec:modules}

\subsection{Soergel bimodules}

Define the $R$-$R$ bimodules
$$
B_i=R\otimes_{R^{s_i}} R=\frac{\C[x_1,\ldots,x_n,x'_1,\ldots,x'_n]}{x_i+x_{i+1}=x'_i+x'_{i+1},\ x_ix_{i+1}=x'_ix'_{i+1},\ x_j=x'_j\ (j\neq i,i+1)}
$$
where $s_i=(i, i+1)$ is the simple reflection. The category of Soergel bimodules $\SBim_n$ is defined as a smallest full subcategory of the category of $R$-$R$ bimodules containing $R$ and $B_i$ and closed under tensor products, direct sums and direct summands. This is an additive tensor category, although it is not abelian. We denote by $\CK(\SBim_n)$ the homotopy category of bounded above complexes in
$\SBim_n$.

An important result of Soergel \cite{Soergel} states that the indecomposable objects $B_w$ in  $\SBim_n$ are in bijection with the permutations $w\in S_n$. The bimodule $B$ from previous section in fact coincides with $B_{w_0}$ corresponding to the longest element $w_0$ in $S_n$. We note that the action of $R\otimes R$ on any Soergel bimodule factors through $B$ under the natural projection $R\otimes R\to B$.

\subsection{Rouquier complexes and Khovanov-Rozansky homology}
\label{sec:rouquier}

Let $b_i:B_i\rightarrow R$ be the $R^e$-linear map sending $1\mapsto 1$, and let $b_i^\ast:R\rightarrow B_i(2)$ be the bimodule map sending $1\mapsto x_i-x_{i+1}'$.

In \cite{Rouquier} Rouquier defined the following two-term complexes of Soergel bimodules:
$$
T_i=\left[\underline{B_i}(1)\xrightarrow{b_i} R(1)\right], \qquad\quad T_i^{-1}=\left[R(-1)\xrightarrow{b^*_i} \underline{B_i}(1)\right]
$$
and proved that $T_i,T_i^{-1}$ satisfy braid relations up to homotopy:
$$
T_i\otimes_R T_i^{-1}\simeq T_i^{-1}\otimes_R T_i\simeq R,\qquad
$$
$$
T_i\otimes_R T_{i+1}\otimes_R T_i\simeq T_{i+1}\otimes_R T_{i}\otimes_R T_{i+1},\ T_i\otimes_R T_j\simeq T_j\otimes_R T_i\ (|i-j|>1).
$$
To a braid $\beta$ one can associate  a complex of Soergel bimodules $T_{\beta}$ (the product of $T_i$ and $T_i^{-1}$ corresponding to crossings in $\beta$) which is well defined up to homotopy equivalence.

Given a Soergel bimodule $M$, one can define its Hochschild cohomology $\HH^i(M)=\Ext^i_{R\otimes R}(R,M)$. Given a complex of Soergel bimodules
$$
\ldots \rightarrow M_k\rightarrow M_{k-1}\rightarrow M_{k-2}\rightarrow\ldots
$$
one can associate the complex of Hochschild cohomologies
$$
\ldots \rightarrow \HH^{i}(M_k)\rightarrow \HH^i(M_{k-1})\rightarrow \HH^i(M_{k-2})\rightarrow\ldots,
$$
its homology is denoted by $\HHH(M)$. The Khovanov-Rozansky homology \cite{Kh, KR2} of the braid $\beta$ is defined as $\HHH(T_{\beta})$,
it is a topological invariant of the closure of $\beta$. Khovanov-Rozansky homology is triply graded: in addition to homological grading $k$ and Hochschild grading $i$, there is a quantum grading $j$ induced by grading on Soergel bimodules $B_i$. We assume that all variables $x_i$ have quantum grading 2. There is also an overall grading shift, see Theorem \ref{thm:link invariant}.

For higher Hochschild degrees, we normalize the quantum grading such that the minimal quantum grading in $\HH^i(R)$ is zero for all $i$ (see also Proposition \ref{prop:markov 1}). 
With this normalization, the homology of the $n$-component unlink is given by
\begin{equation}
\label{eq: hhh unlink}
\HHH(R)=\C[x_1,\ldots,x_n,\theta_1,\ldots,\theta_n].
\end{equation}
Here $x_i$ are even variables of quantum grading $2$ and Hochschild grading 0, and $\theta_i$ are odd variables of quantum grading 0 and Hochschild grading 1.

\subsection{Action of $\CA$}

We will construct an action of the dg algebra $\CA$ and its twists on various complexes.

\begin{lemma}
	\label{lem:action ti}
	The $B$-module structure on the Rouquier complexes $T_i$, $T_i^{-1}$ lifts to an action of $\CA_s$, where $s=(i, i+1)$, and $u_k$ acts by 0 for all $k$.
\end{lemma}

\begin{proof}
We define the action of $\xi_i$ by ``dot-sliding homotopies" as in \cite{GH}.
	Recall that $T_i=[\underline{B_i}\xrightarrow{b_i} R]$ and $\xi_i$ acts as $b_i^*:R\to B_i$ while $\xi_{i+1}$ acts by $-b_i^*$ and $\xi_j$ acts by 0 for $j\neq i,i+1$. Indeed,
$$
b_ib^*_i=b^*_ib_i=x_i-x'_{i+1}=-(x_{i+1}-x'_i).
$$
It is clear that the actions of all generators commute.	Now
	$$
	\sum _{j=1}^{n} h_{k-1}\left(x_j,x'_{w^{-1}(j)}\right)\xi_j=\left(h_{k-1}\left(x_i,x'_{i+1}\right)-h_{k-1}\left(x_{i+1},x'_{i}\right)\right)b_i.
	$$
	On $R$ we have $x_i=x'_i$ and $x_{i+1}=x'_{i+1}$, so
$$
h_{k-1}\left(x_i,x'_{i+1}\right)-h_{k-1}\left(x_{i+1},x'_{i}\right)=h_{k-1}\left(x_i,x_{i+1}\right)-h_{k-1}\left(x_{i+1},x_{i}\right)=0.
$$

The proof for $T_i^{-1}$ is similar:  $\xi_i$ acts as $b_i:B_i\to R$ while $\xi_{i+1}$ acts by $-b_i$ and $\xi_j$ acts by 0 for $j\neq i,i+1$.
 Therefore we obtain an action of $\CA_s$.
\end{proof}

\begin{lemma}
	\label{lem:unique}
	Suppose that $X$ and $Y$ are invertible complexes of Soergel bimodules admitting an action of $\CA_w$. Any homotopy equivalence $X\to Y$ of $B$-modules lifts to a unique isomorphism in $\CA_w/B$.
\end{lemma}
\begin{proof}
	The statement is analogous to \cite[Proposition 2.20]{GH}. Let $\phi:X\to Y$ be a homotopy equivalence. Since $X$ and $Y$ are homotopy equivalent and invertible, we have quasi-isomorphisms $\Hom^B_\bullet(X,Y)\cong \Hom^B_\bullet(X,X) \cong R$, so $\HH^0_{\CA_w}(\Hom_B(X, Y))$, which classifies morphisms from $\widetilde X$ to $Y$ by Proposition \ref{prop:localization hochschild}, equals
	\[
	\HH^0_{\CA_w}(R) = \Hom_{\CA_w\otimes_B \CA_w}(\CA_w, R) = R,
	\]
	since $\CA$ is concentrated in non-negative degrees. Thus the map
    \[
    \Hom_{\CA_w/B}(X, Y)\to \Hom_{B-\Mod}(X, Y)
    \]
    is an isomorphism.
\end{proof}

\begin{theorem}
\label{thm:braid action}
The $B$-module structure on any Rouquier complex $T_\beta$ lifts to an action of $\CA_w$, where $w$ is the permutation represented by $\beta$. This action is invariant under braid relations (up to a 
weak equivalence).
\end{theorem}

\begin{proof}
Suppose we are given a braid $\beta = \sigma_{i_1}^{\e_1}\cdots\sigma_{i_r}^{\e_r}$, presented as a product of Artin generators, where $1\leq i_j\leq n-1$ and $\e_j\in \{-1,1\}$, and let 
$T_{\beta} = T_{i_1}^{\e_1}\otimes_R \cdots  \otimes_R T_{i_r}^{\e_r}$ be the associated Rouquier complex.  Lemma \ref{lem:action ti} constructs 
an action of $\CA_{w_j}$ on each $T_{i_j}$, where $w_j = (i_j\ i_j+1)$.  The coproduct gives an algebra map
\[
\CA_w\rightarrow \CA_{w_1}\otimes_R \cdots \otimes_R \CA_{w_r},
\]
canonical up to homotopy.  Pulling back along this algebra map gives an action of $\CA_w$ on $T_\beta$, where $w=w_1\cdots w_r$.  The complex $T_\beta$ is invertible, so by Lemma \ref{lem:unique}  the action of
	$\CA_w$ is unique up to a weak equivalence and invariant under homotopy equivalences (in particular, braid relations).
\end{proof}

\begin{remark}
\label{rem: u1}
The action of $u_1$ vanishes on any Rouquier complex. Indeed, it vanishes on $T_i^{\pm}$, and $\Delta(u_1)=u_1\otimes 1+1\otimes u_1$.
\end{remark}

\begin{example}
	Consider the full twist on two strands $\FT_2=T_1^2$. We have
$$
\Delta(u_2)=u_2\otimes 1+1\otimes u_2+\sum_i  \xi_i\otimes \xi_i,
$$
the actions of $u_2\otimes 1$ and $1\otimes u_2$ vanish, so $u_2$ acts on the complex $\FT_2$ by
	$$
	u_2= \xi_1\otimes \xi_2+ \xi_2\otimes \xi_1.
	$$
	If we write $\FT_2=\left[\underline{B}\to B\to R\right]$, one can check that $\xi_1\otimes \xi_2=\xi_2\otimes \xi_1$ acts as the map $b^*$ from $R$ to the leftmost $B$, and $u_2$ acts by $2b^*$ (see equation \eqref{eq: intro full twist}). Note that $u_1$ acts by 0 by Remark \ref{rem: u1}.
\end{example}

\begin{definition}
An $\CA_w$-module will be called \emph{elementary} if $u_k$ acts by zero for all $k=1,\ldots,n$.
\end{definition}
For example, $T_i^{\pm}$ are elementary.  The tensor product of elementary modules is no longer elementary; for instance we have the following:
\begin{lemma}
\label{lem: u2}
Let $w_1,\ldots,w_r$ be permutations and suppose $X_1,\ldots,X_r$ are elementary modules over $\CA_{w_1},\ldots,\CA_{w_r}$.  The action of $u_2$ on $X_1\otimes_R X_2\cdots \otimes_R X_r$ is given by
	\begin{equation}
		\label{eq: u2}
		u_2=\sum_{1\leq j\leq n} \sum_{k<l} \xi_j^{(k)}\circ \xi_{(w_k\cdots w_{l-1})\inv(j)}^{(l)},
	\end{equation}
	where $\xi_i^{(k)}$ denotes the action of $\xi_i$ on the $k$-th tensor factor of $X_1\otimes_R\cdots\otimes_R X_r$.
\end{lemma}

\begin{proof}
It is easy to see that the $r$-th iterated coproduct\footnote{This iterated coproduct is not well-defined since coassociativity holds only up to homotopy.  Nonetheless, any two choices for the $r$-th iterated coproduct act in the same way on $u_2$.} on $\CA$ satisfies
\[
\Delta^{(r)}(u_2) = \sum_{1\leq k\leq r}  1^{\otimes k-1}\otimes u_2\otimes 1^{\otimes r-k} + \sum_{1\leq i\leq n}\sum_{1\leq k<l\leq r} 1^{\otimes k-1}\otimes \xi_i\otimes 1^{\otimes l-k-1} \otimes \xi_i\otimes 1^{\otimes r-l}
\]

For any permutation $w$, the isomorphism $\CA R_w\buildrel\cong\over\rightarrow R_w\CA$ sends $\xi_i\mapsto \xi_{w\inv(i)}$.  The composition
\[
\CA R_{w_1\cdots w_r} \buildrel \Delta^{(r)}\over\longrightarrow \CA\otimes_R \cdots\otimes_R \CA R_{w_1\cdots w_r} \cong \CA R_{w_1}\otimes_R \cdots\otimes_R \CA R_{w_r}
\]
sends
\begin{align*}
u_2
&\mapsto \sum_{1\leq k\leq r}  1^{\otimes k-1}\otimes u_2\otimes 1^{\otimes r-k}\\
&  + \sum_{1\leq i\leq n}\sum_{1\leq k<l\leq r} 1^{\otimes k-1}\otimes \xi_{(w_1\cdots w_{k-1})\inv(i)}\otimes 1^{\otimes l-k-1} \otimes \xi_{(w_1\cdots w_{l-1})\inv(i)}\otimes 1^{\otimes r-l},
\end{align*}
since the standard bimodule $R_{w_1\cdots w_{k-1}}$ must migrate past the $k$-th tensor factor.

Now, when acting on the tensor product of elementary modules as in the statement the only surviving summands are those of the form $\xi_{(w_1\cdots w_{k-1})\inv(i)}^{(k)}\circ  \xi_{(w_1\cdots w_{l-1})\inv(i)}^{(l)}$ with $1\leq i\leq n$ and $1\leq k<l\leq r$.  Letting $j=(w_1\cdots w_{k-1})\inv(i)$ proves the formula in the statement.
\end{proof}

\subsection{Koszul complexes}

Recall that inside $S_n$ we have the \emph{simple reflections} (or \emph{simple transpositions}) $s_i:=(i,i+1)$ which swap indices $i$ and $i+1$.  Any permutation which is conjugate to a simple reflection is called a \emph{reflection}.  The reflections in $S_n$ are just the permutations of the form $(i, j)$ which swap indices $i$ and $j$, leaving all other indices fixed.  


\begin{definition}\label{def:Kij}
For each reflection $r=(i,j)$ we let $K_r$ (also denoted $K_{ij}$, or $K_{i,j}$) denote the Koszul complex 
\[
K_{ij} = [R(-1) \xrightarrow{x_i-x_j} \underline{R}(1)].
\]
This can be conveniently described as a dg algebra by
\begin{equation}\label{eq:K as dga}
K_{ij} = R[\eta\;|\;d\eta=x_i-x_j](1),
\end{equation}
where $\deg_h\eta=1$.
\end{definition}
It is clear that $K_{ij}\cong K_{ji}$, so the definition of $K_r$ doesn't depend on the ordering of indices in the expression $r=(i,j)$.

\begin{proposition}
The Koszul complex $K_r$ admits a structure of an $r$-twisted $\CA$-module.
\end{proposition}
\begin{proof}
We let $r=(i,j)$.  Observe that $x_k-x_{r(k)}'$ is zero unless $k\in\{i,j\}$.  Thus, we may take $\xi_k=0$ for $k\not\in\{i,j\}$.  We let $\xi_i$ act by the multiplication by $\eta$ in the expression \eqref{eq:K as dga}, or diagrammatically as the morphism
\[
\begin{tikzpicture}[baseline=-3em]
\tikzstyle{every node}=[font=\small]
\node (a) at (0,0) {$K_r$};
\node at (.8,0) {$=$};
\node at (1.6,0) {$\Big($};
\node (b) at (2.2,0) {$R(-1)$};
\node (c) at (4.8,0) {$\underline{R}(1)$};
\node at (5.3,0) {$\Big)$};
\node (d) at (0,-2) {$K_r$};
\node at (.8,-2) {$=$};
\node at (1.6,-2) {$\Big($};
\node (e) at (2.2,-2) {$R(-1)$};
\node (f) at (4.8,-2) {$\underline{R}(1)$};
\node at (5.3,-2) {$\Big)$};
\path[->,>=stealth,shorten >=1pt,auto,node distance=1.8cm]
(b) edge node {$x_i - x_j$} (c)
(e) edge node {$x_i - x_j$} (f)
(c) edge node[above] {$1$} (e)
(a) edge node {$\xi_i$} (d);
\end{tikzpicture},
\]
and we let $\xi_j$ act by $-\xi_i$.  The action of $u_1,\ldots,u_n$ must be zero for degree reasons.  In order for this to define a valid action of $\CA_r$, we must verify that following elements act by zero on $K_r$:
\[
d(u_l) = \sum_k h_{l-1}\left(x_k,x_{r(k)}'\right) \xi_k.
\]
Indeed, $x_k-x_k'$ acts by zero on $K_r$ for all $k$, and $\xi_k$ acts by zero unless $k=i,j$ and so the above becomes $\left(h_{l-1}\left(x_i,x_j\right) - h_{l-1}\left(x_j,x_i\right)\right)\xi_i$ which indeed is zero.
\end{proof}

We expect the following to be true, but do not need it or prove it here.

\begin{conjecture}\label{conj:Kw central}
Let $X$ be an $\CA_w$ module for $w\in S_n$.  Then for any reflection $r$ we have a weak equivalence of $\CA_{wr}$-modules
\begin{equation}\label{eq:K_r commutes}
X\otimes_R K_r \cong K_{wrw\inv}\otimes_R  X.
\end{equation}
\end{conjecture}
We will prove this in the special case when $X$ is $T_i^\pm$ or $K_v$, below.

\begin{proposition}\label{prop:relations}
We have weak equivalences of twisted $\CA$-modules:
\begin{enumerate}
\item $T_i\otimes_R K_r \simeq K_{srs}\otimes_R T_i$.
\item $T_i\inv\otimes_R K_r \simeq K_{srs}\otimes_R T_i\inv$.
\item $K_v\otimes_R K_r \simeq K_{vrv}\otimes_R K_v$.
\end{enumerate}
where $s=(i,i+1)$ and $v,r\in S_n$ are reflections.
\end{proposition}
In fact, the proof will show that (2) and (3) are honest isomorphisms.
\begin{proof}
The proof (2) and (3) amounts to the following case-by-case analysis:
	\begin{itemize}
		\item [a)] $T_i^{-1}K_{j k}= K_{j k}  T_i^{-1} \qquad j,k\notin\{i,i+1\},$
		\item [b)] $T_i^{-1}  K_{ik}= K_{i+1\; k}  T_i^{-1} \qquad k\notin\{i,i+1\},$
		\item [c)] $T_i^{-1} K_{i+1\;k}= K_{i k}  T_i^{-1} \qquad k\notin\{i,i+1\},$
		\item [d)] $T_i^{-1}  K_{i\; i+1}= K_{i\; i+1}  T_i^{-1},$
		\item [e)] $K_{ij} K_{kl} = K_{kl} K_{ij} \qquad i,j,k,l\;\text{distinct},$
		\item [f)] $K_{ij} K_{jk} = K_{ik} K_{ij} = K_{jk} K_{ik} \qquad \qquad i,j,k\;\text{distinct}.$
	\end{itemize}
We will use equation \eqref{eq:Delta vw} for the twisted coproduct. The case a) is straightforward. We continue to the case b). It is convenient to view $T_i^{-1}$, $K_{ij}$ as dg algebras.  So $T_i^{-1}$ is presented by
	\[
	T_i = B_i[\xi_i\;|\;d \xi_i = x_i-x_{i+1}']/\left(\xi_i(x_i-x_i')\right).
	\]
	Thus $T_i^{-1}   K_{ik}$ can be presented by
	\[
	T_i^{-1}   K_{ik} = B_i[\xi_i,\xi_k \;|\;d \xi_k = x_k-x_i',\; d \xi_i = x_i-x_{i+1}']/\left(\xi_i(x_{i}-x_{i}')\right).
	\]
	The only remaining non-zero generators of $\CA$ act by
	\[
	\xi_{i+1}=-\xi_i-\xi_k, \quad u_m = h_{m-2}\left(x_{i+1},x_i', x_k\right)\xi_i \xi_k.
	\]
	Similarly, we compute
	\[
	K_{i+1\; k}   T_i^{-1} = B_i[\xi_i, \xi_{i+1}\;|\; d \xi_i = x_i-x_{i+1}',\; d \xi_{i+1} = x_{i+1}-x_k]/\left(\xi_i(x_i-x_i')\right).
	\]
	The $\CA$-action is given by
	\[
	\xi_k=-\xi_i-\xi_{i+1},\quad u_m = h_{m-2}\left(x_k, x_{i+1}, x_i'\right) \xi_{i+1} \xi_{i}.
	\]
	We observe that the modules and the actions coincide.

	The case c) is analogous.

	Let us describe the case d).
	We have
	\[
	T_i^{-1}   K_{i\; i+1} = B_i[\eta_1, \eta_2\;|\;d \eta_1 = x_i-x_{i+1}',\; d \eta_2 = x_i'-x_{i+1}']/\left(\eta_1(x_i-x_i')\right).
	\]
	The action of $\CA$ is given by $\xi_i = -\xi_{i+1}= \eta_1 - \eta_2$,
	\[
	u_k = h_{k-2}\left(x_i, x_{i+1}', x_i'\right) \eta_1 (-\eta_2) + h_{k-2}\left(x_{i+1}, x_i', x_{i+1}'\right) (-\eta_1)\eta_2
	\]
	\[
	= \eta_1 (\eta_1-\eta_2)\left(h_{k-2}\left(x_i, x_i, x_{i+1}\right) + h_{k-2}\left(x_i, x_{i+1}, x_{i+1}\right)\right).
	\]
	Similarly, we have
	\[
	K_{i\; i+1}   T_i^{-1} = B_i[\eta_1, \eta_2\;|\;d \eta_1 = x_i-x_{i+1},\; d \eta_2 = x_{i}-x_{i+1}']/\left(\eta_2(x_i-x_i')\right),
	\]
	using the identities of $B_i$.
	The action of $\CA$ is given by $\xi_i = -\xi_{i+1}= \eta_1 - \eta_2$,
	\[
	u_k = h_{k-2}\left(x_i, x_{i+1}, x_i'\right) \eta_1 (-\eta_2) + h_{k-2}\left(x_{i+1}, x_i, x_{i+1}'\right) (-\eta_1)\eta_2
	\]
	\[
	=\eta_2(\eta_1-\eta_2)\left(h_{k-2}\left(x_i, x_{i+1},x_i\right) + h_{k-2}\left(x_{i+1},x_i,x_{i+1}\right)\right),
	\]
	using the identities $\eta_2 x_i = \eta_2 x_i'$ and $\eta_2 x_{i+1} = \eta_2 x_{i+1}'$. The required isomorphism is given by the substitution $(\eta_1,\eta_2)\to (\eta_2, 2\eta_2-\eta_1)$.

	The case e) is straightforward. The case f) is somewhat similar to b), but simpler.
	
This completes the proof of statements (2) and (3) from the statement.  We deduce (1) from (2) by multiplying with $T_i$ on the left and right.  The resulting calculation involves coassociativity of the tensor product of $\CA$-modules, hence a priori is only a weak equivalence.
\end{proof}

%

%

\begin{lemma}\label{lemma:w and refl}
Given a permutation $w\in S_n$, suppose $w=r_1\cdots r_l$ expresses $w$ as a minimal length product of reflections. Then the twisted $\CA$-module
\[
K_w = K_{r_1}\cdots K_{r_l}
\]
depends only on $w$, and not the choice of reflections $r_1,\ldots,r_l$, up to weak equivalence.
\end{lemma}
\begin{proof}
Consider the action of $w$ on $\{1,\ldots,n\}$.  We decompose into orbits:
\[
\{1,\ldots,n\} = O_1\sqcup \cdots \sqcup O_c
\]
Let $w= r_1\cdots r_l$ be an minimal expression of $w$ as product of reflections.  Then each reflection $r$ appearing in this expression must be of the form $r = (i,j)$ with $i,j$ both living in the same orbit $O_m$.   Note that if $r=(i,j)$ and $r'=(i',j')$ with $i,j\in O_m$ and $i',j'\in O_{m'}$ with $m\neq m'$, then $r$ and $r'$ commute.  Thus, by rearranging we may assume that  our expression for $w$ respects the cycle decomposition in the following sense:
\[
w = \left(r_1^{(1)}\cdots r_{l_1}^{(1)}\right) \cdots \left(r_1^{(c)}\cdots r_{l_c}^{(c)}\right)
\] 
where the parenthesized expressions are disjoint cycles in $S_n$.

Without loss of generality we may as well assume that $c=1$, so that $w$ is an $n$-cycle.  Let $I$ denote the collection of tuples $(r_1,\ldots,r_{n-1})$ such that $w = r_1\cdots r_{n-1}$ is a minimal length expression of $w$ as a product of reflections.   The braid group $\Br_n$ acts on $I$ according to
$$
\sigma_i \ :\ (r_1,\ldots, r_{n-1}) \ \mapsto\ \left( r_1,\cdots ,r_{i-1}, r_i', r_{i+1}',r_{i+2},\ldots, r_{n-1}\right).
$$
where $r_i' = r_i r_{i+1} r_i$ and $r_{i+1}' = r_i$.  Here $\sigma_i\in \Br_n$ is the standard Artin generator.  If $(r_1,\ldots,r_l)$ and $(r_1',\ldots,r_l')$ are related by the braid group action then the associated Koszul complexes $K$ and $K'$ are weakly equivalent since $K_{r_i} K_{r_{i+1}}=K_{r_i r_{i+1} r_i} K_{r_i}$.   By \cite[Proposition 1.6.1]{Bessis}, the braid group acts transitively on $I$, which completes the proof.
\end{proof}

\begin{lemma}\label{lemma:K squared central}
We have weak equivalences
\[
K_v  K_r^2 \simeq K_r^2  K_v \simeq K_v K_{vrv}^2\simeq  K_{vrv}^2 K_v
\]
for all reflections $r,v\in S_n$.
\end{lemma}
\begin{proof}
Since $rvr$ and $vrv$ are both reflections, by Lemma \ref{lemma:w and refl} we get
\[
K_{v}K_{r}\simeq K_{vrv}K_v\simeq K_rK_{rvr}.
\]
Indeed, for $v=r$ this is trivial and for $v\neq r$ we get minimal length factorizations
$$
v\cdot r=(vrv)\cdot v=r\cdot (rvr)
$$
as a product of reflections. Now we compute
\[
K_v K_r^2   \simeq K_r  K_{rvr} K_r \simeq  K_r^2K_v
\]
and
\[
K_v K_r^2  \simeq  K_{vrv} K_v K_r \simeq  K_{vrv}^2 K_v.
\]
Finally, by the first equation we get 
\[
 K_{vrv}^2 K_v\simeq K_v K_{vrv}^2.\qedhere
\]
\end{proof}

Below, we will use the notion of \emph{reflection length} of a permutation $w$, which is defined to be the length of a minimal product of reflections yielding $w$.  We will denote the reflection length by $\ell_{\mathrm{refl}}(w)$.  If $r=(i,j)$ then $\ell_{\mathrm{refl}}(wr) = \ell_{\mathrm{refl}}(w)+1$ if $i,j$ belong to different cycles of $w$; otherwise if $i,j$ belong to the same cycle of $w$ then $\ell_{\mathrm{refl}}(wr) = \ell_{\mathrm{refl}}(w)-1$.
\begin{lemma}\label{lemma:refl length and K}
If $\ell_{\mathrm{refl}}(wr) = \ell_{\mathrm{refl}}(w)+1$ then $K_w K_r = K_{wr}$.  Otherwise $\ell_{\mathrm{refl}}(wr) = \ell_{\mathrm{refl}}(w)-1$, and $K_w K_r\simeq K_{wr} K_r^2$.
\end{lemma}
\begin{proof}
Let $w=u_1\cdots u_l$ be a minimal product of reflections yielding $w$.  If $u_1\cdots u_l r$ is minimal, then $K_{wr} \simeq K_w K_r$ by construction.  Suppose on the other hand that $u_1\cdots u_l r$ is non-minimal.  Then $\ell_{\mathrm{refl}}(w)=\ell_{\mathrm{refl}}(wr)+1$, hence $K_w \simeq K_{wr}K_r$ by the first statement of this lemma, hence
\[
K_w K_r \simeq K_{wr} K_r K_r ,
\]
which is the second statement.
\end{proof}

\begin{lemma}
\label{lem: K standard}
Up to weak equivalence relative to $\CB$, any product $K$ of $K_{ij}$ can be written in the form
\begin{equation}\label{eq:decomposition K}
K\simeq K_{w} K_{i_1,j_1}^2 \cdots  K_{i_s,j_s}^2,
\end{equation}
where $w$ is the permutation obtained by multiplying transpositions in $K$. Moreover, $K$ is completely determined up to weak equivalence by the following data:
\begin{itemize}
    \item the permutation $w$, which is the product of transpositions in $K$;
    \item an equivalence relation $\sim$ on the set $\{1,\ldots,n\}$, defined to be the minimal equivalence relation such that $i\sim w(i)$ for all $i$ and $i_t\sim j_t$ for all $1\le t\le s$;
    \item for each equivalence class $S$ for $\sim$, the number of $K_{i,j}$ in $K$ such that $i,j\in S$.
\end{itemize}
\end{lemma}
\begin{proof}
We prove \eqref{eq:decomposition K} by induction on the number of factors in $K$.  The base case $K=K_{(i,j)}$ is trivial.  Suppose now we have a complex of the form
\[
K_{w} K_{i_1,j_1}^2 \cdots  K_{i_s,j_s}^2,
\]
and tensor on the right with $K_{k,l}$.   If $k$ and $l$ belong to different cycles of $w$, then  then $w'=w\cdot(k,l)$ has $\ell_{\mathrm{refl}}(w') = \ell_{\mathrm{refl}}(w)+1$, hence $K_w K_{k,l} \simeq K_{w\cdot (k,l)}$, hence
\[
\left(K_{w} K_{i_1,j_1}^2 \cdots  K_{i_s,j_s}^2\right)K_{k,l} \simeq K_{w\cdot(k,l)} K_{i_1,j_1}^2 \cdots  K_{i_s,j_s}^2.
\]

On the other hand, if $k,l$ belong to the same cycle of $w$, then $w\cdot(k,l)$ has one more cycle than $w$, and $\ell_{\mathrm{refl}}(w\cdot(k,l))=\ell_{\mathrm{refl}}(w)-1$, and we find that
\[
\left(K_{w} K_{i_1,j_1}^2 \cdots  K_{i_s,j_s}^2\right)K_{k,l} \simeq K_{w\cdot(k,l)} K_{k,l}^2 K_{i_1,j_1}^2 \cdots  K_{i_s,j_s}^2.
\]
This proves existence of the decomposition \eqref{eq:decomposition K}.

To prove the second statement, we choose a representative for $K$ up to weak equivalence. Let $J\subset\{1,\dots,n\}$ be a set containing exactly one element out of each cycle of $w$ and let $I\subset J$ be a set containing exactly one element of each equivalence class with respect to $\sim$. For any $j\in J \setminus I$ there is a unique $i(j)\in I$ such that $i(j)\sim j$. Note that the relations $j\sim i(j)$ and the cycles of $w$ already generate the equivalence relation $\sim$. Let us prove that we can rewrite the product \eqref{eq:decomposition K} such that it contains the factors $K_{j,i(j)}^2$ for all $j\in J \setminus I$.

Indeed, Lemma \ref{lemma:K squared central} tells us that 
\[
K_{ab}^2  K_{a c} \simeq K_{a c}K_{ab}^2\simeq K_{c b}^2  K_{a c}\simeq K_{a c}K_{c b}^2.
\]
Therefore if $K$ is a product \eqref{eq:decomposition K} which contains a factor $K_{a c}$ then $$K K_{ab}^2\simeq K K_{c b}^2.$$ 

Since $j\sim i(j)$, we can find a chain of pairwise distinct elements $j_0=j,j_1,\ldots,j_{r}=i(j)\in J$ and a set of pairs $(k_1,l_1),\ldots, (k_{r},l_{r})$ appearing in \eqref{eq:decomposition K}  such that $k_{t+1}$ and $l_{t}$ are in the same cycle of $w$ as $j_t$ for all $t$. Then
$$
K_wK_{k_1,l_1}^2\cdots K_{k_{r},l_{r}}^2\simeq K_wK_{j,j_1}^2\cdots K_{j_{r-1},i(j)}^2\simeq K_wK_{j,i(j)}^2\cdots K_{j_{r-1},i(j)}^2.
$$
All remaining factors $K_{k,l}^2$ can be replaced by $K_{k',l'}^2$ for $k'\sim k\sim l\sim l'$  and therefore their number in each class is a complete invariant.

\end{proof}


\begin{corollary}
\label{cor: T and K standard}
Given a product $C$ of complexes $T_i,T_i^{-1}$ and $K_{jk}$, let $u\in S_n$ be the product of the corresponding transpositions. Then $C$ can be written in the form
$$
C\simeq T_{\beta} K_w  K_{i_1, j_1}^2  \cdots K_{i_s,j_s}^2
$$
where $\beta$ is the braid obtained by erasing all $K_{ij}$ from the product, and $w$ is determined by $u$ and $\beta$. This isomorphism agrees with  the
$\CA$-module structures.
\end{corollary}

\begin{proof}
By Proposition \ref{prop:relations} we can move all Rouquier complexes $T_i$ to the left without changing their indices or the order, so we get $C=T_{\beta}  K$ where $K$ is some product of $K_{ij}$. By Lemma \ref{lem: K standard} we can write
 $K=K_w\otimes K_{i_1, j_1}^2  \cdots K_{i_s,j_s}^2$ and the result follows.
\end{proof}

\section{Markov moves}\label{sec:markov}
\subsection{Cyclic property}




Recall that any $\CA_w$-module $X$ is a complex of  $(R,R)$ bimodules, and one consider its Hochschild homology 
 $\HH^k(X)$ component-wise.
Since $\HH$ is a functor, 
$\HH^k(X)$ again has the structure of an $\CA_w$ module, with the property that the difference $x_i'-x_i$ acts by zero for each $i$.

\begin{definition}
If $M$ is a $B$-module, then we let $\HH_0(M)=M\otimes_B R$.  Equivalently, $M/IM$ where $I\subset B$ is the ideal generated by $x_i'-x_{i}$ for $i=1,\ldots,n$.  Let $\CCA_w:=\HH_0(\CA_w)$.
\end{definition}

Modules over $\CCA_w$ are simply modules over $\CA_w$ on which the actions of $x_i'$ and $x_i$ coincide for all $i$.  In particular we have the following.
\begin{proposition}
    For any $\CA_w$-module $X$ the Hochschild cohomology $\HH^k(X)$ is a $\CCA_w$-module.\qed
\end{proposition}

It is clear that $\CCA_w-\Mod$ is a full subcategory of $\CA_w-\Mod$.  Furthermore the functor $\HH_0$ defines a left inverse to the inclusion $\CCA_w-\Mod\rightarrow \CA_w-\Mod$.  This implies that the inclusion functor $\CCA_w/R = \CCA_w/B \to \CA_w/B$ is fully faithful.  The tensor product of a $\CCA_{w_1}$-module and a $\CCA_{w_2}$-module is a $\CCA_{w_1 w_2}$-module. So the following is evident:
\begin{proposition}\label{prop:projection}
    For each $w\in S_n$, the category $\CCA_w/R$ is a full subcategory of $\CA_w/B$. The collection of categories $\CCA_w/R$ is closed under the tensor product, and we have a projection formula:
    \[
    \HH^k(X\otimes_R Y) = \HH^k(X) \otimes_R Y \qquad(X\in\CA_{w_1}-\Mod,\;Y\in\CCA_{w_2}-\Mod)
    \]
    provided $Y$ is free as an $R$-module.
\end{proposition}

\begin{remark}
We observe that $\HH_0$ is a functor from $B$-modules from $R$-modules.  This functor comes equipped with natural maps
\[
\HH_0(M\otimes_R N)\rightarrow \HH_0(M)\otimes_R \HH_0(N)
\]
satisfying an appropriate version of the coassociative property\footnote{In literature, such functors are sometimes called \emph{oplax monoidal}.}.  Combining with the usual coproduct for algebras $\CA_w$ gives an algebra map
\[
\HH_0(\CA_{vw})\rightarrow \HH_0(\CA_v\otimes_R \CA_w) \rightarrow \HH_0(\CA_v)\otimes_R \HH_0(\CA_w).
\]
This defines the \emph{coproduct} $\CCA_{vw}\rightarrow \CCA_v\otimes_R \CCA_w$.  This in turn allows us to define the structure of a $\CCA_{vw}$-module on the tensor product $M\otimes_R N$ where $M$ is a $\CCA_v$-module and $N$ is a $\CCA_w$-module.  This tensor product coincides with the tensor product of twisted $\CA$-modules as in Proposition \ref{prop:projection}.
\end{remark}

\begin{definition}\label{def:cycling}
For each pair of permutations, let $\tau_{v,w}: \CCA_{vw}\rightarrow \CCA_{wv}$ be the $B$-algebra isomorphism sending $x_i\mapsto x_{v\inv(i)}$, $\xi_i\mapsto \xi_{v\inv(i)}$, and $u_k\mapsto u_k$.
\end{definition}

\begin{proposition}\label{prop:cyclicity}
    Let $X$ and $Y$ be $\CA_{v}$ and $\CA_{w}$-modules respectively, for $v,w\in  S_n$. Then
 \begin{equation}\label{eq:HH cyclic}
 \HH^k(X\otimes_R Y) \cong \tau_{v,w}^\ast\left(\HH^k(Y\otimes_R X)\right).
 \end{equation}
inside $\CCA_{vw}/\C$. 
\end{proposition}
\begin{proof}
On the level of complexes of vector spaces, the isomorphism \eqref{eq:HH cyclic} is well-known.   This isomorphism endows $\HH^k(Y\otimes_R X)$ with the structure of an $\HH_0(\CA_{vw})$-module by pulling back the obvious $\HH_0(\CA_w\otimes_R \CA_v)$-action along the algebra map $\sigma\circ [\Delta_{v,w}]$, where $[\Delta_{v,w}]:=\HH_0(\Delta_{v,w})$ is $\HH_0$ applied to the twisted coproduct \eqref{eq:Delta vw} and
\[
\sigma_{v,w} : \HH_0(\CA_{v}\otimes_R \CA_{w})\rightarrow  \HH_0(\CA_{w}\otimes_R \CA_{v})
\]
is inherited from the isomorphism which swaps the order of tensor factors in $\HH_0(-\otimes_R -)$.   Note that $\sigma$ is $\C$-linear but not $R$-linear, since the left (=right) $R$-action on $\HH_0(X\otimes_R Y)$ corresponds to ``middle'' $R$-action on $\HH_0(Y\otimes_R X)$. 

Thus, we have two $\CA_{vw}$-actions on $\HH^k(Y\otimes_R X)$, given by pulling back along $\sigma_{v,w}\circ \Delta_{v,w}$ and $\Delta_{w,v}\circ \tau_{v,w}$, respectively.  The former is isomorphic to $\HH_0(X\otimes_R Y)$ with its $\CA_{vw}$-action, by construction.  To show that these two $\CA_{vw}$-module structures are weakly equivalent, we show that the algebra maps $\sigma_{v,w}\circ \Delta_{v,w}$ and $\Delta_{w,v}\circ \tau_{v,w}$  become equal after post-composing with an appropriate quasi-isomorphism, and use Lemma \ref{lem:algebra homomorphisms}.

We first examine the action of $\sigma_{v,w}\circ \Delta_{v,w}$ and $\Delta_{w,v}\circ \tau_{v,w}$ on generators.  We have
\[
\sigma_{v,w}\circ \Delta_{v,w} \ :\ \begin{cases}
x_i\mapsto  x_i' \\
\xi_i\mapsto \xi_{v\inv(i)}\otimes 1 + 1\otimes \xi_i\\
u_k\mapsto u_k\otimes 1 + 1\otimes u_k +\sum_i h_{k-2}\left(x_i', x_{v\inv (i)}, x_{w\inv v\inv(i)}'\right) \xi_{v\inv(i)} \otimes \xi_i.
\end{cases}
\]
\[
\Delta_{w,v}\circ \tau_{v,w} \ :\ \begin{cases}
x_i\mapsto x_{v\inv(i)}\\
\xi_i\mapsto \xi_{v\inv(i)}\otimes 1 + 1\otimes \xi_{w\inv v\inv(i)}\\
u_k\mapsto u_k\otimes 1 + 1\otimes u_k +\sum_j h_{k-2}\left(x_{w(j)}, x_{j}', x_{v\inv (j)}\right) \xi_{w(j)} \otimes \xi_{j}.
\end{cases}
\]
In the first line, we have let $x_i' = x_i'\otimes 1 = 1\otimes x_i$ in $\HH_0(\CA_w\otimes_R \CA_v)$.  Now apply the twisted counit $\e_v$ from \eqref{eq:epsilon v} on the right.  From the formulas it is clear that
\[
 (\Id\otimes \e_v) \circ \sigma_{v,w}\circ \Delta_{v,w} = (\Id\otimes \e_v) \circ \Delta_{w,v}\circ \tau_{v,w}.
\]
On the other hand $\Id\otimes \e_v$ is a quasi-isomorphism because $\e_v$ is. This completes the proof.
\end{proof}

\begin{remark}
    In general, we do not expect the isomorphism to exist in $\CCA_w/R$. The homomorphism $\sigma$ is not $R$-linear, so the proof for $R$ instead of $\C$ fails. In our situation, the modules will turn out to be free over $R$ and we will show that the isomorphism can be upgraded to an isomorphism in $\CCA_w/R$.
\end{remark}


\subsection{Stabilization}\label{ssec:stabilization}
In this section we will be dealing with operations that change the number of strands, so we will use $n$ as an index in $\CA_n$, $\CCA_n$ to specify the number of strands. Given an $\CA_n$-module $M$ we construct an $\CA_{n+1}$ module $M\sqcup \one_1$ as follows. As a $B$-module this is simply $M\otimes \C[x_{n+1}]$. Construct a homomorphism of dg algebras
\[
\varphi:\CA_{n+1} \to \CA_n \otimes \C[x_{n+1}]
\]
as follows. The generators $x_i$, $x_i'$, $\xi_i$, $u_k$ for $i\leq n$, $k\leq n$ are sent to the corresponding generators of $\CA_n$. The generator $\xi_{n+1}$ is sent to $0$, the generators $x_{n+1}, x_{n+1}'$ both go to $x_{n+1}$. Finally, for the remaining generator $u_{n+1}$ we choose any homogeneous element $u$ of $\CA_{n}$ satisfying
\[
d u = \sum_{i=1}^n h_n(x_i, x_i') \xi_i
\]
and set $\varphi(u_{n+1})=u$. Such $u$ exists by Proposition \ref{prop:resolution} and Lemma \ref{lem:corrections} explains how to construct it explicitly. Any two choices of $u$ differ by a boundary. The $\CA_{n+1}$-module structure on $M\sqcup \one_1$ is obtained with the help of $\varphi$:
\[
M\sqcup \one_1 = \varphi^* (M\otimes \C[x_{n+1}]).
\]
The following says that the resulting functor $M\mapsto M\sqcup \one_1$ from $\CA_n/B_n$ to $\CA_{n+1}/B_{n+1}$ is monoidal.
\begin{proposition}\label{prop:adding a thread}
    For any $\CA_{n}$-modules $M$, $N$ the isomorphism
    \[
    (M\otimes_{R_n} N)\sqcup \one_1 \cong (M\sqcup \one_1) \otimes_{R_{n+1}} (N\sqcup \one_1)
    \]
    of complexes of $B$-modules lifts to an isomorphism in the category $\CA_{n+1}/B$.
\end{proposition}
\begin{proof}
Similarly to the proofs of Propositions \ref{prop:associativity on modules} and \ref{prop:cyclicity} we use Lemma \ref{lem:algebra homomorphisms} to compare the pullbacks via the compositions of the maps in the diagram
\[
\begin{tikzcd}
    \CA_{n+1} \arrow{d}{\Delta}\arrow{r}{\varphi} & \CA_{n}\otimes \C[x_{n+1}] \arrow{d}{\Delta\otimes\Id_{\C[x_{n+1}]}}\\
    \CA_{n+1}\otimes_{R_{n+1}} \CA_{n+1} \arrow{r}{\varphi\otimes\varphi} & (\CA_n\otimes \C[x_{n+1}]) \otimes_{R_{n+1}}(\CA_n\otimes \C[x_{n+1}]) \cong \left(\CA_{n}\otimes_{R_n} \CA_n\right) \otimes \C[x_{n+1}]\\
\end{tikzcd}
\]
We compose the two maps with the homotopy equivalence 
$$\varepsilon_n \otimes \Id\otimes \Id:\left(\CA_{n}\otimes_{R_n} \CA_n\right) \otimes \C[x_{n+1}] \to \CA_n \otimes \C[x_{n+1}].$$
 Using the identity $(\varepsilon_n\otimes\Id_{\C[x_{n+1}]})\circ \varphi=\varepsilon_{n+1}$
and \eqref{eq: epsilon delta} we conclude that both of the two resulting maps $\CA_{n+1}\to \CA_n\otimes \C[x_{n+1}]$ agree with $\varphi$.
\end{proof}

If $M$ is a module over $\CA_w$, then $M\sqcup \one_1$ is naturally a module over $\CA_{w\sqcup\one_1}$, where $w\sqcup\one_1\in S_{n+1}$ is the permutation whose restriction to $1,\ldots,n$ coincides with $w$.

The following observation turns out to be very useful:
\begin{lemma}
    \label{lem:skein}
    There are maps of Rouquier complexes $T^{-1}_i \to K_{i,i+1} \to T_i[-1]$ whose composition is homotopic to zero. The maps and the homotopy are $\CA$-linear. The resulting twisted complex is contractible. In particular, in the category $\CA/B$ we have isomorphisms
    \[
    T_i^{-1} \cong \left[\underline{K_{i,i+1}} \to T_i[-1]\right], \qquad T_i \cong \left[\underline{T_i^{-1}} \to K_{i,i+1}\right].
    \]
\end{lemma}
\begin{proof}
    This is clear from the diagram
    \[
    \begin{tikzcd}
        R(-1) \ar{r}{b_i^*} \ar{d}{\Id_R} & B_i(1) \ar{d}{b_i} \\
        R(-1) \ar{r}{x_i-x_{i+1}} \ar{d}{b_i^*} & R(1) \ar{d}{\Id_R}\\
        B_i(1) \ar{r}{b_i} & R(1)
    \end{tikzcd}
    \]
\end{proof}

The Markov II property first has a version which holds in $\CCA_w/R$.
\begin{proposition}\label{prop:markov 1}
    For any $\CA_w$-module $M$ and any $k$ we have the following isomorphisms in $\CCA_{w\sqcup\one_1}/R_{n+1}$, respectively $\CCA_{(w\sqcup\one_1)(n\; n+1)}/R_{n+1}$:
    \begin{equation}\label{eq:decomposition}
    \HH^k(M\sqcup\one_1) = (\HH^k(M)\sqcup\one_1) \oplus (\HH^{k-1}(M)\sqcup\one_1),
    \end{equation}
    \begin{multline}\label{eq:markov}
    \HH^{k+1}\left((M\sqcup\one_1) \otimes_{R_{n+1}}T_n^{-1} \right) \xrightarrow{\sim} \left(\HH^{k}(M)\sqcup\one_1\right) \otimes_{R_{n+1}} K_{n,n+1}
    \\
    \xrightarrow{\sim} \HH^k\left((M\sqcup\one_1) \otimes_{R_{n+1}}T_n \right)[-1].
    \end{multline}
\end{proposition}
\begin{proof}
    The decomposition \eqref{eq:decomposition} is a property of the functor $\HH$. The functor is computed by the Koszul complex, which splits as a direct sum in the case of a module of the form $M\sqcup\one_1$.
    The first isomorphism in \eqref{eq:markov} is obtained by composing the three maps
    \[
    \HH^{k+1}\left((M\sqcup\one_1) \otimes_{R_{n+1}}T_n^{-1} \right) \to
    \HH^{k+1}\left((M\sqcup\one_1) \otimes_{R_{n+1}}K_{n,n+1} \right)
    \]
    \[ \to \HH^{k+1}\left(M\sqcup\one_1\right) \otimes_{R_{n+1}}K_{n,n+1} \to \left(\HH^{k}(M)\sqcup\one_1\right) \otimes_{R_{n+1}} K_{n,n+1}.
    \]
    The first map is induced by a map from Lemma \ref{lem:skein}. The second map is the isomorphism of the projection formula (Proposition \ref{prop:projection}). Finally, the last map is obtained from \eqref{eq:decomposition}. So the composition is a well-defined map of $\CCA_{n+1}$-modules. It is a homotopy equivalence of $R_{n+1}$-modules by a well-known argument, e.g. see \cite{Kh}. Construction of the second isomorphism in \eqref{eq:markov} is analogous.
\end{proof}

\subsection{The stabilized algebras}
\begin{definition}
    For any $c\leq n$, where $c$ is a positive integer and $n$ is a positive integer or infinity, let
    \[
    \CCA_{c,n} = \C\left[(x_i)_{i=1}^c, (\xi_i)_{i=1}^c, (u_k)_{k=1}^n \;|\; d x_i=d\xi_i=0,\; d u_k = k \sum_{i=1}^c x_i^{k-1} \xi_i\right].
    \]
\end{definition}
Clearly, we have $\CCA_{n,n}=\CCA_{\Id_n}$. More generally, the algebra $\CCA_{c,n}$ is a free extension of $\CCA_{\Id_c}$ by the generators $u_{c+1},\dots,u_{n}$.

\begin{remark}
    The ring of symmetric polynomials $R^{S_n}$ is a free commutative algebra generated by the power sum polynomials $p_1,\dots,p_n$. Let $U$ be the unique derivation $R^{S_n}\to \CCA_{n,n}$ satisfying $U(p_k)=u_k$ for $k\leq n$. Then for any symmetric polynomial $f$ we have the identity
    \[
    d U(f) = \sum_{i=1}^n \frac{\partial f}{\partial x_i} \xi_i.
    \]
    One can obtain different presentations of the algebra $\CCA_{n,n}$ by choosing different sets of generators of $R^{S_n}$.
\end{remark}

Let $w\in S_n$ be a permutation with $c$ cycles. Let $C_1, \ldots, C_c$ be the cycles and choose a representative $j_i\in C_i$ for each cycle $C_i$. Define a homomorphism
$\alpha:\CCA_w \to \CCA_{c,n}$ by
\[
\alpha(x_j) = x_i\;(j\in C_i),\quad \alpha(\xi_j) = \begin{cases}\xi_i& (j=j_i)\\ 0 & \text{otherwise}\end{cases},\quad \alpha(u_k)=u_k.
\]
Define a homomorphism
$\beta:\CCA_{c,n}\to \CCA_w$ by
\[
\beta(x_i) = x_{j_i},\quad \beta(\xi_i) = \sum_{j\in C_i} \xi_j,\quad \beta(u_k) = u_k + \text{correction}.
\]
\begin{proposition}
\label{prop: twisted projector}
\begin{enumerate}
\item[(a)]  For each $k$ there exists a correction making $\beta$ into a dg algebra homomorphism such that the composition $\alpha\circ\beta$ is the identity and
 $\alpha$ and $\beta$ are quasi-isomorphisms.
\item[(b)] The pull-back functors
\[
\alpha^*: \CCA_{c,n}/\C\to \CCA_w/\C,\quad \beta^*:\CCA_w/\C \to \CCA_{c,n}/\C
\]
are mutually inverse equivalences of categories .
\item[(c)] For any $\CCA_w$-module $X$ the isomorphism class of the pull-back $\beta^* X$ in the category $\CCA_{c,n}/\C$ is independent of the choices of the representatives $j_i$ and the corrections, and similarly for $\alpha^*$.
\end{enumerate}
\end{proposition}
\begin{proof}
(a) The correction needs to satisfy
\[
d(\text{correction}) = \beta(d u_k)  - d(u_k)= \sum_{i=1}^c \sum_{j\in C_i} \left(k x_{j_i}^{k-1} - h_{k-1}\left(x_j, x_{w^{-1}(j)}\right)\right) \xi_j.
\]
So it is sufficient to find a separate correction for each cycle. Without loss of generality, assume $w$ consists of a single cycle. Moreover, by reindexing assume $w(j)=j-1 \pmod{n}$ and $j_1=1$. A correction can be constructed as follows:
\[
d\left(\sum_{i=1}^{n-1} h_{k-2}\left(x_1, x_i, x_{i+1}\right) \xi_i (\xi_{i+1}+\cdots+\xi_n)\right)
\]
\[
= \sum_{i=1}^{n-1} \left(h_{k-1}\left(x_1, x_i\right) - h_{k-1}\left(x_1, x_{i+1}\right)\right)(\xi_{i+1}+\cdots+\xi_n) + \sum_{i=1}^{n-1} \left(h_{k-1}\left(x_1,x_i\right)-h_{k-1}\left(x_i,x_{i+1}\right)\right) \xi_i.
\]
\[
= \sum_{i=1}^{n-1} \left(h_{k-1}\left(x_1,x_1\right)-h_{k-1}\left(x_i,x_{i+1}\right)\right) \xi_i,
\]
as required. By construction, $\alpha\circ\beta=\Id$.

To show that $\beta$ is a quasi-isomorphism notice that via $\beta$ we can view the algebra $\CCA_w$ as
    \[
    \CCA_w = \CCA_{c,n}\left[t_j,\xi_j\;(j\in\{1,\ldots,n\}\setminus\{j_1,\ldots,j_c\})\right],
    \]
    where $t_j=x_j-x_{w^{-1}(j)}$. We have $d \xi_j=t_j$, so the algebras $\CCA_w$ and $\CCA_{c,n}$ are quasi-isomorphic.
     The homomorphism $\alpha$ is a partial inverse to $\beta$, and therefore is a quasi-isomorphism too.

By Lemma \ref{lem:algebra weak equivalence} the functors $\alpha^*$ and $\beta^*$ are equivalences of categories. Since we have $\alpha\circ\beta=\Id$ the composition $\beta^* \alpha^*$ is the identity functor. Therefore $\alpha$ and $\beta$ are inverses of each other.

Now we can construct $\alpha$ and $\beta$ using different choices of $j_i$. Keeping the choice for $\alpha$ fixed, if $\beta$ and $\beta'$ are defined using two different choices, since both $\beta^*$ and $\beta'^*$ are inverses of $\alpha^*$ they must be isomorphic. We proceed similarly in the case of two different versions of $\alpha$.
\end{proof}

\begin{remark}
A similar result is proved in \cite[Theorem 4.43]{AH}.
\end{remark}

Let $\mu=(\mu_1,\ldots,\mu_c)$ be a collection of positive integers so that $\sum \mu_i = n$. Such a collection is called a composition of $n$. We write $\mu\vDash n$.
For $N>n$ and $\mu\vDash n$ we construct a homomorphism $\Phi_\mu^N:\CCA_{c,N}\to \CCA_{c,n}$ as follows:
\begin{proposition}\label{prop:limit}
    \begin{enumerate}
    \item[(a)] There is a unique homomorphism $\Phi_\mu^N:\CCA_{c,N}\to \CCA_{c,n}$ which sends the variables $x_i$, $\xi_i$ and $u_k$ for $k\leq n$ to themselves and the generating series
    \[
    \left(\sum_{k=1}^N \frac{\Phi_\mu^N(u_k)}{k} t^k\right) \prod_{i=1}^c (1-t x_i)^{\mu_i} + O(t^{N+1})
    \]
    has no terms of degree $>n$ in $t$.
    \item[(b)] Let $x^\mu$ be a sequence of variables $x_1, \cdots,x_1, x_2,\ldots,x_c$ where each $x_i$ occures $\mu_i$ times. We have
    \[
    \Phi_\mu^N(u_k) = \sum_{\ell=1}^n u_\ell \frac{\partial p_k}{\partial p_\ell} (x^\mu),
    \]
    where the partial derivative is taken in the ring of symmetric functions in $n$ variables.
    \item[(c)] for any $\mu'\vDash N$ satisfying $\mu_i'\geq \mu_i$ and $N'\geq N$ we have $\Phi_\mu^N \circ \Phi_{\mu'}^{N'} = \Phi_\mu^{N'}$.
    \end{enumerate}
\end{proposition}
\begin{proof}
    Since $\Phi_\mu^N(u_k)=u_k$ for $k\leq n$, we can determine the expansion of the generating series up to $O(t^{n+1})$, but since there are no terms of degree $>n$ the series is completely determined, and we can compute $\Phi_\mu^N(u_k)$ for $k>n$. Let us prove that the obtained homomorphism commutes with the differential. Applying $d$ to the generating series we see that the series
    \begin{equation}\label{eq:series1}
    \left(\sum_{k=1}^N \frac{d(\Phi_\mu^N(u_k))}{k} t^k\right) \prod_{i=1}^c (1-t x_i)^{\mu_i} + O(t^{N+1})
    \end{equation}
    has no terms of degree $>n$. On the other hand, the series
    \begin{equation}\label{eq:series2}
    \left(\sum_{k=1}^N \left(\sum_{i=1}^c \xi_i x_i^{k-1}\right) t^k\right) \prod_{i=1}^c (1-t x_i)^{\mu_i} + O(t^{N+1})
    \end{equation}
    also has no terms of degree $>n$ because we have
    \[
    \sum_{k=1}^\infty x_i^{k-1} t^k = \frac{t}{1-x_i t}.
    \]
    So the series \eqref{eq:series1} and \eqref{eq:series2} must agree and therefore we have
    \[
    d(\Phi_\mu^N(u_k)) = \sum_{i=1}^c k x_i^{k-1} \xi_i = \Phi_\mu^N(d(u_k)).
    \]
    This establishes (a). To verify (b) it is sufficient to verify that the series
    \[
    \sum_{\ell=1}^n u_\ell \sum_{k=1}^\infty \frac{t^k}{k}  \frac{\partial p_k}{\partial p_\ell} (x^\mu)
    \]
    when multiplied by $\prod_{i=1}^c (1-t x_i)^{\mu_i}$ has no terms of degree $>n$. The series in question can be rewritten as
    \[
    \sum_{\ell=1}^n u_\ell \frac{\partial}{\partial p_\ell} \log(1-e_1 t +\cdots \pm e_n t^n),
    \]
    and so is a rational function in $t$ with numerator of degree $n$ and denominator $1-e_1 t +\cdots \pm e_n t^n$, which when specialized to $x^\mu$ becomes $\prod_{i=1}^c (1-t x_i)^{\mu_i}$. So the claim is evident.

    To prove (c) we use (a). For $k\leq N$ we have $\Phi_{\mu'}^{N'}(u_k)=u_k$ and $\Phi_\mu^N(u_k)=\Phi_\mu^{N'}(u_k)$, so the claim holds. The series $\sum_{k=1}^{N'} \frac{\Phi_\mu^{N'}(u_k)}{k} t^k$ when multiplied by the product $\prod_{i=1}^c (1-t x_i)^{\mu_i}$ has no terms of degree $>n$. Multiplying further by $\prod_{i=1}^c (1-t x_i)^{\mu_i'-\mu_i}$ cannot introduce terms of degree $>N$. So we have that
    \[
    \sum_{k=1}^{N'} \frac{\Phi_\mu^{N'}(u_k)}{k} t^k \prod_{i=1}^c (1-t x_i)^{\mu_i'}
    \]
    has no terms of degree $>N$. The same holds for $\Phi_{\mu'}^{N'}(u_k)$, and therefore for $\Phi_\mu^N(\Phi_{\mu'}^{N'}(u_k))$ in the place of $\Phi_\mu^{N'}(u_k)$. Since the sequences $\Phi_\mu^N(\Phi_{\mu'}^{N'}(u_k))$ and $\Phi_\mu^{N'}(u_k)$ have the same first $N$ terms, the sequences must agree.
\end{proof}
Now we can finish proving the Markov property. For a permutation $w$ denote by $\mu(w)$ the composition $\mu_i = |C_i|$. Denote $w'= (w\sqcup\one_1)(n\;n+1)$.
\begin{proposition}\label{prop:markov 2}
    Let $w\in S_n$ and let $X$ be a $\CCA_w$-module. Suppose $N\geq n+1$. Then the modules $\Phi_{\mu(w')}^{N*} \beta^*((X\sqcup\one_1)\otimes_{R_{n+1}} K_{n,n+1}(-1))$ and $\Phi_{\mu(w)}^{N*} \beta^* X$ are isomorphic in the category $\CCA_{c,N}/\C$.
\end{proposition}
\begin{proof}
    In view of (3) of Proposition \ref{prop:limit} it is sufficient to prove the statement for $N=n+1$. Consider the following diagram of dg algebras
    \[
    \begin{tikzcd}
        \CCA_{c,n+1} \arrow{rr}{\beta} \ar{d}{\Phi_{\mu(w)}^{n+1}} && \CCA_{w'} \arrow{d}{\widetilde\varphi} \arrow{rr}{\alpha} && \CCA_{c,n+1} \ar{d}{\Phi_{\mu(w)}^{n+1}}\\
        \CCA_{c,n} \arrow{r}{\beta} & \CCA_{w} \arrow{r}{\sim} & \CCA_w[x_{n+1},\xi_{n+1}]\arrow{r}{\sim} & \CCA_{w}\arrow{r}{\alpha} & \CCA_{c,n}\\
    \end{tikzcd}
    \]
    The algebra $\CCA_w[x_{n+1},\xi_{n+1}]$ has the differential $d x_{n+1}=0, \ d\xi_{n+1}=x_{n+1} - x_n$. The quasi-isomorphisms to and from $\CCA_w$ to this algebra are evident. The morphism $\widetilde\varphi$ is defined similarly to the homomorphism $\varphi$ in Section \ref{ssec:stabilization} in such a way that
    \[
    (X\sqcup\one_1)\otimes_{R_{n+1}} K_{n,n+1}(-1) = \widetilde\varphi^* X[x_{n+1},\xi_{n+1}]
    \]
    holds. Indeed, the left hand side is isomorphic to $X[x_{n+1},\xi_{n+1}]$ as an abstract complex, and the action of each generator of $\CCA_{w'}$ on it is given by an explicit expression which involves the action of $\CCA_w$ on $X$ and the elements $x_{n+1},\xi_{n+1}$. These are packaged into $\widetilde\varphi$. 
    
    The large square containing both $\alpha$ and $\beta$ is commutative. Let us show that the rightmost square is commutative. It is clearly commutative on the generators $x_i$ and $\xi_i$, as well as the generators $u_k$ for $k\leq n$. It remains to consider the generator $u_{n+1}$. The element $\widetilde \varphi(u_{n+1})$ is the sum of
    \[
    \sum_{k=1}^n u_k \frac{\partial p_{n+1}}{\partial p_k} (x_1,\ldots,x_n) + h_{n-1}\left(x_{w^{-1}(n)}, x_n, x_{n+1}\right) \xi_{w^{-1}(n)} (-\xi_{n+1}),
    \]
    and certain corrections (see Lemma \ref{lem:corrections}). The second term above goes to zero in $\CCA_w$. By Lemma \ref{lem:corrections} we can make sure that the corrections belong to the ideal generated by $x_i-x_{w(i)}$, and therefore go to zero after applying $\alpha$. So we are left with
    \[
    \sum_{k=1}^n u_k \frac{\partial p_{n+1}}{\partial p_k} (\alpha(x_1),\ldots,\alpha(x_n)),
    \]
    which equals $\Phi_{\mu(w)}^{n+1}(u_{n+1})$ by  Proposition \ref{prop:limit}(b).

    So we have shown that the rightmost square of the diagram is commutative. Applying Lemma \ref{lem:algebra homomorphisms}, we obtain that module $\beta^*((X\sqcup\one_1)\otimes_{R_{n+1}} K_{n,n+1}(-1))$ is isomorphic to the module obtained from $X[x_{n+1},\xi_{n+1}]$ by pulling back along $\beta$, $\Phi_{\mu(w)}^{n+1}$ and the quasi-isomorphism $\CCA_w\simeq \CCA_w[x_{n+1},\xi_{n+1}]$. But as soon as we pull it back to $\CCA_w$, the module becomes isomorphic in the category $\CCA_w/\C$ to $X$.
\end{proof}

Putting things together, we obtain
\begin{theorem}\label{thm:link invariant}
    Let $L$ be a link with $c$ components labeled $1, \ldots, c$ represented as the closure of a braid $b$ on $n$ strands with the corresponding permutation $w$. Let $e$ be the number of the positive crossings of $\beta$ minus the number of the negative ones. For each $k$ define a $\CCA_{c,\infty}$-module by
    \[
    \HH^k(L) = \Phi_{\mu(w)}^{*\infty} \beta^* \HH^{k+ \frac{n-e-c}2} (T_b)\left[\frac{-n-e+c}2\right](c-n).
    \]
    This module is independent of the presentation of the link up to an isomorphism in the category $\CCA_{c,\infty}/\C$.
\end{theorem}
\begin{proof}
The cyclic invariance was established in Proposition \ref{prop:cyclicity}. Indeed, for permutations $v,w$ there is a bijection between the cycles of permutations $vw$ and $wv$, therefore the isomorphism $\tau_{v,w}:\CA_{vw}\to \CA_{wv}$ intertwines the maps $\beta_{vw}:\CA_{c,n}\to \CA_{vw}$ and  $\beta_{wv}:\CA_{c,n}\to \CA_{wv}$ for some choices of representatives $j_i$ and corrections. By Proposition  \ref{prop: twisted projector} these choices do not affect the isomorphism classes of pullbacks under  $\beta_{vw}$ and $\beta_{wv}$ respectively.  

Let us prove invariance under the Markov II moves. If $b$ is a braid on $n$ strands, then
\[
\HH^k\left((T_b \sqcup \one_1)\otimes_{R_{n+1}} T_n\right)\cong \left(\HH^k(T_b) \sqcup \one_1\right)\otimes_{R_{n+1}} K_{n,n+1}[1]
\]
by Proposition \ref{prop:markov 1}. The claim follows from Proposition \ref{prop:markov 2}. The proof for $T_n^{-1}$ is analogous.
\end{proof}

\subsection{$\CCA_{c,\infty}/\C$ vs. $\CCA_{c,\infty}/R_c$} The complexes $\HH^k(T_b)$ are known to be complexes of free $R$-modules, hence free over the subring $R_c=\C[x_1,\ldots,x_c]$.
This is also clear from Proposition \ref{prop:reduction 1} below. In this case there is no difference between viewing them as objects of $\CCA_{c,\infty}/\C$ or $\CCA_{c,\infty}/R_c$ by the following:
\begin{proposition}\label{prop:two versions of the category}
    The category $\CCA_{c,\infty}/\C$ is equivalent to the full subcategory of $\CCA_{c,\infty}/R_c$ consisting of objects which are free over $R_c$.
\end{proposition}
\begin{proof}
    Denote the category of objects of $\CCA_{c,\infty}/R_c$ which are free over $R_c$ by $\CCA_{c,\infty}/R_c/\C$.
    The identity functor induces a functor
    \begin{equation}\label{eq:3 functor arrows}
    \CCA_{c,\infty}/R_c/\C \to \CCA_{c,\infty}/\C.
    \end{equation}
    Let $X$ be a $\CCA_{c,\infty}$-module and let $\widetilde X$ be its resolution relative to $\C$ with the counit map $\widetilde X \to X$.
    Since $\CCA_{c,\infty}$ is free over $R_c$, any $\CCA_{c,\infty}$ module which is induced from $\C$ is free over $R_c$. In particular, $\widetilde X$ is free over $R_c$ and so we see that $\widetilde X$ belongs to $\CCA_{c,\infty}/R_c/\C$. So \eqref{eq:3 functor arrows} is essentially surjective. If $X$ is free over $R_c$ then when restricted to $R_c$ it can be viewed as its own resolution relative to $\C$, so the counit map must be a homotopy equivalence in $R_c-\Mod$. So $\widetilde X$ can be viewed as a resolution of $X$ relative to $R_c$. So $\widetilde X$ is both a resolution of $X$ relative to $\C$ and relative to $R_c$, so the $\Hom$ spaces in the two categories are isomorphic.
\end{proof}

\subsection{Basic objects}\label{sec:basic objects classification}
We can restrict the kind of objects that appear as $\HH^k(L)$ as follows:
\begin{proposition}\label{prop:reduction 1}
    For each link $L$ with $c$ components and any $k$ the object $\HH^k(L)$ is equivalent to a twisted complex
    \[
    \HH^k(L) \cong \left[\cdots \to X_{1} \to \underline{X_0} \to X_{-1} \to \cdots \right],
    \]
    where each $X_i$ is a direct sum of objects of the form $\Phi_{\mu(w)}^{*\infty} \beta^* K(c-n)$,  $K=K_{r_1}\cdots K_{r_l}$ is some Koszul complex and $w=r_1\cdots r_l \in S_n$ is a permutation with $c$ cycles. 
\end{proposition}
\begin{proof}
If $L$ is presented as the closure of a braid $b\in \Br_n$, then $\HH^k(L)$ is isomorphic to $\HH^k(b)(c-n)$ up to homological shift, by definition.  We will prove that
\[
\Phi_{\mu(w)}^{*\infty} \beta^* \HH^k(T_b \otimes_{R_n} K)
\]
is weakly equivalent to a twisted complex consisting of objects of the form $\Phi_{\mu(w)}^{*\infty} \beta^* K$ (and homological shifts thereof), for any braid $b\in \Br_n$ and any Koszul complex $K$. We proceed by induction on the length of $b$. Using Lemma \ref{lem:skein} we can always replace $T_i$ by $T_i^{-1}$ and vice versa modulo complexes for shorter braids. Moreover, we are free to replace $b$ by a conjugate braid by Proposition \ref{prop:cyclicity} and Corollary \ref{cor: T and K standard}.

Using the above transformations we can always reduce the number of crossings in $b$ (this is essentially Jaeger's algorithm for computing HOMFLY-PT polynomial \cite{Jaeger}) until we obtain a product of the form $T_{i_1} T_{i_2} \cdots T_{i_m} K$ with $i_1<\ldots<i_m$. Consider
\[
\Phi_{\mu(w)}^{*\infty} \beta^* \HH^k(T_{i_1} T_{i_2} \cdots T_{i_m} K).
\]
Applying Proposition \ref{prop:projection}, we can move $K$ outside $\HH^k$. Finally, we iteratively apply Propositions \ref{prop:adding a thread} and \ref{prop:markov 1} to show that $\HH^k(T_{i_1} T_{i_2} \cdots T_{i_m})$
is a direct sum of copies of $K_{i_1,i_1+1} K_{i_2, i_2+1}\cdots K_{i_m, i_m+1}$.
\end{proof}

Next we want to classify objects of the form $\Phi_{\mu(w)}^{*\infty} \beta^* K(c-n)$.


 By Lemma \ref{lem: K standard} the object $K$ is completely classified by its permutation $w$, the equivalence relation $\sim$ and the number of $K_{i,j}^2$ in each equivalence class.
\begin{proposition}
    In Proposition \ref{prop:reduction 1} we can assume that the product of Koszul complexes $K$ satisfies the following extra assumptions:
    \begin{enumerate}
        \item All cycles of $w$ have length $1$ or $2$.
        \item If $(i,j)$ is a cycle of $w$ of length $2$ then $K$ contains some odd power of $K_{i,j}$ and no other $K_{i',j'}$ with $i'\in \{i,j\}$ or $j'\in \{i,j\}$.
    \end{enumerate}
\end{proposition}
\begin{proof}
    If $w$ has a cycle of length at least $3$ we first renumber the strands to make sure that $n-1, n, n+1$ are in the cycle. Choose the decomposition of $w$ to contain exactly one $(n\ n+1)$ and replace any occurrence $K_{i,n+1}^2$ by $K_{i,n}^2$ ($i<n$) and any $K_{n,n+1}^2$ by $K_{n-1,n}^2$. This does not change the isomorphism class of the object $K$ by Lemma \ref{lem: K standard}. Then apply Proposition \ref{prop:markov 2} to reduce the number of strands.

    If $(i,j)$ is a cycle of $w$ of length $2$ and there exist $K_{i, j'}^2$, we can relabel the indices so that $i=n$, $j=n+1$ and $j'=n-1$. Then continue as in the first case.
\end{proof}

We encode any object satisfying (1) and (2) by an unordered collection of pairs
\[
(\bar n, \bar g) = ((n_1, g_1),\ldots,(n_s, g_s))
\]
as follows. Each cycle $(i,j)$ of $w$ of length $2$ is encoded as $(1,g)$ if the number of occurrences of $K_{i, j}$ is $1+2 g$. The remaining cycles all have lengths $1$. Each equivalence class $S$ of  cycles of lengths $1$ is encoded as $(n, g)$ where $n$ is the number of cycles in the class and $g+n-1$ is the number of occurrences of $K_{i,j}^2$ in the decomposition of $K$ with $i,j\in S$. Note that $n-1$ squares are used to generate the equivalence relation. Note that $(1,0)$ may correspond to an index which doesn't appear in $K$ or it can correspond to a $2$-cycle $i,j$ with a single $K_{i,j}$, but these produce equivalent objects by Proposition \ref{prop:markov 2}.

To go back, a single pair $(n,g)$ corresponds to $K_{12}^{2g+1}(-1)$ if $n=1$ and
\[
K_{12}^{2g+2} K_{23}^2 \cdots K_{n-1, n}^2 \qquad(n\geq 2).
\]
The objects for a collection of pairs have to be stacked together horizontally, i.e. the total set of indices is identified with the disjoint union of the sets of indices, one set for each pair, the generators $K_{i,j}$ have to be relabeled accordingly and multiplied together. We denote the result by $K_{\bar n, \bar g}$.

\begin{remark}\label{rmk:K abuse}
Below we will frequently abuse notation, denoting $\Phi_{\mu(w)}^{*\infty} \beta^*(K_{\bar n, \bar g})$  simply by $K_{\bar n, \bar g}$.
\end{remark}

\begin{remark}
	Products of $K_{ij}$ have the following topological interpretation. Consider $n$ disks labeled by $1,\ldots, n$ and connect $i$-th and $j$-th disk by a twisted band for each $K_{ij}$ appearing in $K$. The result is an oriented surface $\Sigma$.  The classification in Lemma \ref{lem: K standard} has a simple topological meaning: the cycles in $w$ correspond to the components of the boundary $\partial \Sigma$, equivalence classes for $\sim$ correspond to the connected components of $\Sigma$ and the number of  $K_{i,j}^2$ in each equivalence class encodes the Euler characteristic of the corresponding component so that adding an extra $K_{ij}^2$  for $i$ and $j$ in the same equivalence class corresponds to adding a handle. The topological meaning of the invariant $(\bar n, \bar g)$ is that $\Sigma$ is the disjoint union of surfaces $\Sigma_{n_i,g_i}$, where $\Sigma_{n,g}$ stands for the connected surface of genus $g$ with $n$ boundary components.

For instance, $K_{12}^{2g+1}$ corresponds to $2g+1$ twisted bands between two disks. This is a surface with one boundary component and Euler characteristic $2-(2g+1)=1-2g$, so it is a genus $g$ surface with one puncture. The product $K_{12}^{2g+2} K_{23}^2 \cdots K_{n-1, n}^2 $
corresponds to a surface with $n$ boundary components and Euler characteristic $n-(2g+2n-2)=2-2g-n$, hence it is a genus $g$ surface with $n$ punctures. 
\end{remark}

Finally, we describe $\beta^*K_{\bar n, \bar g}$. It is convenient to notice that objects on which $\xi_i$ acts as zero can be factored away:
\begin{proposition}\label{prop:factor away xi zero}
    Suppose $M$ is a module over $\CCA_w$ for some permutation $w\in S_n$ with $c$ cycles, and suppose $N$ is a module over $\CCA_n$ which is free over $R$ such that every $\xi_i$ acts as zero. Then we have an isomorphism of $\CCA_{c,n}$-modules
    \[
    \beta^*(M\otimes_{R_n} N) \cong \beta^*(M) \otimes_{R_c} \left(R_c\otimes_{R_n} N\right),
    \]
    where the homomorphism $R_n\to R_c$ is the natural homomorphism sending each variable $x_i$ to the variable corresponding to the cycle to which $i$ belongs. The action of $u_k$ on the right hand side is given by $u_k\otimes 1 + 1\otimes u_k$ and the action of $\xi_i$ is the one coming from $\beta^*(M)$.
\end{proposition}
\begin{proof}
    First, since $\alpha\circ\beta\circ\alpha=\alpha$ and $\alpha$ is a quasi-isomorphism, the pullbacks via $\beta\circ\alpha$ and $\Id$ produce equivalent objects by Lemma \ref{lem:algebra weak equivalence}. Since $N$ is free, tensoring by $N$ preserves quasi-isomorphisms, so we can replace $M$ by $\alpha^*\beta^*M$ on the left hand side. On $\alpha^*\beta^*M$ variables $x_i$ corresponding to $i$ from the same cycle act in the same way, so we have an isomorphism
    \[
    \alpha^*\beta^*M\otimes_{R_n} N = \alpha^*\beta^*M\otimes_{R_n} \left(R_c\otimes_{R_n} N\right).
    \]
    Because all the $\xi_i$ vanish on $N$, the action of $\xi_i$ on the tensor product comes from the action on $M$ only. Now when we apply $\beta^*$ the corrections in the definition of $\beta^*$ are all contained in the corrections for $M$, so we have
    \[
    \beta^*(\alpha^*\beta^*M\otimes_{R_n} \left(R_c\otimes_{R_n} N\right)) = \beta^*(\alpha^*\beta^*M)\otimes_{R_c} \left(R_c\otimes_{R_n} N\right) = \beta^*(M) \otimes_{R_c} \left(R_c\otimes_{R_n} N\right).
    \]
\end{proof}

The following is useful in computations of higher powers of $K_{i,j}$:
\begin{proposition}\label{prop:factor cube}
    Denote by $K_{i,j}^*$ the $\CCA$-module $R[\eta_1,\eta_2](2)$, where $\eta_1, \eta_2$ have homological degree $1$, $d \eta_1=d\eta_2=0$, all $\xi_i$ act by zero and $u_k$ acts via
    \[
    u_k \to k h_{k-2}(x_i, x_j) \eta_1 \eta_2.
    \]
    Then we have
    \[
    K_{i,j}^3 = K_{i,j} K_{i,j}^*.
    \]
\end{proposition}
\begin{proof}
    Explicit computation of $K_{i,j}^3$ produces $R[\eta_1,\eta_2,\eta_3](3)$ with
    \[
    d \eta_i = x_i-x_j,\quad \xi_i = -\xi_j= \eta_1-\eta_2+\eta_3,
    \]
    \[
    u_k=(\eta_1(-\eta_2) + (\eta_1-\eta_2) \eta_3) (h_{k-2}(x_i,x_j,x_i) + h_{k-2}(x_j,x_i,x_j)).
    \]
    The first factor in $u_k$ can be written as $(\eta_1-\eta_3)(\eta_1-\eta_2)$. The second factor contains every monomial $x_i^p x_j^q$ with coefficient $p+1+q+1=k$, so we obtain that $u_k$ acts by
    \[
    (\eta_1-\eta_3)(\eta_1-\eta_2) k h_{k-2}(x_i, x_j).
    \]
    Redefining $\eta_1$, $\eta_2$ as $\eta_1-\eta_3$, $\eta_1-\eta_2$ we arrive at the following presentation of $K_{i,j}^3$:
    \[
    K_{i,j}^3 = R[\eta_1,\eta_2,\xi_i | d \eta_1=d \eta_2=0,\; d \xi_i=x_i-x_j],
    \]
    the action of $\xi_j$ is $-\xi_i$, and $u_k$ acts via $k h_{k-2}(x_i, x_j) \eta_1 \eta_2$. This is precisely the module $K_{i,j} K_{i,j}^*$.
\end{proof}

Finally we obtain
\begin{proposition}\label{prop:basic object}
    Let $\bar n = (n_1,\ldots,n_s)$, $\bar g=(g_1,\ldots,g_s)$, $c=\sum_i n_i$, $g=\sum_i g_i$, $m_i=n_1+\cdots+n_i$, $I=\{m_i\}_{i=1}^s$.
    The module $K_{\bar n, \bar g}$ is isomorphic in the category $\CCA_{c,\infty}/\C$ to
    \[
    \C[(x_j)_{j\in I},\; (\xi_j)_{j\notin I}, \; (\eta_{i j})_{j\leq 2 g_i,i\leq s}]\left(c-s+2 g\right)
    \]
    with zero differential and the action of the remaining variables given by
\begin{equation}\label{eq:remaining variables}
    x_j=x_{m_i}\;(m_{i-1}<j<m_i),\quad \xi_{m_i} = -\sum_{j=m_{i-1}+1}^{m_i-1} \xi_j,
\end{equation}
    \[
    u_k = k(k-1) \sum_{i=1}^s x_{m_i}^{k-2} \sum_{j=1}^{g_i} \eta_{i, 2j-1} \eta_{i, 2 j}.
    \]
\end{proposition}
\begin{proof}
Suppose $g=0$. Then each $2$-cycle of $w$ corresponds to $n_i=1, g_i=0$, so the corresponding element $K_{i',i'+1}$ appears in degree $1$. Applying Proposition \ref{prop:markov 2} we get rid of this element decreasing the number of strands by $1$. So the permutation becomes the identity permutation and each $n_i$ corresponds to a shifted version of the product
\[
K^{(n_i)}=K_{12}^2 K_{23}^2 \cdots K_{n_i-1,n_i}^2.
\]
These products are multiplied together over $\C$. Let us consider each product separately.
Let $n=n_i$. From the definition, $K_{i,j}^2$ can be computed as follows:
\[
K_{i,j}^2 = R_n[\eta_1,\eta_2,\,|\, d \eta_1 = d \eta_2 = x_1-x_2](1),
\]
where the action is given by
\[
\xi_i = -\xi_j = \eta_1-\eta_2,\; u_k = - k \eta_1 \eta_2 h_{k-2}(x_i, x_j).
\]
After a change of variables, we can represent it as
\[
K_{i,j}^2 = R_n[\eta,\xi_i\,|\, d \eta = x_1 - x_2,\; d \xi_i = 0](1),
\]
with the action $u_k = k \eta \xi_i h_{k-2}(x_i, x_j)$. Tensoring these together, we obtain
\[
K^{(n)} \cong R_n[\xi_1^*,\cdots,\xi_{n-1}^*, \eta_1,\cdots,\eta_{n-1}\,|\, d\xi_i^* = 0, d\eta_i = x_{i+1} - x_i](n-1),
\]
with the action given by $\xi_{n}=-\xi_{n-1}^*$, $\xi_1=\xi_1^*$,
\[
\xi_i = \xi_i^* - \xi_{i-1}^*,\;u_k = k \sum_{i=1}^{n-1} \eta_i \xi_i^* h_{k-2}(x_i, x_{i+1}) + \sum_{i=2}^{n-1} (-\xi_{i-1}^*) \xi_i^* h_{k-2}(x_i, x_i, x_i)
\]
\[
= \sum_{n-1\geq i\geq j\geq 1} \left( k \eta_i \xi_j h_{k-2}(x_i, x_{i+1}) + \binom{k}{2} \xi_i \xi_j x_i^{k-2}\right).
\]
Let us choose $\xi_1,\cdots,\xi_{n-1}, \eta_1+\frac{\xi_1}{2},\cdots,\eta_{n-1}+\frac{\xi_{n-1}}{2}$ as the new generators of $K^{(n)}$. The description of the module becomes
\[
K^{(n)} \cong R_n[\xi_1,\cdots,\xi_{n-1}, \eta_1,\cdots,\eta_{n-1}\,|\, d\xi_i = 0, d\eta_i = x_{i+1} - x_i](n-1),
\]
\[
u_k = \sum_{n-1\geq i\geq j\geq 1} k \left(\eta_i \xi_j h_{k-2}(x_i, x_{i+1}) + \frac{1}{2} \xi_i \xi_j (x_i - x_{i+1}) h_{k-3}(x_i, x_i, x_{i+1})\right)
\]
Let $K':=\C[x_n,\xi_1,\cdots,\xi_{n-1}]$ with zero differential and zero action of $u_k$. Each $x_i$ acts by $x_n$ and $\xi_n$ acts by $-\sum_{i=1}^{n-1}\xi_i$. Then we have a quasi-isomorphism $K^{(n)}\to K'(n-1)$ which sends $\eta_i$ to $0$ and all $x_i$ to $x_n$.

Tensoring these quasi-isomorphisms over $\C$ we obtain a quasi-isomorphism from $K_{\bar n, 0}$ to \[
\C[(x_i)_{i\in I},\; (\xi_i)_{i\notin I}](n-1)
\]
with zero differential and $u_k$ acting by zero, the remaining variables acting as in \eqref{eq:remaining variables}. Note that applying $\Phi^*$ to a module on which $u_k$ for all $k$ act by zero again produces a module on which all $u_k$ act by zero. So the statement for $g=0$ is proved.

Now suppose $g$ is arbitrary. The ``genus contribution'' can be factored out using Propositions \ref{prop:factor cube} and \ref{prop:factor away xi zero}:
\[
\beta^*(K_{\bar n, \bar g})\cong \beta^*(K_{\bar n, 0}) \otimes_{R_c} \bigotimes_{i=1}^c (K_{i,i}^*)^{g_i},
\]
where $K_{i,i}^*$ is the object defined analogously to $K_{i,j}^*$, but with $j=i$. So the statement follows from the case $g=0$.
\end{proof}

\section{$y$-ification of $\CA$-modules}
\label{sec:yified}

\subsection{$y$-ification}

Let $(C,d)$ be a complex of Soergel bimodules with a structure of $w$-twisted module over the the algebra $\CA$. By definition, this implies that $C$ admits an action of anticommuting operators $\xi_i$ such that $d(\xi_i)=x_i-x'_{w^{-1}(i)}.$

Let $R_y = \C[y_1,\ldots,y_n]$. Following \cite{GH}, we define the {\em strict $y$-ification} of $C$  as the complex $\CY(C):=C[y_1,\ldots,y_n]=C \otimes R_y$ with the twisted differential
\[
d_y = d + \sum_{i=1}^n \xi_i y_i.
\]
The variables $y_i$ are placed in homological degree $\deg_h y_i=-2$. Note that
\[
d_y^2= \sum_{i=1}^{n}\left(x_i-x'_{w^{-1}(i)}\right)y_i.
\]
Since $d_y^2\neq 0$, we will sometimes refer to $\CY(C)$ as to curved complex. It is easy to see that the definition of tensor product of $y$-ifications in \cite{GH} agrees with $y$-ification of tensor product of $\CA$-modules.

In particular, given a braid $\beta$ with the corresponding permutation $w$, we can use the $w$-twisted $\CA$-module structure on the Rouquier complex $T_{\beta}$ from Theorem \ref{thm:braid action} to define the $y$-ified complex $\CY(\beta)=\CY(T_\beta)$.

Passing to the Hochschild homology we have:
\begin{proposition}
    Consider the differential $d_{\HH,y}$ on the curved complex $\HH(\CY(\beta))$. Then
    $$
    d_{\HH,y}^2=0\mod \left(y_i-y_{w(i)}\right).
    $$
    After taking quotient by the ideal generated by $\left(y_i-y_{w(i)}\right)$ the curved complex $\HH(\CY(\beta))$ becomes an honest chain complex.
\end{proposition}

\begin{remark}
    Note that $\C[y_1,\ldots,y_n]/\left(y_i-y_{w(i)}\right)$ is a polynomial ring in $c$ variables if the closure of $\beta$ has $c$ components.
\end{remark}

This yields the following definition.
\begin{definition}
The $y$-ified Khovanov-Rozansky homology of the braid $\beta$ is defined as
$$
\HY(\beta):=\HHH\left(\CY(\beta) \otimes_{R_y} \C[y_1,\ldots,y_n]/\left(y_i-y_{w(i)}\right)_{i=1}^n,d_y\right).
$$
\end{definition}

It is proved in \cite{GH} that $\HY(\beta)$ up to grading shifts is the topological invariant of the closure of $\beta$.
We refer to \cite{GH} for more details on $y$-ifications and their properties.

The above constructions use only the action of $\xi_i$. The action of $u_k$ gives rise to interesting operators in the $y$-ified homology.

\begin{theorem}
	\label{thm:y braid action}
	Let $\beta$ be an arbitrary braid. Then:

\noindent (a) There is a family of chain maps $F_k$ on the $y$-ified Rouquier complex $\CY(\beta)$ satisfying
$$
[d_y,F_k]=[F_k,F_m]=[F_k,x_i]=0,\ [F_k,y_i]=h_{k-1}\left(x_i,x'_{w^{-1}(i)}\right).
$$
(b) There is a family of chain maps $F_k$ on the $y$-ified link homology $\HY(\beta)$ satisfying
$$
[F_k,F_m]=[F_k,x_i]=0,\ [F_k,y_i]=kx_i^{k-1}
$$
Both actions are invariant up to homotopy under braid relations.
\end{theorem}
\begin{proof}
Let $C$ be a complex of Soergel bimodules  which admits a $w$-twisted module structure over $\CA$.
Its $y$-fication $\CY(C)$  is a free module over $\C[y_1,\ldots,y_n]$, hence one can define the operators $\frac{\partial}{\partial y_i}$ on $\CY(C)$. We define
\begin{equation}
\label{eq: def Fk}
F_k:=\sum_{i=1}^{n}h_{k-1}\left(x_i,x'_{w^{-1}(i)}\right)\frac{\partial}{\partial y_i}+u_k.
\end{equation}
By Theorem \ref{thm:braid action}  the operators $F_k$ are  invariant up to homotopy under braid relations.
Let us check that they are chain maps. Indeed,
$$
[d_y,F_k]=\left[d+\sum_{i=1}^{n} \xi_i y_i, \sum_{i=1}^{n}h_{k-1}\left(x_i,x'_{w^{-1}(i)}\right)\frac{\partial}{\partial y_i}+u_k\right]=$$
$$
[d,u_k]-\sum_{i=1}^{n}h_{k-1}\left(x_i,x'_{w^{-1}(i)}\right)\xi_i=0.
$$
\end{proof}

\begin{example}
By Remark \ref{rem: u1} we have $u_1=0$, hence $F_1=\sum_{i=1}^{n}\frac{\partial}{\partial y_i}$.
\end{example}

\subsection{From $\CA$ modules to $y$-ifications}\label{sec:from A to yifications}
The $y$-ification can be recognized as a special instance of Theorem \ref{thm:koszul}. We have

\begin{corollary}
    Let $B_\xi$ be the subalgebra of $\CA$ generated by $x_i, x_i', \xi_i$. Then $\CY$ defines a fully faithful embedding of $B_\xi/B$ into the homotopy category of curved complexes of modules over $B_y=B[y_1,\ldots,y_n]$ with curvature $\sum_{i=1}^{n}(x_i-x'_i)y_i$.
\end{corollary}

\begin{remark}
Theorem \ref{thm:koszul} requires power series in $y_1,\ldots,y_n$ but here we choose to work with polynomials in $y_i$ instead.  For $y$-ifications of bounded complexes, the two approaches are identical.

Indeed, if $C$ is a bounded complex then any  power series $\sum f_{k_1,\ldots,k_n}y_1^{k_1}\cdots y_n^{k_n}$ defines an endomorphism of $\CY(C)=C[y_1,\ldots,y_n]$, where $f_{k_1,\ldots,k_n} \in \End(C)$ are any elements with
$$
\deg_h\left(f_{k_1,\ldots,k_n}\right) =  \deg_h(f_0)+2(k_1+\ldots+k_n).
$$
Since $C$ is bounded, all but finitely many of the $f_{k_1,\ldots,k_n}$ must vanish for
degree reasons.
\end{remark}

So the $y$-ification essentially corresponds to forgetting the generators $u_k$. Theorem \ref{thm:koszul} tells us how to capture the complete information.

\begin{definition}
    For an $\CA$-module $X$ the extended $y$-ification is defined by
    \[
    X_{y,\nu}=\left(X[y_1,\ldots,y_n,\nu_1,\ldots,\nu_n],d_{y,v}\right)
    \]
     with the variables of degrees $\deg_h y_i=-2$, $\deg_h \nu_k=-3$. The differential is given by
    \[
    d_{y,v} = d_y+\sum_{k=1}^n \nu_k F_k = d + \sum_{i=1}^n y_i \xi_i + \sum_{k=1}^n \nu_k F_k.
    \]
\end{definition}

We have
\begin{corollary}
    The extended $y$-ification $X\to X_{y,v}$ is a fully faithful embedding of the category $\CA/B$ into the category of curved complexes over $B_{y,\nu}$, where
    \[
    B_{y,\nu} = B\left[y_1,\ldots,y_n,\nu_1,\ldots,\nu_n \,|\, d \nu_k = 0,\; d y_i = \sum_{k=1}^n h_{k-1}(x_i,x_i') \nu_k\right].
    \]
    The curvature on $B_{y,\nu}$-modules is $\sum_{i=1}^{n}(x_i-x'_i)y_i$.
\end{corollary}

Collecting the linear terms in $\nu_k$ we obtain
\begin{corollary}\label{cor:extra structures}
    Let $X$, $Y$ be $\CA$-modules. Any morphism of $y$-ifications $\CY(X)\to\CY(Y)$ coming from a morphism in $\CA/B$ commutes with the operators $F_k$ up to homotopy. In particular, the action of the operators $F_k$ up to homotopy is an invariant of the isomorphism class of an object in $\CA/B$.
\end{corollary}

Analogously, we can define extended $y$-ifications of objects of $\CCA_{c,n}/R_c$ and $\CCA_{c,\infty}/R_c$. Together with Theorem \ref{thm:link invariant} and Proposition \ref{prop:two versions of the category} we obtain
\begin{corollary}
    \label{cor: full Markov}
    The $y$-ified link homology as a module over $\C[(x_i)_{i=1}^c, (y_i)_{i=1}^c, F_1,F_2,\ldots]$ is an invariant of a link.
\end{corollary}

\subsection{$y$-ification of the basic objects}\label{sec:yified building blocks}
From the description in Proposition \ref{prop:basic object} it is easy to compute the $y$-ifications. Indeed, by \eqref{eq:remaining variables}
we have
$$
    \xi_{m_i} = -\sum_{j=m_{i-1}+1}^{m_i-1} \xi_j, 
$$
so 
$$
d_y=\sum_{i=1}^s\sum_{j=m_{i-1}+1}^{m_i}\xi_jy_j=\sum_{i=1}^s\sum_{j=m_{i-1}+1}^{m_i-1}\xi_j(y_j-y_{m_i})
$$
Therefore the $y$-ification of $K_{\bar n, \bar g}$ has cohomology
\[
H(K_{\bar n, \bar g}\otimes \C[y_1,\ldots,y_n],d_y)\simeq \C[(x_j)_{j\in I},\; (y_j)_{j\in I}, \; (\eta_{i j})_{j\leq 2 g_i,i\leq s}]\left(2 g\right).
\]
Since $K_{\bar n, \bar g}$ is free over $R$, taking
Hochschild cohomology and then homology with respect to $d_y$ is the same as
taking homology with respect to $d_y$ and tensoring with an exterior algebra:
\[
\HY(K_{\bar n, \bar g})\simeq \C\left[(x_j)_{j\in I},\; (y_j)_{j\in I}, \; (\eta_{i j})_{j\leq 2 g_i,i\leq s}, \; (\theta_k)_{1\le k\le n}\right]\left(2 g\right).
\]
The action of the operators $F_k$ is given by
\[
F_k = \sum_{i=1}^s \left(k x_{m_i}^{k-1} \frac\partial{\partial y_{m_i}} + k(k-1) x_{m_i}^{k-2} \sum_{j=1}^{g_i} \eta_{i, 2j-1} \eta_{i, 2 j}\right).
\]
Here the notations follow Proposition \ref{prop:basic object} and \eqref{eq: hhh unlink}. 

\section{Hard Lefshetz and symmetry}
\label{sec:Lefshetz}

\subsection{Lefshetz operators}

In this section we work with bigraded complexes and grading-preserving chain maps between them. Let $\left(A=\bigoplus A_{j,k},D\right)$ be a doubly graded complex with differential $D:A_{j,k}\to A_{j,k-1}$. We will assume that for nonzero components  $A_{j,k}$ the values of $j$ are even.

We call a chain map $F:A\to A$ a {\em Lefshetz map} if the following statements hold:
\begin{itemize}
	\item[(a)] $F$ sends $A_{j,k}$ to $A_{j+4,k+2}$.
	\item[(b)] For all $j\ge 0$ the map in homology $F^{j}:H_{-2j,k}(A)\to H_{2j,k+2j}(A)$ is an isomorphism
\end{itemize}

\begin{lemma}
	\label{lem:5 lemma}
	Assume that $0\to A\to B\to C\to 0$ is a short exact sequence of triply graded complexes where the maps preserve all three gradings. Furthermore, assume that $A, B, C$ are equipped with endomorphisms $F_A,F_B,F_C$ which commute with maps between them up to homotopy. Then if $F_A,F_C$ are Lefshetz then $F_{B}$ is Lefshetz too.
\end{lemma}

\begin{proof}
	We have the following commutative diagrams of long exact sequences:
	\begin{center}
		\begin{tikzcd}
			H_{-2j,k+1}(C) \arrow{r} \arrow{d}{F_C^{j}} &H_{-2j,k}(A) \arrow{r} \arrow{d}{F_A^{j}}& H_{-2j,k}(B) \arrow{r} \arrow{d}{F_B^{j}}& H_{-2j,k}(C) \arrow{r} \arrow{d}{F_C^{j}}&  H_{-2j,k-1}(A)    \arrow{d}{F_A^{j}}  \\
			H_{2j,k+1+2j}(C) \arrow{r} &H_{2j,k+2j}(A) \arrow{r} & H_{2j,k+2j}(B) \arrow{r} & H_{2j,k+2j}(C) \arrow{r} &  H_{2j,k-1+2j}(A)
		\end{tikzcd}
	\end{center}
	Since $F_A$, $F_C$ are Lefshetz, by 5-lemma $F_B$ is Lefshetz as well.
\end{proof}

\begin{lemma}
	\label{lem:Jacobson Morozov}
	Assume that $F$ is a Lefshetz operator for a complex $A$, and for all $N$ the sum $\bigoplus_{j-2k=N}H_{j,k}(A)$ is finite-dimensional. Then the action of $F$ extends to an action of $\mathfrak{sl}_2$ on homology of $A$ where $H$ acts on $H_{2j,k}(A)$ by $j$.
\end{lemma}

\begin{proof}
	Clearly, $[F,H]=2F$, so it is sufficient to construct an operator $E$. By the Jacobson-Morozov theorem, we can construct it as follows.

	Since the operator $F$ preserves the sum $j-2k$, by the assumption of the lemma it is locally nilpotent.  Let us prove that for $j\ge 0$ and $s\ge 0$ one has
	$$
	H_{2j,k}(A)=\Ker(F^s)\oplus \Imm(F^{j+s}).
	$$
	Assume that $v\in \Ker(F^s)\cap \Imm(F^{j+s})$ and $v\neq 0$, then $v=F^{j+s}u$ for some $u\in H_{-2j-4s,k-2j-2s}(A)$, and $F^{j+2s}(u)=0$.
	Contradiction, so $\Ker(F^s)\cap \Imm(F^{j+s})=0$.

	On the other hand, for arbitrary $v\in H_{2j,k}(A)$ we can write $F^{s}v=F^{j+2s}(u)$ for some $u\in H_{-2j-4s,k-2j-2s}(A)$,
	so $v-F^{j+s}(u)\in \Ker(F^s)$, so $\Ker(F^s)+\Imm(F^{j+s})=H_{2j,k}(A)$.

	Now on $H_{2j,k}(A)$ we have an ascending filtration by $\Ker(F^s)$ and a complementary descending filtration by $\Imm(F^{j+s})$, so we can split these filtrations by
	$$
	H_{2j,k}(A)=\bigoplus_s H_{2j,k}^{(s)}(A),\  H_{2j,k}^{(s)}(A)=\Ker(F^{s+1})\cap \Imm(F^{j+s}).
	$$
	Then $E$ is defined on $H_{2j,k}^{(s)}$ by taking preimage under $F$ multiplied by a constant and projecting to $H_{2j-4,k-2}^{(s+1)}$.
\end{proof}

\begin{lemma}
	\label{lem:Lefshetz product}
	Let $A,B$ be two complexes of vector spaces with Lefshetz maps $F_{A}$ and $F_{B}$ satisfying the assumptions of Lemma \ref{lem:Jacobson Morozov}. Then  $F_{A}\otimes 1+1\otimes F_{B}$ is a Lefshetz map for $A\otimes B$.
\end{lemma}

\begin{proof}
	We have $(A\otimes B)_{j,k}=\bigoplus_{j',k'}A_{j',k'}B_{j-j',k-k'}$. Since $F_{A}:A_{j',k'}\to A_{j'+4,k'+2}$ and
	$F_{B}:B_{j-j',k-k'}\to A_{j-j'+4,k-k'+2}$, both $F_{A}\otimes 1$ and $1\otimes F_{B}$ send $(A\otimes B)_{j,k}$ to $(A\otimes B)_{j+4,k+2}$.

	Furthermore, by the K\"unneth formula we have $H_{*,*}(A\otimes B)=\bigoplus H_{*,*}(A)\otimes H_{*,*}(B)$. By Lemma \ref{lem:Jacobson Morozov} the actions of $F_A$ and $F_B$ extend to the actions of $\mathfrak{sl}_2$, so $F_{A}\otimes 1+1\otimes F_{B}$ defines the action of $F$ on the product of $\mathfrak{sl}_2$ representations. Since this product is a finite dimensional representation of $\mathfrak{sl}_2$, it is symmetric and $F$ is a Lefshetz map.
\end{proof}

\begin{lemma}
	\label{lem:lef poly}
	Let $A=\C[x_1,\ldots,x_n,y_1,\ldots,y_n]$ with zero differential, where $x_i$ has bidegree $(2,0)$ and $y_i$ has bidegree $(-2,-2)$. Then the operator $F=\sum x_i\frac{\partial}{\partial y_i}$ is Lefshetz.
\end{lemma}

\begin{proof}
	We can write $A=\C[x_1,y_1]\otimes \cdots \otimes \C[x_n,y_n]$. It is easy to see that $x_i\frac{\partial}{\partial y_i}$ is Lefshetz on $\C[x_i,y_i]$, and the statement follows from Lemma \ref{lem:Lefshetz product}.
\end{proof}

\begin{lemma}
	\label{lem:lef wedge}
	Let $A=\C[\eta_1,\ldots,\eta_{2n}](2n)$, where $\eta_i$ has bidegree $(2,1)$, and $F=\eta_1\eta_2+\eta_3\eta_4+\ldots+\eta_{2n-1}\eta_{2n}$.
	Then $F$ is Lefshetz.
\end{lemma}

\begin{proof}
	$$A=\C[\eta_1,\ldots,\eta_{2n}](2n)=\C[\eta_1,\eta_2](2)\otimes \C[\eta_3,\eta_4](2)\otimes \cdots \otimes \C[\eta_{2n-1},\eta_{2n}](2).$$
	It is easy to see that $\eta_i\eta_j$ is Lefshetz for $\C[\eta_i,\eta_j](2)$, so by Lemma \ref{lem:Lefshetz product} $F$ is Lefshetz for $A$.
\end{proof}

Lemmas \ref{lem:Lefshetz product}, \ref{lem:lef poly} and \ref{lem:lef wedge} together with the explicit description in Section \ref{sec:yified building blocks} immediately implies
\begin{corollary}
    The $y$-ified basic objects $\CY(K_{\bar n, \bar g})$ are Lefschetz.
\end{corollary}

Proposition \ref{prop:reduction 1} together with the classification of the basic objects implies that $\CY(\HH^k(L))$ is homotopy equivalent to a twisted complex built from copies of $\CY(K_{\bar n, \bar g})$ and homological shifts thereof.  By Lemma \ref{lem:5 lemma} we obtain our main theorem:
\begin{theorem}\label{thm:full lefschetz}
    For any link $L$ the $y$-ified homology $H(\CY(\HH^k(L)))$ is Lefschetz.
\end{theorem}

\begin{example}
By the main result of \cite{GH}, the $y$-ified homology of $(n,n)$ torus link has the form:
$$
\HY(T(n,n))=\bigcap_{i\neq j}(x_i-x_j,y_i-y_j,\theta_i-\theta_j)\subset \C[x_1,\ldots,x_n,y_1,\ldots,y_n,\theta_1,\ldots,\theta_n].
$$
The symmetry exchanges $x_i$ and $y_i$, clearly leaving the ideals $$(x_i-x_j,y_i-y_j,\theta_i-\theta_j)$$ and their intersection unchanged. Furthermore, the operator $F=\sum x_i\frac{\partial}{\partial y_i}$ satisfies 
$$
F\left(f(x_i-x_j)+g(y_i-y_j)+h(\theta_i-\theta_j)\right)=F(f)(x_i-x_j)+F(g)(y_i-y_j)+
$$
$$
g(x_i-x_j)+F(h)(\theta_i-\theta_j)
$$
and preserves these ideals as well. The operator $E$ acts by $\sum y_i \frac{\partial}{\partial x_i}$, so $\HY(T(n,n))$ has a natural action of $\mathfrak{sl}_2$.
\end{example}

\begin{appendix}
\section{Higher coproducts on $\CA$}
\label{sec: higher}

\subsection{Higher coproducts}

As we discussed in Section \ref{sec:coproduct}, the coproduct $\Delta:\CA\to \CA\otimes_{R}\CA$ is coassociative up to homotopy. In this section we write this homotopy and its higher analogues explicitly.
Define $\delta^{(3)}:\CA\to \CA\otimes_{R}\CA\otimes_{R} \CA$ by the equation $\delta^{(3)}(x_i)=\delta^{(3)}(x'_i)=\delta^{(3)}(\xi_i)=0$ and 
$$
\delta^{(3)}(u_k)=\sum_{i} h_{k-3}\left(x_i,x'_i,x''_i,x'''_i\right)\xi_i\otimes \xi_i\otimes \xi_i,\ k\ge 3.
$$
We extend $\delta^{(3)}$ to all monomials inductively by the equations:
\begin{multline}
\label{eq:coproduct extend}
\delta^{(3)}(ax_i)=\delta^{(3)}(a)x_i,\ \delta^{(3)}(ax'_i)=\delta^{(3)}(a)x''_i,\\
\delta^{(3)}(a\xi_i)=\delta^{(3)}(a)\left(\xi_i\otimes 1\otimes 1+1\otimes \xi_i\otimes 1+1\otimes 1\otimes \xi_i\right),\\
\delta^{(3)}(au_k)=\delta^{(3)}(a)\Delta_1\Delta(u_k)+\Delta_2\Delta(a)\delta^{(3)}(u_k).
\end{multline}
To make the last equation in \eqref{eq:coproduct extend} well defined, we may assume that $u_k$ is lexicographically maximal in the monomial $au_k$.

\begin{lemma}
\label{lem: assoc}
We have $\Delta_1\Delta-\Delta_2\Delta=[d,\delta^{(3)}]$
where $\Delta_i$ acts on $i$-th tensor factor.
\end{lemma}

\begin{proof}
Note that $\delta^{(3)}(du_k)=0$. We have
$$
\Delta_1\Delta(u_k)=u_k\otimes 1\otimes 1+1\otimes u_k\otimes 1+1\otimes 1\otimes u_k+
$$
$$
\sum_i h_{k-2}\left(x_i,x'_i,x''_i\right)\xi_i\otimes \xi_i\otimes 1+\sum_i h_{k-2}\left(x_i,x''_i,x'''_i\right) \left(\xi_i\otimes 1\otimes \xi_i+1\otimes \xi_i\otimes \xi_i\right),
$$
so
\begin{multline*}
\Delta_1\Delta(u_k)-\Delta_2\Delta(u_k)=\sum_i\left(h_{k-2}\left(x_i,x'_i,x''_i\right)-h_{k-2}\left(x_i,x'_i,x'''_i\right)\right)\xi_i\otimes \xi_i\otimes 1+\\
\sum_i\left(h_{k-2}\left(x_i,x''_i,x'''_i\right)-h_{k-2}\left(x_i,x'_i,x'''_i\right)\right)\xi_i\otimes 1\otimes \xi_i+\\
\sum_i\left(h_{k-2}\left(x_i,x''_i,x'''_i\right)-h_{k-2}\left(x'_i,x''_i,x'''_i\right)\right)1\otimes \xi_i\otimes \xi_i=\\
\sum_i h_{k-3}\left(x_i,x'_i,x''_i,x'''_i\right)\left[(x''_i-x'''_i)\xi_i\otimes \xi_i\otimes 1+(x''_i-x'_i)\xi_i\otimes 1\otimes \xi_i+(x_i-x'_i)1\otimes \xi_i\otimes \xi_i\right]=\\
\sum_i h_{k-3}\left(x_i,x'_i,x''_i,x'''_i\right) d(\xi_i\otimes \xi_i\otimes \xi_i)=d(\delta^{(3)}(u_k)).
\end{multline*}

Next, we check that the homotopy extends correctly to all monomials  \eqref{eq:coproduct extend}. Note that $\delta^{(3)}(du_k)=0$,
$$
\Delta_1\Delta(x_i)=\Delta_2\Delta(x_i)=x_i,\ \Delta_1\Delta(x'_i)=\Delta_2\Delta(x'_i)=x''_i,
$$
and
$$
\Delta_1\Delta(\xi_i)=\Delta_2\Delta(\xi_i)=\xi_i\otimes 1\otimes 1+1\otimes \xi_i\otimes 1+1\otimes 1\otimes \xi_i.
$$
In other words, all equations in \eqref{eq:coproduct extend} have the form
$$
\delta^{(3)}(ab)=\delta^{(3)}(a)\Delta_1\Delta(b)+\Delta_2\Delta(a)\delta^{(3)}(b),
$$
where we assume that $a$ and $b$ satisfy the statement of the lemma.
Therefore
$$
[d,\delta^{(3)}](ab)=d\delta^{(3)}(a)\Delta_1\Delta(b)+\delta^{(3)}(a)\Delta_1\Delta(db)+\Delta_2\Delta(da)\delta^{(3)}(b)+\Delta_2\Delta(a)\cdot d\delta^{(3)}(b)-$$
$$
\delta^{(3)}(da\cdot b)-\delta^{(3)}(a\cdot db)=
[d,\delta^{(3)}(a)]\Delta_1\Delta(b)+\Delta_2\Delta(a)[d,\delta^{(3)}](b)=
$$
$$
\left(\Delta_1\Delta(a)-\Delta_2\Delta(a)\right)\Delta_1\Delta(b)+\Delta_2\Delta(a)\left(\Delta_1\Delta(b)-\Delta_2\Delta(b)\right)=
$$
$$
\Delta_1\Delta(ab)-\Delta_2\Delta(ab).
$$
\end{proof}

Furthermore, we can define 
$$
\delta^{(s)}(u_k)=\begin{cases}
\sum_i h_{k-s}\left(x_i,x'_i,\ldots,x^{(s)}_i\right)\xi_i^{\otimes s} & \text{if}\ s\le k\\
0 & \text{otherwise}.
\end{cases}
$$

\begin{lemma}
We have the following $A_{\infty}$ relations for $\delta^{(s)}$:
\begin{equation}
\label{eq: A infty}
\sum_{s+t=m+1} \sum_{a=1}^{t}\pm \delta^{(s)}_{a}\delta^{(t)}(u_k) = d(\delta^{(m)}(u_k))
\end{equation}
Equivalently, if we define a differential $\delta=d+\sum_{s,a} \delta^{(s)}_a$ acting on $\bigoplus_{m=0}^{\infty} \CA^{\otimes m}$ then $\delta^2=0$.
\end{lemma}

\begin{proof}
The case $m=3$ follows from Lemma \ref{lem: assoc}, consider the case $m>3$.
By definition, $\delta^{(s)}(\xi_i)=0$ for $s>2$, so $\delta^{(s)}_{a}\delta^{(t)}=0$ unless $s=2$ or $t=2$.
Now it is easy to see that each term $\xi_i^{a}\otimes 1\otimes \xi_i^{m-1-a}$ appears twice in \eqref{eq: A infty},
with coefficients complete symmetric functions $h_{k-m+1}$ evaluated at $m$ arguments with $x_i^{(a)}$ and $x_i^{(a+1)}$ missing.
The difference of these functions is precisely $\left(x_i^{(a)}-x_i^{(a+1)}\right)h_{k-m}\left(x_i,\ldots,x_i^{(m)}\right)$, up to a sign, so overall sum agrees with $ d\left(\delta^{(m)}\right)$.
\end{proof}

\begin{remark}
All higher coproducts $\delta^{(s)}$ vanish on $u_2$, so the the coproduct on the subalgebra of $\CA$ generated by $u_2$ and $\xi_i$ is coassociative on the nose.
\end{remark}

\begin{remark}
Equations \eqref{eq:coproduct extend} can be interpreted as a part of a bigger structure expressing the homotopy $\delta^{(3)}$ between two homomorphisms of $A_{\infty}$ algebras, see e.g. \cite[Section 3.7]{Keller}. It would be very interesting to use similar ideas to extend all $\delta^{(s)}$
to a family of maps $\delta^{(s)}:\CA\to \CA^{\otimes s},  2\le s$ satisfying \eqref{eq: A infty}. It appears that this would endow $\CA$ with the stricture of  $A_{\infty}$ bialgebra in the sense of \cite{LOTdiag} and \cite{SU,U}.
\end{remark}


\subsection{Integral formulas}

The above constructions appear to be specific to power sums, but they can be generalized to other symmetric functions.

\begin{lemma}
Let $Q$ be a symmetric function. Define
\begin{equation}
\label{eq: canonical factorization}
a_i^{Q}(x,x')=\int_{0}^{1}\frac{\partial Q}{\partial x_i}\left(tx_i+(1-t)x'_i\right)dt
\end{equation}
then
$$
Q(x)-Q(x')=\sum a_i^Q(x,x')(x_i-x'_i).
$$
\end{lemma}

\begin{proof}
We have
$$
Q(x)-Q(x')=Q\left(tx_i+(1-t)x'_i\right)|_{0}^{1}=\int_{0}^{1}\left[Q\left(tx_i+(1-t)x'_i\right)\right]'_{t} dt=
$$
$$
\sum_{i=1}^{i}\int_{0}^{1}(x_i-x'_i)\frac{\partial Q}{\partial x_i}\left(tx_i+(1-t)x'_i\right)dt.
$$
\end{proof}

\begin{example}
If $Q=p_k$ is the power sum then $\frac{\partial Q}{\partial x_i}=kx_i^{k-1}$, and
$$
a_i^{Q}(x,x')=\int_{0}^{1}k\left(tx_i+(1-t)x'_i\right)^{k-1}dt=\int_{0}^{1}\frac{dt}{x_i-x'_i}\frac{d}{dt}\left(tx_i+(1-t)x'_i\right)^k=
$$
$$
\frac{x_i^{k}-(x'_i)^k}{x_i-x'_i}=h_k\left(x_i,x'_i\right).
$$
\end{example}

As a consequence, for any symmetric function $Q$ the element $\sum a_i^Q(x,x')\xi_i$ is a cycle in the algebra $\CA$. More generally, we call
a collection of functions $a_i(x,x')$ a factorization of $Q$ if $\sum a_i(x,x')(x_i-x'_i)=Q(x)-Q(x')$. Let $I$ be the ideal generated  by $x_i-x'_i$.

\begin{lemma}
\label{lem: comparing factorizations}
Suppose that $\{a_i(x,x')\}$ is a factorization of $Q$. Then the following holds:

(a) There exist $C_{ij}$ such that $a_i(x,x')=a_i^{Q}(x,x')+\sum_{j\neq i}C_{ij}(x,x')(x_j-x_j')$

(b) If $a_i(x,x')-a_i^{Q}(x,x')\in I^p$ then $C_{ij}\in I^{p-1}$.
\end{lemma}

\begin{proof}
Let us change variables $t_i=x_i-x'_i$ and consider the Koszul complex with differential $d(\xi_i)=x_i-x'_i=t_i$. We have
$$
d\left(\sum\left (a_i(x,x')-a^{Q}_i(x,x')\right)\xi_i\right)=\sum \left(a_i(x,x')-a^{Q}_i(x,x')\right)\left(x_i-x'_i\right)=0,
$$
so $\sum (a_i(x,x')-a^{Q}_i(x,x'))\xi_i$ is a cycle. Since $t_i$ form a regular sequence, there exist $C_{ij}(x,x')$ such that
$$
\sum \left(a_i(x,x')-a^{Q}_i(x,x')\right)\xi_i=d\left(\sum_{i<j} C_{ij}(x,x')\xi_i\xi_j\right)=
$$
$$
\sum_{i}\left(\sum_{j<i}C_{ji}t_j-\sum_{j>i}C_{ij}t_j\right)\xi_i.
$$
This implies (a) after changing the signs of $C_{ij}$ appropriately. For (b) it is sufficient to note that the differential increases the $t$-degree by 1,
so each homogeneous summand in $a_i(x,x')-a^{Q}_i(x,x')$ of $t$-degree $p$ corresponds to $C_{ij}$ of degree $p-1$.
\end{proof}

Since $\CA$ is a resolution of $R$,  the cycle $\sum a_i^Q(x,x')\xi_i$ is a boundary, and the following result gives an explicit construction of bounding element.

\begin{lemma}\label{lem:corrections}
Consider the derivation $\sum u_k\frac{\partial}{\partial p_k}$ which sends a symmetric function $Q$ to an element
\[
U(Q)=\sum_{k=1}^n u_k\frac{\partial Q}{\partial p_k}\in \CA.
\]
Let $I\in B$ be the ideal generated by the differences $x_i-x_i'$.
For each symmetric function $Q$ there exist elements $C_{ij}^Q\in I$ such that
 we have
\[
d \left(U(Q) + \sum_{i,j} C_{ij}^Q \xi_i \xi_j \right) = \sum_i a_i^Q(x,x')\xi_i.
\]
\end{lemma}
\begin{proof}
The derivation $Q\mapsto U(Q)$ is characterized by $U(p_k)=u_k$ and
\[
U(fg)=U(f) g + f U(g).
\]
Suppose the corrections $C_{ij}^f,C_{ij}^g\in I$ have been constructed for symmetric functions $f$, $g$.
Set
\[
C_{ij}^{fg} = f C_{ij}^g + g C_{ij}^f + C_{ij}^{f,g},
\]
where
\[
C_{ij}^{f,g} = \frac12 \left( \int_0^1 dt \int_0^t ds - \int_0^1 ds \int_0^s dt \right) \frac{\partial f}{\partial x_i} \left(s x + (1-s) x'\right) \frac{\partial g}{\partial x_j} \left(t x + (1-t) x'\right).
\]
Setting $x=x'$ makes the integrand independent of $s$, $t$. Hence the difference of the integrals vanishes and we have $C_{ij}^{f,g}\in I$. Hence $C_{ij}^{fg}\in I$. We have
\[
\sum_i C_{ij}^{f,g} d\xi_i = \frac12 \left( \int_0^1 dt \int_0^t ds - \int_0^1 ds \int_0^s dt \right) \frac{\partial f}{\partial s} \left(s x + (1-s) x'\right) \frac{\partial g}{\partial x_j} \left(t x + (1-t) x'\right)
\]
\[
=\frac12 \left( \int_0^1 dt \left(f\left(t x + (1-t) x'\right) - f(x')\right) - \int_0^1 dt \left(f(x) - f\left(t x + (1-t) x'\right)\right)\right) \frac{\partial g}{\partial x_j} \left(t x + (1-t) x'\right).
\]
\[
= \int_0^1 dt \left(f\left(t x + (1-t) x'\right) - f(x)\right)\frac{\partial g}{\partial x_j} \left(t x + (1-t) x'\right).
\]
Similarly, we have
\[
\sum_j C_{ij}^{f,g} d\xi_j = - \int_0^1 dt \left(g\left(t x + (1-t) x'\right) - g(x)\right)\frac{\partial f}{\partial x_i} \left(t x + (1-t) x'\right).
\]
Therefore we obtain
\[
\sum_i (C_{ij}^{f,g} - C_{ji}^{f,g}) d\xi_i = a_j^{fg} - f a_j^g - g a_j^f.
\]
Summing over all $j$ leads to
\[
d\left(\sum_{i,j} C_{ij}^{f,g} \xi_i \xi_j\right) = \sum_j\left(a_j^{fg} - f a_j^g - g a_j^f\right) \xi_j,
\]
which is precisely what is required to show that the correction $C_{ij}^{fg}$ satisfies the condition for the function $fg$. Now any symmetric function can be expressed as a polynomial of $p_1,\ldots,p_n$, for which the statement is clearly true, so the statement is true in general.
\end{proof}

Note that the formulas for the coproduct and its higher analogues for $u_Q$ can be written compactly using multivariable integrals:
$$
\delta^{(s)}(u_Q)=\int_{\Delta_s}Q(t_0x_i+t_1x'_i+\ldots+t_sx^{(s)}_i+\xi_i(dt_0-dt_1)+\xi'_i(dt_1-dt_2)+\ldots+\xi^{(s)}_i(dt_{s-1}-dt_{s})).
$$
For example, one can check that for $Q=p_k$ we get
$$
h_{k-s}\left(x_i,x'_i,\ldots,x^{(s)}_i\right)=\frac{(k-s)!}{k!}\int_{\Delta_{s}}\left(t_0x_i+t_1x'_i+\ldots+t_sx^{(s)}_i\right)^{k-s}d\Vol
$$
where $\Delta_{s}=\{(t_0,\ldots,t_s):t_i\ge 0, \sum t_i=1\}$ is the standard $s$-dimensional simplex.
This follows from the multinomial theorem and the multivariate beta integral
$$
\int_{\Delta_s}t_0^{a_0}\cdots t_s^{a_s}d\Vol=\frac{a_0!\cdots a_s!}{(a_0+\ldots+a_s+s)!}.
$$

\section{Group cohomology}
\label{sec: groups}

The constructions in this paper are motivated by the work of the third author on ``curious hard Lefshetz" property in cohomology of character varieties \cite{Mchar}. The Lefshetz operator in \cite{Mchar} corresponds to a tautological 2-form on the character variety.
In this appendix, we review the construction of tautological forms on groups and other varieties following  Bott, Shulman and Jeffrey \cite{Bott, BSS, Jeffrey}.

\subsection{Transgressions}

Let $G$ be a Lie group with Lie algebra $\fg$.  Given an invariant function $Q$ on $\fg$, we define a family of differential forms
$\Phi_n(Q)$ on $G^n$.

First we describe simplicial model for the universal bundle $EG\to BG$.
Consider the family of spaces $\eNG(n)=G^{n+1}$ with boundary maps $\varepsilon_j:\eNG(n)\to \eNG(n-1)$,
$\varepsilon_j(U_0,\ldots,U_n)=(U_0,\ldots,\widehat{U_j},\ldots,U_n)$. We also consider the spaces $\bNG(n)=G^n$ and the boundary maps $\varepsilon_j:\bNG(n)\to \bNG(n-1)$:
$$
\varepsilon_0:(g_1,\ldots,g_n)\to (g_2,\ldots,g_n),\
$$
$$
\varepsilon_i(g_1,\ldots,g_n)=(g_1,\ldots,g_ig_{i+1},\ldots,g_n),\ (1\le i\le n-1),\ \varepsilon_n(g_1,\ldots,g_n)=(g_2,\ldots,g_n).
$$
There are maps $q:\eNG(n)\to \bNG(n)$ and $\sigma: \bNG(n)\to \eNG(N)$ defined as follows:
$$
q(U_0,\ldots,U_n)=(U_0^{-1}U_1,\ldots, U_{n-1}^{-1}U_n),\ \sigma(g_1,\ldots,g_n)=(1,g_1,g_1g_2,\ldots,g_1\cdots g_n).
$$
It is easy to see that $q\circ \sigma=\Id_{\bNG(n)}$ and both $q$ and $\sigma$ commute with the action of $\varepsilon_j$ for all $j$. Note that the choice of $q$ and $\sigma$ is reverse from the one in \cite{Jeffrey}.

We get the following commutative diagram:
$$
\begin{tikzcd}
\eNG: & \ldots \arrow{r} & G\times G\times G \arrow{r}{\varepsilon_j} \arrow[bend left]{d}{q} & G\times G \arrow{r}{\varepsilon_j}  \arrow[bend left]{d}{q} & G \arrow[bend left]{d}{q} \\
\bNG: & \ldots \arrow{r}   & G\times G\arrow{r}{\varepsilon_j} \arrow[bend left]{u}{\sigma} &  G \arrow{r} \arrow[bend left]{u}{\sigma} & * \arrow[bend left]{u}{\sigma}\\
\end{tikzcd}
$$
We define the complex $(\bigoplus_n \Omega^{\bullet}(NG(n)),\delta)$ with the differential
$$
\delta:=\Omega^{\bullet}(NG(n))\to \Omega^{\bullet}(NG(n+1)),\ \delta=\sum (-1)^{i}\varepsilon^*_i.
$$

Let $\theta=g^{-1}dg$ be the left-invariant $\fg$-valued one-form on $G$. Consider the $n$-symplex $\Delta_n=\{(t_0,\ldots,t_n): 0\le t_i, \sum t_i=1\}$, we define the one-form $\theta(t)$ on  $\Delta_n\times G^{n+1}$
by the equation
$$
\theta(t)=\sum_{i=0}^{n}t_i\theta_i
$$
Consider the curvature
$$
F=d\theta(t)+[\theta(t),\theta(t)].
$$
Given a symmetric function $Q\in S(\fg^*)^{G}$ of degree $r$, we define
$$
\overline{\Phi_n}(Q)=\int_{\Delta}Q(F) \in \Omega^{2r-n}(G^{n+1}).
$$
Finally, define $\Phi_n(Q)=\sigma^*\overline{\Phi_n}(Q)$ where
$$
\sigma:G^n\to G^{n+1}, \sigma(g_1,\ldots,g_n)=(g_1\cdots g_n,g_2\cdots g_n,\ldots, g_n,1).
$$

It is easy to see that $\Phi_n(Q)=0$ for $n>r$, so there are finitely many forms for a given $Q$. The forms $\Phi_n(Q)$ are not closed, but it is known that
$$
(d\pm \delta)(\Phi_1(Q)+\ldots+\Phi_r(Q))=0,
$$
so $d\Phi_n(Q)=\pm \delta \Phi_{n-1}(Q)$.

\begin{example}
For $G=GL(n)$ we have $\theta=g^{-1}dg$, and for $Q=\Tr(g^2)$ we get
$$
\Phi_1(Q)=\Tr(\theta,[\theta,\theta])\in \Omega^3(G), \Phi_2(Q)=\Tr(f^{-1}df\wedge dg g^{-1})=(f|g)\in \Omega^2(G\times G)
$$
Then we have three equations
$$
d\Phi_1(Q)=0,\ d(\Phi_2(Q))=\delta(\Phi_1(Q)),\ \delta(\Phi_2(Q))=0.
$$
The latter equation can be written as $(f|g)+(fg|h)=(f|gh)+(g|h)$.
\end{example}

\begin{example}
We have
$$
(UV^{-1}|VW^{-1})=(UV^{-1}|V|W^{-1})-(V|W^{-1})=(U|V^{-1}|V|W^{-1})-(U|V^{-1})-(V|W^{-1})=
$$
$$
(U|1|W^{-1})-(U|V^{-1})-(V|W^{-1})=(U|W^{-1})-(U|V^{-1})-(V|W^{-1})
$$
since $(V|V^{-1})=0$.
\end{example}

\begin{example}
Suppose that $G=(\C^*)^n$ is the abelian group of diagonal matrices. Then $\theta=u^{-1}du$ and $d\theta=0$,
so
$$
\theta(t)=\sum t_i\theta_i,\ F=\sum \theta_i dt_i.
$$
Pick $Q=Tr(g^r)$, then
$$
\overline{\Phi_k}(Q)=\int_{\Delta_n}\Tr(\sum \theta_i dt_i)^k=\begin{cases}
\Tr(\theta_1\wedge \cdots \wedge \theta_r) & \text{if}\ k=r\\
0 & \text{otherwise}.\\
\end{cases}
$$
This is nothing but the degree 0 part of $\delta^{(r)}(u_r)$, where we identify $\theta = \diag(\xi_1,\ldots,\xi_n)$, and
$$
\Tr(\theta_1\wedge \cdots \wedge \theta_r)=\sum \xi_i^{\otimes r}.
$$
\end{example}

\subsection{Maps to the group}

We can use the classes $\Phi_n(Q)$ in the following construction. We say that a map $f:X\to G$ is $Q$-exact if $f^*\Phi_1(Q)=d\omega$ for some $(2r-2)$-form $\omega$. Given two $Q$-exact maps $f:X\to G$ and $g:Y\to G$ as the product
$fg:X\times Y\to G\times G\xrightarrow{m} G$ and the form
$$
\omega=\omega_1+\omega_2-(f\times g)^*\Phi_2(Q).
$$
Recall the equation
$$
d\Phi_2(Q)=\Phi_1(Q)\otimes 1-m^*\Phi_1(Q)+1\otimes \Phi_1(Q),
$$
then
$$
d\omega=d\omega_1+d\omega_2-(f\times g)^*\Phi_2(Q)=(f\times g)^*\left(\Phi_1(Q)\otimes 1+1\otimes \Phi_1(Q)-d\Phi_2(Q)\right)=
$$
$$
(f\times g)^*m^*\Phi_1(Q).
$$
Unfortunately, this operation is not associative: on $(X\times Y)\times Z$ we get the form
$$
\omega_1+\omega_2+\omega_2-(f\times g\times h)^*(\Phi_2\otimes 1+m_{12}^*\Phi_2)
$$
while  on $X\times (Y\times Z)$ we get the form
$$
\omega_1+\omega_2+\omega_2-(f\times g\times h)^*(1\otimes \Phi_2+m_{23}^*\Phi_2)
$$
Nevertheless, the equation
$$
d\Phi_3(Q)=1\otimes \Phi_2-m_{12}^{*}\Phi_2+m_{23}^*\Phi_2-\Phi_2\otimes 1
$$
means that the two choices of form on $X\times Y\times Z$ are different by $(f\times g\times h)^*d\Phi_3(Q)$.

\subsection{Equivariant transgression}

The above constructions extend to $H$-equivariant cohomology of $G$, see \cite{Jeffrey} for details. In particular, one can define forms
$$
\overline{\Phi^H_n}(Q)=\int_{\Delta}Q(F+\mu(\theta(t))) \in \Omega^{2r-n}_{H}(G^{n+1}).
$$
where $\mu$ is the moment map for the action of $H$ on $G$.

\begin{example}
Let $H=G=(\C^*)^n$, then as above $\theta(t)=\sum \theta_i t_i$, $F=\sum \theta_i dt_i$ and $\mu(\theta(t))=\sum t_i\mu_i$.
Then
$$
\Phi_1^H(Q)=\int_{0}^{1}Q(\theta_0dt-\theta_1 dt+t\mu_0+(1-t)\mu_1)
$$
We can expand
$$
Q(\theta_0dt-\theta_1 dt+t\mu_0+(1-t)\mu_1)=Q(t\mu_0+(1-t)\mu_1)+\sum_{i} (\theta_0-\theta_1)_i\frac{\partial Q}{\partial x_i}(t\mu_0+(1-t)\mu_1),
$$
so
$$
\Phi_1^H(Q)=\sum_{i}(\theta_0-\theta_1)_{i}\int_{0}^{1}\frac{\partial Q}{\partial x_i}(t\mu_0+(1-t)\mu_1)dt
$$
which is precisely the formula \eqref{eq: canonical factorization} when we write $\xi_i=(\theta_0-\theta_1)_i$, $\mu_0=(x_1,\ldots,x_n)$ and $\mu_1=(x'_1,\ldots,x'_n)$.
\end{example}

\end{appendix}

\bibliographystyle{amsalpha}
\bibliography{refs}

\end{document}